\documentclass[11pt]{article}
\usepackage{amsmath}
\usepackage{slashed}
\usepackage{amssymb}
\usepackage{amsfonts}
\usepackage{amsthm}
\usepackage{amsxtra}
\usepackage{array}
\usepackage{mathrsfs}
\usepackage{tikz-cd}
\usepackage{stmaryrd}
\usepackage[all,cmtip]{xy}

\usepackage[english]{babel}
\usepackage[letterpaper,top=2cm,bottom=2cm,left=3cm,right=3cm,marginparwidth=1.75cm]{geometry}

\usepackage{graphicx}
\usepackage[colorlinks=true, allcolors=blue]{hyperref}
\graphicspath{ {images/} }

\newcommand{\GL}{\mathrm{GL}}

\newcommand{\SO}{\mathrm{SO}}
\newcommand{\Spin}{\mathrm{Spin}}
\newcommand{\U}{\mathrm{U}}
\newcommand{\SU}{\mathrm{SU}}
\newcommand{\Sp}{\mathrm{Sp}}

\newcommand{\Ker}{\mathrm{Ker}}

\newcommand{\Hom}{\mathrm{Hom}}
\newcommand{\End}{\mathrm{End}}

\newcommand{\Ad}{\mathrm{Ad}}

\newcommand{\Sym}{\mathrm{Sym}}

\newcommand{\Z}{\mathbb{Z}}

\newcommand{\R}{\mathbb{R}}
\newcommand{\C}{\mathbb{C}}

\newcommand{\Sph}{\mathbb{S}}

\newcommand{\CP}{\mathbb{CP}}

\newlength{\wdth}

\newcommand{\mres}{\mathbin{\vrule height 1.6ex depth 0pt width
0.13ex\vrule height 0.13ex depth 0pt width 1.3ex}}

\theoremstyle{plain}
\newtheorem{theorem}{Theorem}[section]

\theoremstyle{definition}
\newtheorem{definition}[theorem]{Definition}
\newtheorem{example}[theorem]{Example}

\theoremstyle{plain}

\newtheorem{corollary}[theorem]{Corollary}

\newtheorem{lemma}[theorem]{Lemma}

\newtheorem{proposition}[theorem]{Proposition}

\theoremstyle{remark}
\newtheorem*{acknowledgments}{Acknowledgments}

\newtheorem{remark}[theorem]{Remark}

\title{Cohomogeneity-One Lagrangian Mean Curvature Flow}
\author{Jesse Madnick, Albert Wood}

\begin{document}
\maketitle

\begin{abstract}
We study mean curvature flow of Lagrangians in $\C^n$ that are cohomogeneity-one with respect to a compact Lie group $G \leq \SU(n)$ acting linearly on $\C^n$.  Each such Lagrangian necessarily lies in a level set $\mu^{-1}(\xi)$ of the standard moment map $\mu \colon \C^n \to \mathfrak{g}^*$, and mean curvature flow preserves this containment. \\
\indent We classify all cohomogeneity-one self-similarly shrinking, expanding and translating solutions to the flow, as well as cohomogeneity-one smooth special Lagrangians lying in $\mu^{-1}(0)$. Restricting to the case of almost-calibrated flows in the zero level set $\mu^{-1}(0)$, we classify finite-time singularities, explicitly describing the Type I and Type II blowup models.  Finally, given any cohomogeneity-one special Lagrangian in $\mu^{-1}(0)$, we show it occurs as the Type II blowup model of a Lagrangian MCF singularity.

Throughout, we give explicit examples of suitable group actions, including a complete list in the case of $G$ simple. This yields infinitely many new examples of shrinking and expanding solitons for Lagrangian MCF, as well as infinitely many new singularity models.
\end{abstract}

\tableofcontents

\section{Introduction}

\indent \indent The discovery that Lagrangian submanifolds of Calabi-Yau manifolds are preserved by mean curvature flow \cite{Smoczyk1996}, a phenomenon referred to as \emph{Lagrangian mean curvature flow}, has inspired ambitious conjectures in geometry and theoretical physics. Most notably, the Thomas-Yau conjecture (proposed by Thomas and Yau \cite{Thomas2002a} and refined by Joyce \cite{Joyce2015}) states that, under a stability assumption, an almost-calibrated Lagrangian submanifold deformed by mean curvature flow should converge to the unique special Lagrangian in its Hamiltonian isotopy class. Recently, Lagrangian MCF has been utilised to find special Lagrangian fibrations in log Calabi-Yau manifolds \cite{Collins2019}, confirming a particular case of the SYZ conjecture in Mirror Symmetry. Lagrangian MCF is also of interest in geometric analysis as a particular case of high-codimension mean curvature flow, a challenging subject which is presently far less well understood than the hypersurface case.
 
Lagrangian mean curvature flow typically forms finite-time singularities. Therefore, to resolve the Thomas-Yau conjecture, a surgery procedure for continuing the flow past a singularity must be developed. Surgery for mean curvature flow has been defined in several special cases, e.g., for two-convex hypersurfaces by Huisken-Sinestrari \cite{Huisken2009} and for quadratically pinched manifolds in high-codimension by Nguyen \cite{Nguyen2020}.  In those works, defining the surgery procedure hinges on a complete understanding of the nature of finite-time singularities. To that end, the precise geometry of singularities may be analysed using Type I and Type II blowup procedures, which are limits of rescaled flows at the singular space-time point (see \S\ref{sec-2.3} for definitions). Type I blowups are self-similarly shrinking solutions, and in all known cases, Type II blowups are static or translating soliton solutions. To complete the surgery procedure, one must glue in suitable model manifolds. Self-similarly expanding soliton solutions are ideal candidates for this gluing, see for example \cite{Begley2017}. Therefore, to define a suitable surgery procedure for Lagrangian mean curvature flow, we must classify the possible Type I and II blowups of finite-time singularities, and classify the soliton solutions.

In this work, we study mean curvature flow of Lagrangian submanifolds $L\subset \mathbb{C}^n$ that are invariant under the Hamiltonian action of a compact subgroup $G\leq \SU(n)$. 
The crucial advantage of working with this sub-class is that a $G$-invariant Lagrangian $L$ must lie in a single level set of the \textit{moment map} of the action $\mu:\mathbb{C}^n \to \mathfrak{g}^*$ at a central value $\xi \in \mathfrak{g}^*$, i.e.\ $L \subset \mu^{-1}(\xi)$ (see \S\ref{sec-3} for details). Moreover, this containment is preserved under mean curvature flow. This containment effectively reduces the codimension of $L$, mitigating the key difficulty of working with high-codimension submanifolds. 
	
To maximise this advantage, we study \textit{cohomogeneity-one} Lagrangians. By this, we mean that there exists a conjugacy class $(H)$ of subgroups of $G$ such that for all $z\in L$, the orbit $\mathcal{O}_z \cong G/{G_z}$ has dimension $n-1$, and the isotropy subgroup $G_z$ belongs to $(H)$. It turns out that a cohomogeneity-one Lagrangian is a \textit{hypersurface} within the $(n+1)$-dimensional coisotropic smooth manifold $M_\xi := \mu^{-1}(\xi)\cap \mathbb{C}^n_{(H)}$, where $\mathbb{C}^n_{(H)}$ is the subset of points with isotropy subgroup conjugate to $H$. For example, $\SO(n)$-invariant Lagrangians $L \subset \mathbb{C}^n \setminus 0$ are necessarily cohomogeneity-one, with isotropy type $\SO(n-1)$. Taking the quotient of $M_\xi$ by the $G$-action, we obtain a bijection between $G$-invariant Lagrangians in $M_\xi$ and curves in the 2-dimensional \emph{K\"{a}hler quotient} $Q := M_\xi / G$ (Proposition \ref{prop-bijection1}). In short, cohomogeneity-one Lagrangian mean curvature flow in $\mathbb{C}^n$ corresponds to a modified curve shortening flow in $Q$.
	
We will focus primarily on the case $\xi = 0$, i.e.\ we assume our Lagrangian submanifolds satisfy $L \subset \mu^{-1}(0)$.  This condition occurs naturally in several important settings:
	\begin{itemize}
    \item An almost-calibrated, cohomogeneity-one Lagrangian is exact if and only if it lies in $\mu^{-1}(0)$ (Proposition \ref{prop-exact}). 
    \item All cohomogeneity-one self-similarly shrinking solitons, expanding solitons, and special Lagrangian cones lie in $\mu^{-1}(0)$.
    \item If $G$ is compact and semisimple (e.g.\ $G = \SO(n)$\,), then $0 \in \mathfrak{g}^*$ is the only central value, so \textit{all} $G$-invariant Lagrangian submanifolds lie in $\mu^{-1}(0)$.
\end{itemize}
	
A key feature of $\mu^{-1}(0)$ is its invariance under the multiplicative action of $\mathbb{C}^*$, i.e.\ if $z \in \mu^{-1}(0)$, then the complex line $P_z := \mathbb{C}\cdot z$ is contained in $\mu^{-1}(0)$.  Given a cohomogeneity-one Lagrangian $L \subset \mu^{-1}(0)$, the intersection $l := L \cap P_z$ is a smooth curve, which we refer to as the \textit{profile curve of $L$}.  This yields an alternative bijection, one between cohomogeneity-one Lagrangians in $\mu^{-1}(0)$ and smooth curves in $P_z$ (Proposition \ref{prop-profilebijection}).

\begin{figure}
    \centering
    \includegraphics[scale=0.38]{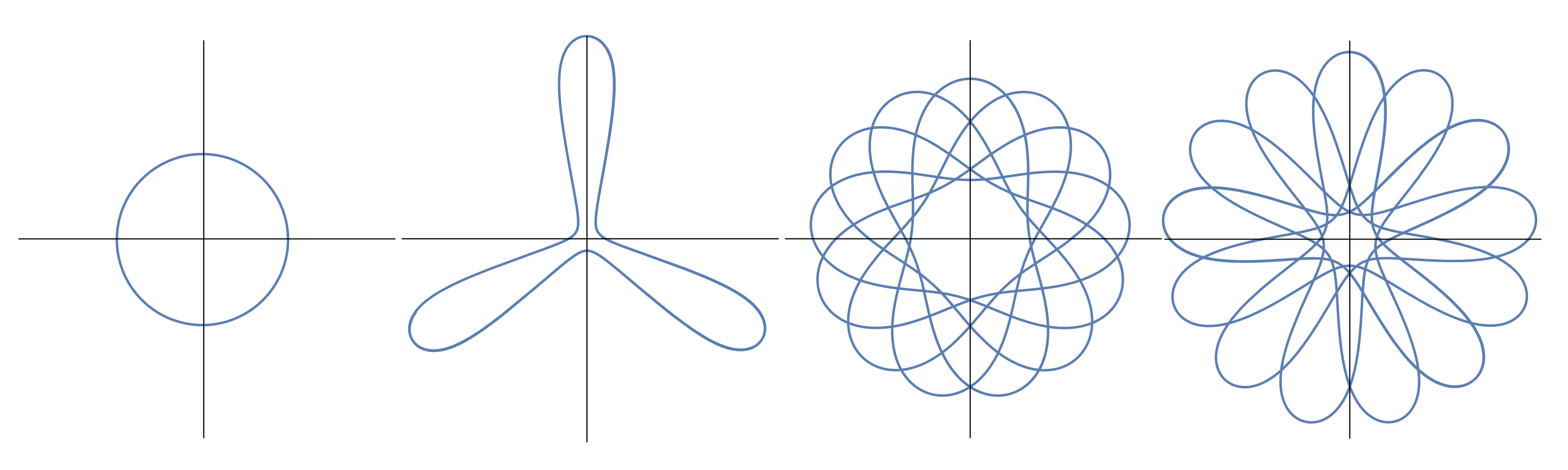}
    \caption{Four examples of self-similarly shrinking solutions to (\ref{eq-introflowequation}) in the case $n=2$, corresponding to $\SO(2)$-invariant shrinking solitons to LMCF in $\mathbb{C}^2$. The first corresponds to the Clifford torus in $\mathbb{C}^2$, and the others are examples of Anciaux with $(p,q)$ equal to $(1,3)$, $(6,13)$ and $(5,13)$ respectively. The first two comprise the only embedded examples of self-similarly shrinking solutions when $n=2$.}
    \label{fig-introshrinkers}
\end{figure}

A surprising observation is that the mean curvature of a cohomogeneity-one Lagrangian $L \subset \mu^{-1}(0)$ may be expressed solely in terms of the profile curve $l \subset P_z$, independently of the group $G$ and isotropy type $(H)$ (Proposition \ref{prop-curvature}). This provides a correspondence between cohomogeneity-one Lagrangian mean curvature flows, and solutions to the following flow of immersed curves in $P_z$:
	\begin{align}
		\frac{\partial f_t}{\partial t}^\perp & = \vec{k} - (n-1)\frac{f_t^\perp}{|f_t|^2}.\label{eq-introflowequation}
	\end{align}
Equation (\ref{eq-introflowequation}) has been well-studied in the context of the $\SO(n)$-action on $\mathbb{C}^n$.  There are two distinct static solutions, corresponding to the $\SO(n)$-invariant \textit{Lawlor neck} (first documented by Harvey and Lawson \cite{Harvey1982}) and the flat special Lagrangian plane. Furthermore, Anciaux, Castro, and Romon \cite{Anciaux2006,AnciauxCastroRomon2006,AnciauxRomon2009} classify all connected shrinking and expanding solutions to (\ref{eq-introflowequation}). We denote these $l^{(p,q)}$ and $l^\alpha$ respectively, where $(p,q)$ ranges over coprime pairs of integers satisfying $\frac{p}{q} \in (\frac{1}{2n}, \frac{1}{\sqrt{2n}})$  ($p$ is the winding number and $q$ the number of maxima of curvature) and $\alpha \in (0, \frac{\pi}{n})$ represents the angle between the asymptotes. This yields a classification of $\SO(n)$-invariant shrinking/expanding solitons of Lagrangian mean curvature flow. Prior work on $\SO(n)$-equivariant Lagrangian MCF also includes singularity analysis by Neves \cite{Neves2007}, Savas-Halilaj and Smoczyk \cite{Savas-Halilaj2018}, Viana \cite{Viana2021}, the second author \cite{Wood2019} and Evans \cite{Evans2022}, as well as long-time existence and convergence results by Evans, Lambert and the second author \cite{Wood2019a} and Su \cite{Su2019}; for a survey, see \cite{Lotay2020b}.  Self-similar solutions to Lagrangian mean curvature flow have also been studied by Lee-Wang \cite{Lee2007, Lee2008}, Joyce-Lee-Tsui \cite{Joyce2010}, Castro-Lerma \cite{Castro2010} and Su \cite{Su2019a, Su2020}.

Symmetry methods have been utilised in the study of Lagrangian mean curvature flow and special Lagrangians in other ways. For example, Lagrangian MCF of orbits (i.e. the  cohomogeneity-zero case) was investigated by Pacini \cite{Pacini2002}, and a construction of $G$-invariant special Lagrangians due to Joyce \cite{Joyce2001a} has been generalised to Lagrangian MCF by Konno \cite{Konno2017} and Ochiai \cite{Ochiai2021}. 
	 
\subsection*{Overview of Results}

\indent \indent In Section $\ref{sec-5}$, we classify self-similarly shrinking, expanding, and translating solutions to Lagrangian mean curvature flow (henceforth LMCF), as well as special Lagrangians, in the setting of cohomogeneity-one Lagrangians in $\mu^{-1}(0)$. In particular, this generalises the work of Anciaux, Castro and Romon.
	 \begin{theorem}[Solitons for Cohomogeneity-one LMCF]\label{thm-introsolitons}
	 	Let $G \leq \SU(n)$ be a compact connected Lie group. Define the immersed curves:
	 	\begin{align*}
	 	    \widetilde c_{k,\overline\theta}: (0,\infty)\rightarrow \mathbb{C}, \quad \widetilde c_{k,\overline\theta}(r) &:= re^{i\left(\tfrac{\overline\theta}{n} + \tfrac{k\pi}{n}\right)},\\
	 	    \widetilde l_{B,k,\overline{\theta}}: \left[-\tfrac{\pi}{2n},\tfrac{\pi}{2n}\right]\rightarrow \mathbb{C},\quad
    \widetilde l_{B,k,\overline{\theta}}(\alpha) &:= \frac{B}{\sqrt[n]{\cos\left( n\alpha\right)}}e^{i\left(\alpha+\tfrac{\overline\theta}{n} - \tfrac{\pi}{2n} + \tfrac{k\pi}{n}\right)},
	 	\end{align*}
	 	and let $l^{(p,q)}$, $l^\alpha$ be the self similarly shrinking/expanding solutions to (\ref{eq-introflowequation}) of Anciaux \cite{Anciaux2006}. Let $L \subset \mathbb{C}^n$ be a connected immersed cohomogeneity-one $G$-invariant Lagrangian submanifold, $z \in L$, and $P_z := \mathbb{C} \cdot z \subset \mathbb{C}^n$. Then:
	 	\begin{itemize}
	 		\item $L$ is a special Lagrangian cone if and only if it is the $G$-orbit of $\widetilde c_{k, \overline\theta} \subset P_z$ for unique $k \in \mathbb{Z}$, $\overline\theta \in \mathbb{R}$.
	 		\item $L$ is a complete special Lagrangian contained in $\mu^{-1}(0)$ if and only if it is the $G$-orbit of $\widetilde l_{B,k, \overline\theta} \subset P_z$ for some $B > 0$, $k \in \mathbb{Z}$, $\overline\theta \in \mathbb{R}$.
	 		\item $L$ is a complete shrinking soliton for LMCF if and only if it is the $G$-orbit of $e^{i\phi}\cdot l^{(p,q)}\subset P_z$ for some coprime $p,q \in \mathbb{Z}$ satisfying $\frac{p}{q} \in \left(\frac{1}{2n}, \frac{1}{\sqrt{2n}}\right)$ and $e^{i\phi} \in S^1$.
	 		\item $L$ is a complete expanding soliton for LMCF if and only if it is the $G$-orbit of $e^{i\phi}\cdot l^{\alpha}\subset P_z$ for some $\alpha \in (0,\frac{\pi}{n})$, $e^{i\phi} \in S^1$.
	 		\item $L$ is a complete translating soliton for LMCF if and only if $n=1$ and $L \subset \mathbb{C}$ is the grim reaper curve - the unique non-static translating solution to curve shortening flow in $\mathbb{C}$.
	 	\end{itemize}
	 \end{theorem}
We remark that Theorem \ref{thm-introsolitons} implicitly identifies $P_z$ and $\mathbb{C}$. When this identification is made carefully, the value of $\overline\theta$ in the first two results is the Lagrangian angle of the special Lagrangian submanifolds. See Theorems \ref{thm-splagcones}, \ref{thm-splags}, \ref{thm-shrinkers} and \ref{thm-translators} for the precise statements.

Theorem \ref{thm-introsolitons} completely classifies cohomogeneity-one shrinking, expanding and translating solitons, and cohomogeneity-one special Lagrangian cones. However, there exist cohomogeneity-one special Lagrangians in $\mathbb{C}^n$ that do not lie in the zero level set of the corresponding moment map, and are therefore not included in Theorem \ref{thm-introsolitons}. For example, there exists a foliation of $\mathbb{C}^3$ by $T^2$-invariant special Lagrangians, discovered by Harvey and Lawson \cite{Harvey1982}.
	
In Section \ref{sec-6}, we turn our attention to singularity analysis, and with a view to applications to the Thomas-Yau conjecture, we focus on flows of almost-calibrated Lagrangians. In \cite{Wood2019}, the author works with $\SO(n)$-invariant Lagrangians, demonstrating that every singularity must occur at the origin, and explicitly describing the Type I and Type II blowup models. We establish analogous results in the general cohomogeneity-one case. 


\begin{figure}[b]
    \centering
    \includegraphics[scale=0.5]{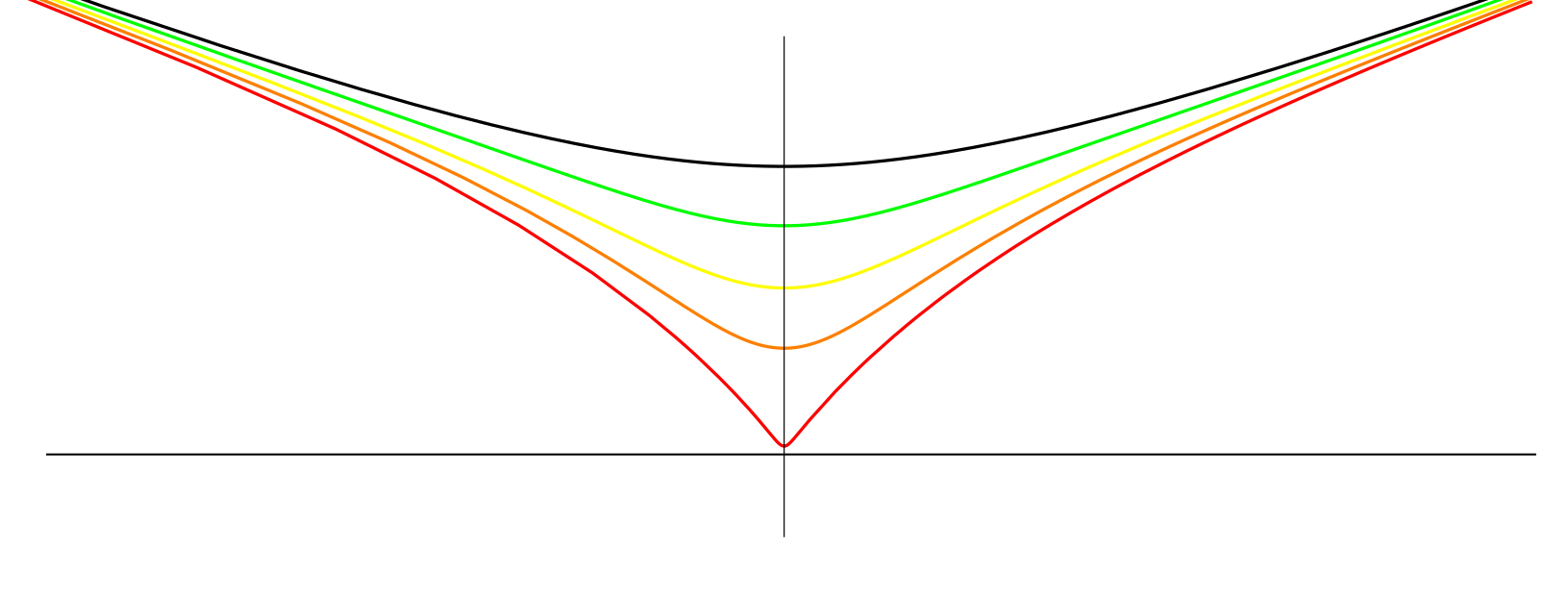}
    \caption{A solution to (\ref{eq-introflowequation}) forming a finite-time singularity at the origin, in the case $n=3$. The uppermost curve is the initial condition.}
    \label{fig-introsingularity}
\end{figure}

\begin{theorem}[Singularities of Cohomogeneity-one LMCF]\label{thm-introsingularities}
	Let $G \leq \SU(n)$ be a compact connected Lie group. Let $L_t \subset \mu^{-1}(0)$ be a connected cohomogeneity-one $G$-invariant almost-calibrated LMCF for $t \in (0,T)$, with a finite-time singularity at a space-time point $(X,T)$. 
	
	Then $X$ must be the origin. Moreover, there exist $\overline\theta \in \mathbb{R}$, $B>0$, $k \in\mathbb{Z}$ and a complex line $P_z \subset \mathbb{C}^n$ such that:
		\begin{itemize}
			\item Every Type I blowup at time $T$ is the special Lagrangian cone $L^\infty$ given by the $G$-orbit of the curve $\widetilde c_{k-1,\overline\theta}\cup \widetilde c_{k, \overline\theta} \subset P_z$.
			\item Every Type II blowup at time $T$ is the special Lagrangian $\widetilde{L}^\infty$ given by the $G$-orbit of the profile curve $\widetilde l_{B,k,\overline\theta} \subset P_z$. The asymptotic cone of $\widetilde L^\infty$ is the Type I blowup $L^\infty$.
		\end{itemize}
\end{theorem}
The reader may wonder whether there exist almost-calibrated Lagrangian mean curvature flows with a finite-time singularity, to which Theorem \ref{thm-introsingularities} may be applied. We answer this in the affirmative, and therefore prove the following existence statement for singularities modelled on cohomogeneity-one special Lagrangians.

\begin{theorem}[Existence of Singularities with Prescribed Models]\label{thm-introexistence}
Let $L^\infty \subset \mathbb{C}^n$ be a complete connected cohomogeneity-one special Lagrangian such that $L^\infty \subset \mu^{-1}(0)$.

Then $L^\infty$ is asymptotically conical, and there exists an almost-calibrated Lagrangian mean curvature flow $L_t$ forming a Type II singularity at the origin, such that:
\begin{itemize}
    \item Any Type I blowup is the asymptotic cone of $L^\infty$,
    \item Any Type II blowup is $L^\infty$.
\end{itemize}
\end{theorem}

\noindent These are proven as Theorems \ref{thm-locationofsingularities}, \ref{thm-typeiblowup}, \ref{thm-typeiiblowup} and \ref{thm-existenceofsingularities}. 

Explicit examples of Type II blowups are rare in the literature, especially in the Lagrangian case. Theorem \ref{thm-introexistence} provides infinitely many previously unobserved singularity models -- one for each group action admitting cohomogeneity-one Lagrangians. In particular, taking $G = T^{n-1} \leq \SU(n)$ as in the work of Harvey-Lawson, we find a singularity with Type I blowup equal to a \textit{pair} of $T^{n-1}$-invariant cones. This contrasts with the result of Lambert-Lotay-Schulze \cite[Thm. 1.2]{Lambert2021} that there is no singularity of almost-calibrated Lagrangian mean curvature flow in a Calabi-Yau 3-fold such that the blowdown of the Type II blowup is given by a \textit{single} Harvey-Lawson $T^{2}$-cone.
	
It should be noted that Type I and Type II blowups are typically non-unique, and so Theorem \ref{thm-introsingularities} includes a uniqueness statement for blowups of cohomogeneity-one flows. It also provides further evidence for the conjecture that the blowdown of a Type II blowup of a singular mean curvature flow should be equal to a Type I blowup. 
	
The assumptions of Theorem \ref{thm-introsingularities} are necessary. Examples of Neves \cite{Neves2007} exhibit several distinct singular behaviours for $\SO(2)$-equivariant LMCF in $\mathbb{C}^2$, showing that the almost-calibrated condition is required for uniqueness of blowup models. Indeed, our work requires the almost-calibrated condition to rule out double-density planes in the Type I blowup (see Lemma \ref{lem-keylemma}). If one considers cohomogeneity-one LMCF in a level set $\mu^{-1}(\xi)$ for $\xi \neq 0$, then the quotient $l := L/G \subset Q$ must be used in place of the profile curve. We expect that as in the $\xi = 0$ case, singularities will occur only as a result of the orbits collapsing, i.e.\  the singular space-time point $(X,T)$ will have a different isotropy type than the flow. Finally, in order to consider flows with non-constant isotropy type, one would need to work within the full level set $\mu^{-1}(\xi)$ in place of $M_\xi$; these level sets are not in general smooth manifolds. 

Finally in Section \ref{sec-7}, we consider the question of which compact connected Lie groups $G \leq \SU(n)$ admit $(n-1)$-dimensional isotropic orbits in $\mathbb{C}^n$, and so yield $G$-invariant cohomogeneity-one Lagrangian submanifolds.  Restricting to the case of $G$ simple, a result of Bedulli-Gori \cite{Bedulli2008} quickly implies a classification of admissible group actions: 
	
\begin{theorem}[Classification for $G$ Simple]
		Let $G \leq \SU(n)$ be a compact simple Lie group. Then there exists a cohomogeneity-one $G$-invariant Lagrangian submanifold in $\mathbb{C}^n$ if and only if $G$ appears in Bedulli-Gori's table, given in Figure \ref{fig-bedulligori}.
\end{theorem}
	
In particular, for each group action in this table, by Theorem \ref{thm-introsolitons} there exists a cohomogeneity-one $G$-invariant special Lagrangian, and by Theorem \ref{thm-introexistence} there exists a Lagrangian mean curvature flow with a finite-time singularity modelled on this special Lagrangian.\\

\noindent \textbf{Conventions:} We set the following conventions.
\begin{enumerate}
    \item A connected immersed submanifold $L$ in a manifold $M$ is a subset that is the image of an immersion with connected domain.
    \item We say $L$ is cohomogeneity-one if $L$ is $G$-invariant of isotropy type $(H)$ for a Lie group $G$ and $H \leq G$ satisfying $\text{dim}(G/H) = n-1$.  This definition is slightly more restrictive than usual, in that we disallow $L$ from containing exceptional or singular orbits.
    \item We say that a mean curvature flow $L_t$ has a singularity at time $T$ if there is some singular spacetime point $(X,T)$ (see Section \ref{sec-2.3} for the definition). In particular, we do not consider singularities at infinity.
\end{enumerate}

\begin{acknowledgments} We thank Chung-Jun Tsai, Wei-Bo Su and Jason Lotay for their invaluable support and conversation, and Ben Lambert for sharing with us his proof of curvature estimates for LMCF, included in Section \ref{sec-6.1}. This work was completed during the authors' postdoctoral fellowships at the National Center for Theoretical Sciences and National Taiwan University. We thank these institutions for their support.
\end{acknowledgments}

\section{Preliminaries}

\subsection{Hermitian Linear Algebra}

\indent \indent Let $(V, \langle \cdot, \cdot \rangle, J, \omega)$ be a Hermitian vector space of real dimension $2n$, where $\langle \cdot, \cdot \rangle$ is a positive-definite inner product, $J \in \End(V)$ is a complex structure, and $\omega \in \Lambda^2(V^*)$ is the non-degenerate $2$-form given by $\omega(X,Y) = \langle JX, Y\rangle$.  A subspace $W \subset V$ is called \textit{isotropic} if $\omega|_W = 0$.  Note that
$$W \text{ is isotropic } \iff  JW \text{ is isotropic } \iff JW \subset W^\perp.$$
If $W \subset V$ is isotropic, then $\dim_\R(W) \leq n$. 
 In the other direction, a subspace $F \subset V$ is called \textit{coisotropic} if $JF \supset F^\perp$.  Note that $F$ is coisotropic if and only if $(JF)^\perp$ is isotropic.  In particular, if $F \subset V$ is coisotropic, then $\dim_\R(F) \geq n$. \\
\indent A subspace $L \subset V$ is \textit{Lagrangian} if $L$ is both isotropic and coisotropic.  Thus, $L$ is Lagrangian if and only if $\omega|_L = 0$ and $\dim_\R(L) = n$, or equivalently, if $JL = L^\perp$.  For future use, we record the following easy linear algebra fact.

\begin{lemma} \label{lem:IsotropicDecomp} Let $(V, \langle \cdot, \cdot \rangle, J, \omega)$ be a Hermitian vector space, and let $W \subset V$ be an isotropic subspace.  Then there exists a complex subspace $E \subset V$ and an orthogonal decomposition
$$V = E \oplus W \oplus JW.$$
\end{lemma}


\indent We now let $(V, \langle \cdot, \cdot \rangle, J, \omega, \Omega)$ be a special Hermitian vector space of real dimension $2n$, meaning that $(\langle \cdot, \cdot\rangle, J, \omega)$ is a Hermitian structure as above, and $\Omega \in \Lambda^{n,0}(V^*)$ is a \textit{complex volume form} --- i.e.\ , an $(n,0)$-form satisfying
\begin{equation}
    \textstyle \frac{1}{n!}\omega^n = (-1)^{n(n-1)/2}\left(\frac{i}{2}\right)^n\,\Omega \wedge \overline{\Omega}.\label{eq-complexvolumeform}
\end{equation}
In particular, since $\frac{1}{n!}\omega^n = \text{vol}_V$ is the volume form of $\langle \cdot, \cdot \rangle$, the $(n,0)$-form $\Omega$ is non-zero.  Finally, note that for each Lagrangian subspace $L \subset V$, there exists $\theta \in \mathbb{R}$, unique up to adding an integer multiple of $\pi$, for which
\begin{equation}
    \Omega|_L = e^{i\theta}\,\text{vol}_L \label{eq-lagangle}
\end{equation}
where $\text{vol}_L$ is a volume form for $L$.  The angle $\theta \in \mathbb{R}$ is called a \textit{Lagrangian angle} (or \textit{phase}) of $L$.

\subsection{Lagrangian Submanifolds of K\"ahler and Calabi-Yau Manifolds}

\indent \indent Let $(\overline{M}^{2n}, \langle \cdot, \cdot \rangle, J, \overline{\omega})$ be a K\"{a}hler manifold, with Levi-Civita connection $\overline \nabla$. 
An immersed submanifold $F:L \to \overline{M}$ is \textit{Lagrangian} (resp., \textit{isotropic}, \textit{coisotropic}) if each of the subspaces $F_*(T_x L) \leq T_{F(x)}\overline{M}$ is Lagrangian (resp., isotropic, coisotropic).

If $F: L \to \overline{M}$ is Lagrangian, then $J: TL \rightarrow T^\perp L$ is a bundle isometry by the compatibility of $J$ and $\langle \cdot, \cdot \rangle$.  Thus, the second fundamental form $A \in \Gamma(\Sym^2(T^*L) \otimes T^\perp L)$ of $L$ may be realised as a fully symmetric $(0,3)$-tensor on $TL$:
\begin{equation}
    h \in \Gamma(\Sym^3(T^*L)), \quad h(X,Y,Z) := \langle \overline\nabla_{X} Y, JZ \rangle. \label{def-h}
\end{equation} 
The mean curvature $\vec{H} \in \Gamma(TL)$ may be similarly represented by a 1-form, $\alpha \in \Omega^1(L)$, which is obtained by taking a trace of $h$. In local coordinates, denoting the components of $h$ by $h_{ijk}$ and the components of the inverse of the metric by $g^{ij},$
\begin{equation}
    \alpha \in \Gamma(T^*L), \quad \alpha_i = g^{ij}h_{ijk}. \label{def-alpha}
\end{equation}Note that since $J$ is an isometry, the norms of these new tensors are the same as originally:
\[ \big| H\big| = |\alpha|, \quad |A| = |h|.\]

Now let $(\overline{M}^{2n}, \langle \cdot, \cdot \rangle, J, \omega, \Omega)$ be a Calabi-Yau manifold, so that $(\overline{M}^{2n}, \langle \cdot, \cdot \rangle, J, \omega)$ is a K\"ahler manifold and $\Omega$ is a \textit{holomorphic volume form}, i.e.\  a complex volume form satisfying $\nabla \Omega = 0$. If $F \colon L \to \overline M$ is an immersed oriented Lagrangian, and $\text{vol}_L$ a volume form for $L$, then by (\ref{eq-lagangle}) there exists $\phi \colon L \rightarrow S^1$ for which $F^*\Omega = \phi \cdot \text{vol}_L$. If there exists a function $\theta: L \rightarrow \mathbb{R}$ such that $\phi= e^{i\theta}$, then the Lagrangian is said to be \textit{zero-Maslov}. The function $\theta$ is known as a \emph{Lagrangian angle} for $L$, and the pair $(L,\theta)$ is known as a \textit{graded Lagrangian}.

The importance of $\theta$ is that it is a primitive for the mean curvature:
\begin{equation}
    d\theta = \alpha, \quad \quad J\nabla \theta = \vec H. \label{eq-jgradtheta}
\end{equation}
In particular, if $\theta = \overline \theta$ is constant, then $L$ is an immersed minimal submanifold. In fact, $L$ is calibrated by the form $\text{Re}(e^{-i\overline\theta}\Omega)$, and is therefore volume minimising by the theory of calibrations \cite{Harvey1982}.  A graded Lagrangian with constant angle $\overline \theta$ is known as a \textit{special Lagrangian of angle $\overline \theta$}. It is also natural to consider graded Lagrangians satisfying the weaker condition $\theta \in (\overline\theta - \frac{\pi}{2} + \varepsilon,\, \overline\theta + \frac{\pi}{2} - \varepsilon)$ for some $\varepsilon > 0$ and  $\overline \theta \in \mathbb{R}$; these are known as \textit{almost-calibrated Lagrangians}.

In the case of $\overline M^{2n} = \mathbb{C}^n$, there is a natural primitive for $\omega$ known as the \textit{Liouville form}:
\begin{equation}\label{eq-liouvilleform}
\lambda := \frac{1}{2} \sum_{j=1}^n x_j dy_j - y_j dx_j.
\end{equation}
We say $L$ is \textit{exact} if the closed $1$-form $F^*\lambda \in \Omega^1(L)$ is exact.  We say $L$ is \emph{rational} if there is $a \in \mathbb{R}$ with
\begin{equation}\label{eq-rational}
    \lambda\left(H_1(L, \mathbb{Z})\right) = \{2\pi k a\,|\, k \in \mathbb{Z}\}.
\end{equation}
Note that $L$ is exact if and only if $L$ is rational with $a = 0$.

\subsection{Mean Curvature Flow}\label{sec-2.3}

\indent \indent Consider a Riemannian manifold $(M^m,\langle\cdot,\cdot\rangle)$ and a smooth manifold $L^n$. A smooth family of immersions $F_t:L\rightarrow M$ for $t \in [t_0,t_1]$ is a \textit{mean curvature flow} if 
\begin{equation} \frac{dF}{dt}^\perp = \vec{H}.\label{eq-mcf}\end{equation}
We denote the image of the immersion by $L_t:= F_t(L)$, and will often refer to a mean curvature flow $L_t$, suppressing mention of the immersion. 

A submanifold $F:L \rightarrow \mathbb{R}^m$ with $\vec H = 0$ is known as a \emph{minimal submanifold}; such submanifolds provide static solutions to the mean curvature flow equation (\ref{eq-mcf}). Other simple solutions are given by \emph{soliton solutions} --- flows which move by ambient isometries or scaling.
Firstly, if a submanifold $F: L \rightarrow \mathbb{R}^m$ satisfies 
\begin{equation} \vec H + \lambda\frac{F^\perp}{2} = 0 \label{eq-shrinker} \end{equation}
then it follows that $F_t := \sqrt{-\lambda t}F$ is a solution to mean curvature flow. If $\lambda < 0$, then $F$ is known as a \emph{shrinking soliton}, and if $\lambda > 0$, then $F$ is known as an \emph{expanding soliton}.
Secondly, if $F$ satisfies 
\begin{align}\vec{H} - V^\perp = 0, \label{eq-translator}\end{align}
for a constant vector $V \in \mathbb{R}^m$, then it follows that $F_t := F + tV$ is a solution to mean curvature flow, which translates in the direction $V$. Such submanifolds are known as \emph{translating solitons}.

For compact submanifolds, the following theorem (proven by Huisken in the hypersurface case \cite{Huisken1990}) describes the behaviour of the flow at the maximal time of existence:
\begin{theorem}[\cite{Smoczyk2012}]\label{thm-curvatureblowup}
    Let $L$ be a closed manifold, $(M,\langle\cdot,\cdot\rangle)$ a complete Riemannian manifold, and $F: L \times [0,T) \to (M,\langle\cdot,\cdot\rangle)$ a smooth solution to the mean curvature flow. Suppose $T$ is the maximal time of existence. Then
    \[ \lim_{t\rightarrow T}\sup_{L_t} |A|^2 = \infty.\]
    Furthermore, if $(M,\langle\cdot,\cdot\rangle) = (\mathbb{R}^n,\langle\cdot,\cdot\rangle_{\text{Eucl}})$, then there exists a constant $c > 0$ such that
    \[ \max_{L_t}|A|^2 \geq \frac{c}{T-t}\quad\forall t \in [0,T). \]
\end{theorem}
This theorem motivates the following definition. If $F_t:L\to M$ is a mean curvature flow for $t\in [0,T)$, then $X = (x, T) \in M \times [0,T]$ is a \textit{singular space-time point} if there exists a sequence of space-time points $(p_i,t_i) \in L \times [0,T]$ such that
\[(F_{t_i}(p_i),t_i) \rightarrow X, \quad \limsup_{i\rightarrow \infty}|A(p_i,t_i)| = \infty,\]
and we say that $F_t$ has a \textit{singularity at time $T$}. We say the singularity is \textit{Type I} if there exists $C > 0$ such that 
\[ \max_{L_t}|A|^2 \leq \frac{C}{T-t} \quad\forall t \in [0,T), \]
and otherwise we call it \textit{Type II}.

There are two common procedures for analysing the structure of a singularity at a singular space-time point $X$; we describe only the case where $M = \mathbb{R}^m$ since it is most relevant to our setting, though these procedures are possible also for general manifolds $M$. Consider a mean curvature flow $F_t:L \rightarrow M$, with singular space-time point $(x,T)$ and area bounds $\mathcal{H}^n(L_0 \cap B_R(0)) \leq C_0 R^n$, and consider a sequence $\lambda_i \in \mathbb{R}^+$ with $\lambda_i \to \infty$, and the corresponding sequence of \textit{Type I rescalings},
\[ F_s^{\lambda_i} := \lambda_i(F_{\lambda_i^{-2}s + T} - x), \]
which may be seen to be solutions to mean curvature flow. It may be proven that the flows subsequentially converge in the sense of Radon measures to a limiting flow $F^\infty_s$, which is known as a \textit{Type I blowup} of $F_t$ at $T$. In general, this is a weak solution to mean curvature flow given by a family of rectifiable varifolds, known as a Brakke flow. However, if the singularity is Type I, then this convergence may be shown to be smooth by the bound on $|A|$, and the limiting smooth flow of complete submanifolds is a shrinking soliton mean curvature flow. In the Type II case, by choosing a \emph{sequence} of points and scaling around these points so as to normalise the curvature, it is also possible to extract a smooth limit from a blowup sequence. To achieve this, for each integer $k \in \mathbb{N}$, choose $(p_k,t_k)$ such that
\[ |A(p_k,t_k)|^2\left(T - \frac{1}{k} - t_k\right) = \max_{t < T-\frac{1}{k}, \, p \in M}|A(p,t)|\left( T - \frac{1}{k} - t\right), \]
and define $A_k := |A(p_k,t_k)|,\,\, \alpha_k = -A_k^2 t_k$, and $\omega_k:= A_k^2(T- \tfrac{1}{k} -t_k)$. Then the family of \textit{Type II rescalings}
\[ F_s^{(p_k,t_k)} := A_k \left( F_{A_k^{-2}s + t_k} - F_{t_k}(p_k) \right), \quad s \in [\alpha_k,\omega_k) \]
are mean curvature flows, and subsequentially smoothly converge to an eternal mean curvature flow $\widetilde F^\infty_s$ with $\sup_{L_0^\infty}|A| = 1$, known as a \textit{Type II blowup}. Such blowups are typically static solutions or translating solutions to mean curvature flow, although this has not been proven in the general case. Note that neither Type I nor Type II blowups are unique in general, in the sense that they depend on the choice of blowup sequence.

Of key importance to this work is the fact that in a Calabi-Yau manifold $(\overline M, \langle\cdot,\cdot\rangle,J,\overline\omega,\Omega)$, the class of Lagrangian submanifolds is preserved by mean curvature flow \cite{Smoczyk1996}. A flow of Lagrangian submanifolds will be referred to as a \emph{Lagrangian mean curvature flow}, or LMCF for short. We will require the following key facts about LMCF:
\begin{proposition}
    Let $F_t : L^n \to (\overline M^{2n}, \langle\cdot,\cdot\rangle,J,\overline\omega,\Omega)$ be a graded Lagrangian mean curvature flow in a Calabi-Yau manifold, with $\theta_t$ the Lagrangian angle of $L_t$. Then:\\
\indent     (a) The Lagrangian angle satisfies $\frac{d\theta_t}{dt} = \Delta \theta_t$, and therefore the graded and almost-calibrated conditions are preserved by the flow.\\
\indent    (b) \emph{\cite{Neves2007}} Any singularity of $F_t$ is Type II.\end{proposition}

Restricting to the case of flows in $\mathbb{C}^n$, note that by \cite[Lem.\ 3.26]{Joyce2015} a graded and embedded Lagrangian MCF must be noncompact; we therefore will often be dealing with noncompact flows. In this situation it is natural to make the assumption of \textit{bounded area ratios}, i.e.\  there exist $R,C>0$ such that $\forall r > R$, $\mathcal{H}^n(L_t \cap B_r(0)) < C.$ 

Under this assumption, it was shown by Neves that Type I blowups of singularities of graded Lagrangian mean curvature flows are unions of special Lagrangian cones:
\begin{theorem}[\cite{Neves2007}]\label{thm-nevesab}
Let $F_t:L\rightarrow \mathbb{C}^n$ be a graded Lagrangian mean curvature flow with bounded area ratios and $\theta$ bounded, $(x,T)$ be a singular spacetime point, and $F^i_s$ be a sequence of Type I rescalings. 

(A) There exist angles $\{\overline\theta_1,\ldots \overline\theta_N\}$ and integral special Lagrangian cones $L_1,\ldots, L_N$ such that after passing to a subsequence, for all $\phi \in C^\infty_c(\mathbb{C}^n)$, $f \in C^2(\mathbb{R})$, and $s<0$,
\[ \lim_{i\rightarrow \infty} \int_{L^i_s} f(\theta_{i,s})\phi \, d\mathcal{H}^n \, = \, \sum_{k=1}^N m_j f(\overline\theta_j)\mu_j(\phi),\]
where $m_j$ and $\mu_j$ denote the multiplicity and underlying Radon measure of $L_j$ respectively. Furthermore, the set of angles is independent of rescaling sequence. 

(B) If furthermore $L_t$ is almost-calibrated and rational, then for all $R>0$ and almost all $s<0$ and for any convergent subsequence of connected components $\Sigma^i$ of $B_{4R}(0) \cap L^i_s$ intersecting $B_R(0)$, there exists a special Lagrangian cone $L$ with angle $\overline \theta$ such that for every $\phi \in C^\infty_c(B_{2R})$, $f \in C^2(\mathbb{R})$,
\[ \lim_{i\rightarrow \infty} \int_{\Sigma^i} f(\theta_{i,s})\phi \, d\mathcal{H}^n \, = \,  m \,f(\overline\theta)\mu(\phi).\]
\end{theorem}

\subsection{Group Actions on Manifolds} \label{sec:GrpAct}

\indent\indent Let $M$ be a smooth manifold equipped with a smooth left $G$-action, where $G$ is a Lie group.  The left action of $g \in G$ on $p \in M$ will be denoted by $L_g(p) = g \cdot p \in M$.  For each point $p \in M$, we let $\mathcal{O}_p \subset M$ denote its $G$-orbit, and let $H_p := \text{Stab}(p) \leq G$ denote its stabiliser subgroup.  If $G$ is a compact Lie group, then $H_p \leq G$ is a closed Lie subgroup, and $\mathcal{O}_p \subset M$ is a compact embedded submanifold diffeomorphic to $G/H_p$. \\
\indent For each point $p \in M$, consider the orbit map $L^{(p)} \colon G \to M$ via $L^{(p)}(g) = g \cdot p$.  The derivative of $L^{(p)}$ at the identity $e \in G$ yields the \textit{infinitesimal action}
\begin{align*}
\rho_p &= -dL^{(p)}_e \colon \mathfrak{g} \to T_pM \\
\rho_p(X) & = dL^{(p)}_e(-X) = \left.\frac{d}{dt}\right|_{t = 0} \exp (-tX) \cdot p.
\end{align*}
The image of $\rho_p$ is the tangent space to the orbit of $p$, and the kernel of $\rho_p$ is the Lie algebra of the stabiliser:
\begin{align*}
    \text{Im}(\rho_p) & = T_p\mathcal{O}_p, & \Ker(\rho_p) & = \mathfrak{h}_p = \text{Lie}(H_p).
\end{align*}
Varying the point $p \in M$ yields a map
$\rho \colon \mathfrak{g} \to \Gamma(TM)$, which is both $G$-equivariant and (by our sign convention) a Lie algebra homomorphism:
	\begin{align}
	\rho(\text{Ad}_g(X)) & = (L_g)_*\rho(X), \\
		\rho([X,Y]) &= [\rho(X),\rho(Y)]. 
	\end{align}

\indent Two stabiliser groups $H_p$ and $H_q$ are said to have the same \textit{isotropy type} if they are conjugate in $G$.  
Isotropy type is an equivalence relation on the collection of stabiliser groups $\{\text{Stab}(p) \colon p \in M\}$; the equivalence class of $H$, denoted $(H)$, is simply the conjugacy class of $H$ in $G$.  We write $(H) \leq (K)$ if $H$ is conjugate to a subgroup of $K$. Similarly, two $G$-orbits $\mathcal{O}_p$ and $\mathcal{O}_q$ are said to have the same \textit{orbit type} if their stabiliser groups $H_p$ and $H_q$ have the same isotropy type.  
\begin{theorem}[Principal Orbit Theorem \cite{Alexandrino2015}] \label{thm:PrincipalOrbit} Let $G$ be a compact Lie group acting isometrically on a connected Riemannian manifold $M$.  Then there exists a unique maximal orbit type (called the \emph{principal orbit type}).  Let $M_0 \subset M$ denote the union of principal orbits.  Then:
\begin{enumerate}
    \item The subset $M_0$ is open and dense in $M$.
    \item The quotient $M_0/G$ is a connected smooth manifold, and is open and dense in $M/G$.
    \item The projection $M_0 \to M_0/G$ is a fiber bundle with fiber $G/H$, where $H$ is a principal stabiliser group.
\end{enumerate}
\end{theorem}

\indent More generally, for any isotropy type $(H)$, we let $M_{(H)} \subset M$ denote the union of $G$-orbits in $M$ with isotropy type $(H)$.  The projection $M_{(H)} \to M_{(H)}/G$ is again a fiber bundle with fiber $G/H$.  For more information, we refer the reader to \cite[pg.\ 43-47]{Boyer2008}, \cite[$\S$3.4-3.5]{Alexandrino2015}.

Throughout this work, we will be concerned with $G$-invariant submanifolds.  We say that a $G$-invariant subset $L \subset M$ is a \emph{$G$-invariant immersed submanifold} if there exists a smooth manifold $\hat L$ with a smooth left $G$-action and a $G$-equivariant immersion $F:\hat L \rightarrow M$ (i.e.\  such that $F(g\cdot p) = g\cdot F(p)$ for $p \in \hat L$, $g \in G$) with image equal to $L$. The following proposition provides a sufficient condition for the quotient of such an immersion to be smooth.

\begin{proposition}[Quotients of Equivariant Immersions]\label{prop-equivariantquotient}

Let $\hat L,M$ be smooth manifolds with a smooth left $G$-action, and $F:\hat L \rightarrow M$ be a smooth $G$-equivariant immersion. Assume that the $G$-action on $M$ has constant isotropy type, i.e.\  $M = M_{(H)}$.

Then $\hat l:=\hat L / G$, $Q := M / G$ are smooth manifolds, and there exists a smooth immersion $f: \hat l \rightarrow Q$, such that $f \circ \pi = \pi \circ F$, where $\pi \colon M \to M/G$ is the quotient map.
\end{proposition}
\begin{proof}

Choose $y \in \hat l := \hat L / G$, and $x \in \hat L$ such that $\pi(x) = y$. Since $F$ is an immersion, there exists $U \subset \hat L$, $U \ni x$ such that $F|_U \colon U\rightarrow M$ is a diffeomorphism onto its image. The $G$-invariant set $G \cdot U \subset \hat L$ is an open subset, and $\pi(G\cdot U) = \pi(U)$ is an open neighbourhood of $ y \in \hat l$.

Since $F$ is a $G$-invariant immersion with image in $M$, it follows that $F(G\cdot U) \subset M$ is a $G$-invariant submanifold of constant isotropy type. From Theorem \ref{thm:PrincipalOrbit}, the quotient $Q := M / G$ is a smooth manifold, and the quotient $\pi \colon F(G\cdot U) \rightarrow F(G\cdot U)/G$ is a smooth submersion to a smooth manifold $l_U \subset Q$. Since $F|_U$ is a $G$-invariant diffeomorphism to its image, it follows that there are diffeomorphisms $\pi(U) \approx \pi(F(U))\approx \pi(F(G\cdot U)) \approx l_U$. We have therefore found a diffeomorphism from a neighbourhood of an arbitrary $y \in \hat l$ to a smooth manifold, and so $\hat l$ is itself a smooth manifold.

Finally, choosing a suitable open cover $U_\alpha$ such that $\bigcup U_\alpha = \hat L$, the smooth maps $\pi(U_\alpha) \xrightarrow{\approx} l_{U_\alpha} \subset Q$ produced by the above construction glue together to produce the required immersion $f\colon \hat l \rightarrow Q$.
\end{proof}

\begin{remark}\label{rmk-immersedunionofembedded}
We note that the proof of Proposition \ref{prop-equivariantquotient} shows that a $G$-invariant immersed submanifold $L \subset M$ is the union of $G$-invariant embedded submanifolds, $L = \bigcup F (G \cdot U_\alpha)$.
\end{remark}

\begin{definition}
A $G$-invariant submanifold $L \subset M$ is \textit{type $(H)$} if $L \subset M_{(H)}$. In analogy with Proposition \ref{prop-equivariantquotient}, this will ensure that our quotients are smooth manifolds in the sequel.
\end{definition}

\section{Hamiltonian K\"{a}hler Actions and $G$-Invariant Lagrangians} \label{sect:HamKahler}\label{sec-3}

\indent \indent In this section, we deal with the theory of Hamiltonian actions on K\"ahler and Calabi-Yau manifolds, and of $G$-invariant Lagrangian submanifolds. Section \ref{sec-3.1} provides an introduction to moment maps for Hamiltonian K\"ahler actions and K\"ahler reduction in the case of general group actions. Section \ref{sec-3.2} is concerned with $G$-invariant Lagrangians in K\"ahler manifolds, particularly the crucial fact that such submanifolds lie in level sets of the moment map (Proposition \ref{prop:constrain}). This allows one to leverage K\"ahler reduction to reduce their dimension and codimension. Finally, in $\S$\ref{sec-3.3} we consider the Calabi-Yau case, and in particular show that if the $G$-action preserves the Calabi-Yau structure, then $G$-invariant Lagrangians remain in the same level set of the moment map under LMCF (Proposition \ref{prop-flowlevelset}).

\indent We maintain the following setup throughout this section.  Let $(\overline{M}^{2n}, \langle \cdot, \cdot \rangle, J, \omega)$ be a connected K\"{a}hler manifold equipped with a smooth $G$-action, where $G$ is a compact connected Lie group.  
The $G$-action on $\overline{M}$ is always assumed to be \textit{K\"{a}hler}, meaning that it preserves the K\"{a}hler structure:
\begin{align*}
L_g^*\langle \cdot, \cdot \rangle & = \langle \cdot, \cdot \rangle & L_g^*J & = J & L_g^*\omega  & = \omega
\end{align*}
for all $g \in G$.  Let $\mathfrak{g}$ be the Lie algebra of $G$, and recall that $\rho \colon \mathfrak{g} \to \Gamma(T\overline{M})$ denotes the infinitesimal action.  Since the $G$-action is K\"{a}hler, the vector fields $\rho(X)$ are Killing, real-holomorphic, and symplectic:
\begin{align}
\mathscr{L}_{\rho(X)} \langle\cdot,\cdot\rangle & = 0 & \mathscr{L}_{\rho(X)} J & = 0 &
\mathscr{L}_{\rho(X)} \omega & = 0.
\label{eq-killingfield}
\end{align}
We let $\Ad \colon G \to \GL(\mathfrak{g})$ be the adjoint action, and $\Ad^* \colon G \to \GL(\mathfrak{g}^*)$ denote the coadjoint action, where $\mathfrak{g}^* = \Hom(\mathfrak{g}; \R)$.  These $G$-actions are related via
$\langle \Ad_g^*(\xi), X \rangle = \langle \xi, \Ad_{g^{-1}}(X) \rangle$
for $\xi \in \mathfrak{g^*}$, $X \in \mathfrak{g}$, and $g \in G$, where here $\langle \cdot, \cdot \rangle \colon \mathfrak{g}^* \times \mathfrak{g} \to \R$ is the dual pairing.  

\subsection{Hamiltonian K\"{a}hler Actions on K\"{a}hler Manifolds}\label{sec-3.1}

\begin{definition} A $G$-action on $\overline{M}$ is \textit{Hamiltonian} if there exists a smooth map $\mu \colon \overline{M} \to \mathfrak{g}^*$, called a \textit{moment map}, such that: \\
\indent 1. For each $X \in \mathfrak{g}$, the function $\mu^X \colon \overline{M} \to \R$ given by $\mu^X(p) = \langle \mu(p), X \rangle$ is a Hamiltonian for the vector field $\rho(X)$, meaning that:
$$d\mu^X = -\iota_{\rho(X)}\omega.$$
\indent 2. The map $\mu$ is $G$-equivariant:
$$\mu(g \cdot p) = \Ad_g^*(\mu(p)).$$
\end{definition}
If there exists a $G$-invariant primitive of the symplectic form, then the action is Hamiltonian with a canonical choice of moment map:
\begin{lemma}\label{lem-momentliouville} Let $\lambda \in \Omega^1(\overline{M})$ be a $1$-form.  If $\lambda$ is a primitive of $\omega$ and a symplectic $G$-action on $(\overline{M}, \langle \cdot, \cdot \rangle, J, \omega)$ infinitesimally preserves $\lambda$, meaning that
\begin{align*}
d\lambda & = \omega, \\    
\mathscr{L}_{\rho(X)}\lambda & = 0, \ \ \ \forall X \in \mathfrak{g},
\end{align*}
	then $\mu^* := \lambda \circ \rho \colon \mathfrak{g} \to C^\infty(\overline{M})$ defines a moment map for the $G$-action via $\mu^X = \mu^*(X)$.
\end{lemma}

Henceforth, we consider a Hamiltonian K\"{a}hler $G$-action on $(\overline{M}, \langle \cdot, \cdot \rangle, J, \omega)$ and fix a moment map $\mu \colon \overline{M} \to \mathfrak{g}^*$.  The $G$-action then gives rise to two distinguished classes of subsets: namely, the $G$-orbits $\mathcal{O}_z \subset \overline{M}$ and the $\mu$-level sets $\mu^{-1}(\xi) \subset \overline{M}$.  The following series of lemmas explores the relationships between them.

\begin{lemma} \label{lem:KerImgMoment} Let $p \in \overline{M}$, let $\mathcal{O} \subset \overline{M}$ be the $G$-orbit of $p$, let $G_p = \mathrm{Stab}(p)$ be the stabiliser of $p$, and let $\mathfrak{g}_p$ be the Lie algebra of $G_p$. Then:
\begin{align*}
\Ker(d\mu_p) & = [J(T_p\mathcal{O})]^\perp. \\
\mathrm{Im}(d\mu_p) & = \mathrm{Ann}(\mathfrak{g}_p) = \{\xi \in \mathfrak{g}^* \colon \xi(X) = 0, \ \forall X \in \mathfrak{g}_p\}.
\end{align*}
\end{lemma}

\begin{proof} For any $V \in T_p\overline{M}$ and $X \in \mathfrak{g}$, we have
$$\langle d\mu(V), X \rangle = d\mu^X(V) = \omega(V, \rho(X)) = -\langle V, J \rho(X) \rangle.$$
Thus, $V \in \text{Ker}(d\mu_p)$ if and only if $V$ is orthogonal to $J\rho(X)$ for all $X \in \mathfrak{g}$, which proves the first claim.  For the second, note that if $X \in \mathfrak{g}_p$, then $\rho(X) = 0$, so that $\langle d\mu(V), X \rangle = \omega(V, \rho(X)) = 0$ for all $V \in \mathfrak{g}$.  This shows that $\text{Im}(d\mu_p) \subset \text{Ann}(\mathfrak{g}_p)$.  The reverse inclusion follows from the following dimension count:
$$\dim(\text{Im}(d\mu_p)) = \dim(T_p\overline M) - \dim(J (T_p\mathcal{O})^\perp) = \dim(T_p\mathcal{O}) = \dim(\mathfrak{g}/\mathfrak{g}_p) = \dim(\text{Ann}(\mathfrak{g}_p)).$$
\end{proof}

\indent For the next result, let $\langle \cdot, \cdot \rangle_{\mathfrak{g}}$ denote an $\Ad$-invariant inner product on $\mathfrak{g}$.  For each covector $\xi \in \mathfrak{g}^*$, let $\xi^\sharp \in \mathfrak{g}$ be its dual vector --- i.e.\  the unique vector for which $\xi(Y) = \langle \xi^\sharp, Y \rangle_{\mathfrak{g}}$ holds for all $Y \in \mathfrak{g}$.  Considering the coadjoint $G$-action on $\mathfrak{g}^*$, for each $\xi \in \mathfrak{g}^*$, we let $G_\xi \leq G$ denote the corresponding stabiliser group:
$$G_\xi = \{g \in G \colon \Ad_g^*(\xi) = \xi\}.$$
\begin{lemma} \label{lem:CentralValue} Let $\xi \in \mathfrak{g}^*$.  Then: \\
\indent (a) The $G_\xi$-action on $\overline{M}$ preserves the $\mu$-level set $\mu^{-1}(\xi) \subset \overline{M}$. \\
\indent (b) We have $G_\xi = G$ if and only if $\xi^\sharp \in \mathfrak{z}(\mathfrak{g})$.  In this case, we call $\xi \in \mathfrak{g}^*$ a \emph{central value}.
\end{lemma}

\begin{proof} (a) Let $p \in \mu^{-1}(\xi)$ and $g \in G_\xi$.  The $G$-equivariance of $\mu \colon \overline{M} \to \mathfrak{g}^*$ gives $\mu(g \cdot p) = \text{Ad}_g^*(\mu(p)) = \text{Ad}_g^*(\xi) = \xi$, so $g \cdot p \in \mu^{-1}(\xi)$. \\
\indent (b) We have $\text{Ad}_g\xi^\sharp = (\text{Ad}_g^*\xi)^\sharp$, and therefore
	\begin{align*}
		G_\xi = G \,\iff\,  \text{Ad}_g^*(\xi) = \xi, \, \forall g \in G  
		\,\iff\, \text{Ad}_g\xi^\sharp = \xi^\sharp, \,  \forall g \in G & \,\iff\, [Y,\xi^\sharp] = 0, \, \forall Y \in \mathfrak{g},
	\end{align*}
where the last reverse implication follows from surjectivity of the exponential map $\exp \colon \mathfrak{g} \to G$ when $G$ is a compact connected Lie group \cite[$\S$4.2]{Brocker1985}
\end{proof}
\begin{remark}\label{rem-semisimple1}
If $G$ is a compact semisimple Lie group, then the only central value is $0 \in \mathfrak{g}$.
\end{remark}
\indent By Lemma \ref{lem:CentralValue}, if $\xi \in \mathfrak{g}^*$ is a central value, then the subset $\mu^{-1}(\xi) \subset \overline{M}$ is $G$-invariant, and is therefore is a union of $G$-orbits.  Conversely, if a $\mu$-level set $\mu^{-1}(\xi) \subset \overline{M}$ contains at least one $G$-orbit, then $\xi \in \mathfrak{g}^*$ must be a central value.  This last claim follows from the following Lemma:

\begin{lemma}[\cite{Pacini2002}]\label{lem-isotropiclevelset} Let $\mathcal{O} \subset \overline{M}$ be a $G$-orbit in $\overline{M}$.  The following are equivalent: \\
\indent 1. The $G$-orbit $\mathcal{O}$ is isotropic. \\
\indent 2. The $G$-orbit $\mathcal{O}$ is contained in a $\mu$-level set $\mu^{-1}(\xi)$ for some $\xi \in \mathfrak{g}^*$. \\
\indent 3. There exists $p \in \mathcal{O}$ such that $\mu(p) \in \mathfrak{g}^*$ is central. \\
\indent 4. For all $p \in \mathcal{O}$, the image $\mu(p) \in \mathfrak{g}^*$ is central.
\end{lemma}

	\begin{proof}
	For the equivalence of 1 and 2, note that
		\begin{equation*}
			\mathscr{L}_{\rho(X)}\mu^Y \, = \, d\mu^Y(\rho(X)) \, = \, -(\iota_{\rho(Y)}\omega)(\rho(X)) \, = \, \omega( \rho(X), \rho(Y)).
		\end{equation*}
		Therefore,
		\begin{align*}
			\mu \text{ constant on }\mathcal{O} &\iff \mathscr{L}_{\rho(X)}\mu^Y = 0 \text{ on }\mathcal{O} &\text{ for all } X,Y \in \mathfrak{g}\\
			&\iff \omega(\rho(X),\rho(Y)) = 0 \text{ on }\mathcal{O} &\text{ for all }X,Y \in \mathfrak{g}\\
			&\iff \mathcal{O} \text{ isotropic}. &
		\end{align*}
    For the equivalence of 2, 3, and 4:
		\begin{align*}
			\mu(p) \text{ central} \iff \text{Ad}_g^*(\mu(p)) = \mu(p),\ \forall g \in G 
			&\iff \mu(g \cdot p) = \mu(p), \ \forall g \in G\\
			&\iff \mu \text{ constant on } \mathcal{O}.
		\end{align*}
	\end{proof}
	
\indent In summary, a $G$-orbit $\mathcal{O} \subset \overline{M}$ is isotropic if and only if it is contained in a $\mu$-level set $\mu^{-1}(\xi) \subset \overline{M}$ for some $\xi \in \mathfrak{g}^*$, and in this case, $\xi$ must be a central value.  Moreover, if $\mathcal{O}$ is an isotropic submanifold of dimension $n-k$, then at each $x \in \mathcal{O}$, there is a complex subspace $\mathcal{E}_x \subset T_x\overline{M}$ of complex dimension $k$ and orthogonal decompositions
\begin{align}
T_x\overline{M} & = \mathcal{E}_x \oplus T_x\mathcal{O} \oplus J(T_x\mathcal{O}) \label{eq-orthdecompositon} \\
\text{Ker}(d\mu_x) & = \mathcal{E}_x \oplus T_x\mathcal{O}. \label{eq-orthdecompositon2}
\end{align}
Here, the first splitting follows from Lemma \ref{lem:IsotropicDecomp}, and the second from Lemma \ref{lem:KerImgMoment}.

\begin{example} \label{ex-SO(n)} Take $\overline{M} = \C^n$, $n \geq 3$, with its usual flat K\"{a}hler structure.  Let $G = \SO(n)$ act on $\C^n = \R^n \oplus \R^n$ diagonally, meaning that $A \in \SO(n)$ acts on $(x,y) \in \R^n \oplus \R^n$ via
$$A \cdot (x,y) := (Ax, Ay).$$
This $\SO(n)$-action has three orbit types, as summarized in the following table:
$$\begin{tabular}{| c | c | c |} \hline
$(x,y)$ & Isotropy type & Diffeomorphism type of orbit \\ \hline \hline 
$(0,0)$ & $\SO(n)$ & Singleton \\ \hline
$(x,y) \neq (0,0)$ & $\SO(n-1)$ & $(n-1)$-dim sphere \\
$\{x,y\}$ linearly dependent & & $\Sph^{n-1} = \SO(n)/\SO(n-1)$ \\ \hline
$(x,y) \neq (0,0)$ & $\SO(n-2)$ & $(2n-3)$-dim Stiefel manifold \\
$\{x,y\}$ linearly independent & & $V_2(\R^n) = \SO(n)/\SO(n-2)$ \\ \hline
\end{tabular}$$
The (singular) orbits with isotropy type $(\SO(n-1))$ are $(n-1)$-dimensional isotropic submanifolds of $\C^n$.  Further, those that lie in the unit sphere $\Sph^{2n-1}(1) \subset \C^n$ are special Legendrian submanifolds of $\Sph^{2n-1}(1)$.

A moment map for the $\SO(n)$-action is $$\mu \colon \C^n \to \mathfrak{so}(n), \quad \mu(z) = \frac{i}{2}\begin{pmatrix} z_i\overline{z}_j - \overline{z}_i z_j \end{pmatrix}_{1 \leq i,j \leq n}.$$
Since $\SO(n)$ is simple, we have $\mathfrak{z}(\mathfrak{so}(n)) = 0$.  The level set $\mu^{-1}(0) \subset \C^n$ is the real affine variety
\begin{align*}
\mu^{-1}(0) & = \left\{z \in \C^n \colon z_i\overline{z}_j = \overline{z}_iz_j, \ \forall 1 \leq i,j \leq n \right\} \\
& = \left\{ (x,y) \in \R^n \oplus \R^n \colon y_ix_j - x_iy_j = 0, \ \forall 1 \leq i,j \leq n \right\} \\
& = \left\{ (x,y) \in \R^n \oplus \R^n \colon \{x,y\} \text{ is linearly dependent} \right\}\!.
\end{align*}
In particular, $\mu^{-1}(0) \subset \C^n$ is precisely the singular locus of the $\SO(n)$-action, consisting of those points with isotropy types $(\SO(n))$ and $(\SO(n-1))$.  Thus, $\mu^{-1}(0)$ decomposes as the disjoint union of the two subsets
\begin{align*}
    \mu^{-1}(0) \cap \C^n_{(\SO(n))} & = \{0\}, &
    \mu^{-1}(0) \cap \C^n_{(\SO(n-1))} & = \mu^{-1}(0) \setminus \{0\}.
\end{align*}
Finally, we point out that the principal locus of the $\SO(n)$-action on $\mu^{-1}(0)$, i.e.\ , the subset $\mu^{-1}(0) \cap \C^n_{(\SO(n-1))} \subset \C^n$, is an $(n+1)$-dimensional coisotropic cone.
\end{example}

\begin{example} \label{ex-Torus}  Take $\overline{M} = \C^n$, $n \geq 3$, with its usual flat K\"{a}hler structure.  Let $G = T^{n-1}$, the maximal torus of $\U(n)$, embedded in the standard way:
$$T^{n-1} = \left\{ \begin{pmatrix} e^{i\theta_1} & & \\ & \ddots & \\ & & e^{i\theta_n} \end{pmatrix} \colon \theta_1 + \cdots + \theta_n = 0 \right\}\!.$$
Accordingly, the induced $T^{n-1}$-action on $\C^n$ is:
$$(e^{i\theta_1}, \ldots, e^{i\theta_n}) \cdot (z_1, \ldots, z_n) := (e^{i\theta_1}z_1, \ldots, e^{i\theta_n}z_n), \ \ \text{ where } \theta_1 + \cdots + \theta_n = 0.$$
This $T^{n-1}$-action has $n$ orbit types, as summarized in the following table:
$$\begin{tabular}{| c | c | c |} \hline
$z = (z_1, \ldots, z_n)$ & Isotropy type & Diffeomorphism type of orbit \\ \hline \hline 
$0$ & $T^{n-1}$ & Singleton \\ \hline
Exactly $(n-1)$ $z_j$'s are $0$ & $T^{n-2}$ & $T^1$ \\ \hline
$\vdots$ & $\vdots$ & $\vdots$ \\ \hline
Exactly $2$ $z_j$'s are $0$ & $T^1$ & $T^{n-2}$ \\ \hline
At most one $z_j$ is $0$ & $\{\mathrm{Id}\}$ & $T^{n-1}$ \\ \hline
\end{tabular}$$
All of the $T^{n-1}$-orbits (regardless of isotropy type) are isotropic submanifolds of $\C^n$.  Note that the singular locus is the union of the $\binom{n}{2}$ axis complex $(n-2)$-planes in $\C^n$. We remark that the principal $T^{n-1}$-orbits that lie in the unit sphere $\Sph^{2n-1}(1) \subset \C^n$ are special Legendrian submanifolds of $\Sph^{2n-1}(1)$. \\
\indent Identifying $\mathfrak{t}^{n-1} = \{(i\xi_1, \ldots, i\xi_n) \in i\R^n \colon \sum \xi_j = 0\}$, a moment map for the action is $$\mu \colon \C^n \to \mathfrak{t}^{n-1},\quad \mu(z) = -\frac{1}{2n}i\left( n|z_1|^2 - |z|^2,\, \ldots,\,  n|z_n|^2 - |z|^2 \right)\!.$$
Since $T^{n-1}$ is abelian, we have $\mathfrak{z}(\mathfrak{t}^{n-1}) = \mathfrak{t}^{n-1} \cong \R^{n-1}$.  For $\xi = (i\xi_1, \ldots, i\xi_n) \in \mathfrak{t}^{n-1}$, the level set $\mu^{-1}(\xi) \subset \C^n$ is 
\begin{align*}
    \mu^{-1}(\xi) & = \left\{ z \in \C^n \colon |z_i|^2 -|z_1|^2 = 2(\xi_1 - \xi_i),\,\, \forall i \in \{2,\ldots,n\} \right\}.
\end{align*}
\end{example}

We now turn to the process of K\"{a}hler reduction.  Classically, if one assumes that the $G$-action on $\mu^{-1}(\xi) \subset \overline{M}$ is free, where $\xi \in \mathfrak{g}^*$ is a central value, then both $\mu^{-1}(\xi)$ and $\mu^{-1}(\xi)/G$ are smooth manifolds, and the latter inherits a natural K\"{a}hler structure.  However, since we are not assuming that the $G$-action is free, the sets $\mu^{-1}(\xi)$ and $\mu^{-1}(\xi)/G$ need not be smooth manifolds in general.  To remedy this, we fix an isotropy type $(H)$ and restrict attention to the stratum $\overline{M}_{(H)} \subset \overline{M}$, the union of $G$-orbits in $\overline{M}$ with isotropy type $(H)$.

\begin{theorem}[K\"{a}hler Reduction]\label{thm:KahlerReduction}
Let $(\overline{M}^{2n}, g, J, \omega)$ be a K\"{a}hler $2n$-manifold equipped with a Hamiltonian Kähler $G$-action, where $G$ is a compact connected Lie group, and fix a moment map $\mu \colon \overline{M} \to \mathfrak{g}^*$.  Moreover: \begin{itemize}
\item Fix a central value $\xi \in \mathfrak{g}^*$, so that $\mu^{-1}(\xi) \subset \overline{M}$ is $G$-invariant.
\item Fix an isotropy type $(H)$ for which the $G$-orbits $G/H$ are $(n-k)$-dimensional.
\item Fix a connected component $M$ of the intersection $M_\xi := \mu^{-1}(\xi) \cap \overline M_{(H)}$.
\end{itemize}
Then: \\
\indent (a) $M$ is a smooth manifold and the quotient space $Q := M/G$ is a smooth manifold.\\
\indent (b) Let $\pi \colon M \to Q$ denote the quotient map.  Then $Q$ admits a K\"ahler structure $(g_Q, J_Q, \omega_Q)$ such that:
\begin{itemize}
	  \item[(i)] $\pi:(M, g)\rightarrow (Q,g_Q)$ is a Riemannian submersion,
	  \item[(ii)] For all $X \in \mathcal{H} := \Ker(\pi_*)^\perp$, we have $\pi_* \circ J(X) = J_Q \circ \pi_*(X)$,
	  \item[(iii)] $\pi^*\omega_Q = \left.\omega\right|_{M}$.
\end{itemize}
\indent(c) For $z\in M$, the pushforward map $\pi_*:TM\rightarrow TQ$ restricted to $\mathcal{H}_z$ is a Hermitian isomorphism. In particular, the horizontal subbundle $\mathcal{H} \subset TM$ is $J$-invariant.
\end{theorem}
\begin{proof}
By redefining the moment map to be $\overline\mu := \mu - \xi$ if necessary, we may without loss of generality assume that $\xi = 0$. In the case where $G$ acts freely on $\overline M$, the theorem is standard, see for example \cite[Thm. 3.1]{Hitchin1987}.

(a) That each connected component $M \subset M_0 := \mu^{-1}(0) \cap \overline M_{(H)}$ is a smooth manifold is shown in \cite[Thm. 3.1]{Sjamaar1991}. That $Q = M/G$ is a smooth manifold then follows from the Principal Orbit Theorem \ref{thm:PrincipalOrbit}.

(b) The following argument is mostly taken from \cite{Sjamaar1991} and \cite{Mayrand2022}, where more details may be found. We consider the set $\overline M_{H} := \{w \in \overline M \, : \, \text{Stab}(w) = H\}$, which may be shown to be a K\"ahler submanifold of $\overline M$. The group $L := N_G(H)/H$ acts freely on $\overline M_H$, where $N_G(H)$ denotes the normaliser of $H$ in $G$. The Lie coalgebra $\mathfrak{l}^*$ may be identified with the subalgebra $\mathfrak{h}^o\cap (\mathfrak{g}^*)^H \leq \mathfrak{g}^*$, where $(\mathfrak{g}^*)^H$ denotes the fixed point set under the coadjoint action and $\mathfrak{h}^o$ is the annihilator of $\mathfrak{h}$ in $\mathfrak{g}^*$. In this way, $\mu$ restricts to a map $\mu_H:\overline M_H \rightarrow \mathfrak{l}^*$, which may be seen to be a moment map for the action of $L$ on $\overline M_H$. \\
\indent Now, $M_H := M \cap \overline M_H$ is a union of connected components of $\mu^{-1}_H(0)$, satisfying $G\cdot M_H = M$. By \cite[Thm. 3.1]{Hitchin1987}, the quotient $M_H / L$ may be given the structure of a K\"ahler manifold, and if $\mathcal{H}'$ denotes the horizontal bundle of the fibration $\pi_L:M_H \rightarrow M_H/L$, then the restriction of the quotient map $\pi_L |_{\mathcal{H}'}$ is a Hermitian isomorphism. Moreover $M_H/L \cong M/G = Q$, and therefore $Q$ inherits a K\"ahler structure $(g_Q, \omega_Q, J_Q)$ from $M_H / L$. We therefore have the following commutative diagram, where $\pi_L$ is a Riemannian submersion and $\iota$ is the inclusion map:
\[\xymatrix{
M_H \ar[r]^\iota \ar[rd]_{\pi_L} &M \ar[d]^{\pi} \\
&Q}\]
Restricting to the horizontal bundle at a point $z \in M_H$, we therefore have Hermitian vector space isomorphisms \[\iota_*|_{\mathcal{H}'_z}: \mathcal{H}'_z \rightarrow \iota_*(\mathcal{H}'_z) \leq T_z M, \quad \pi_L|_{\mathcal{H}'_z}:\mathcal{H}'_z \rightarrow T_{\pi_L(z)}Q,\]
and therefore $\pi|_{\iota_*(\mathcal{H}'_z)}: \iota_*(\mathcal{H}'_z) \rightarrow T_{\pi_L (z)}Q$ is also a Hermitian vector space isomorphism. We use this isomorphism at the points $z\in M_H$, pushed forwards by $(L_g)_*$ to a general point $g \cdot z \in M$, to prove the three claims of part (b).

(i) Note that $\Ker(\pi_*) = T\mathcal{O}$, and so we necessarily have the vector space decomposition $T_{g\cdot z}M = T_{g \cdot z}\mathcal{O} \,\oplus\, (L_g)_*\iota_* \mathcal{H}'_z$. It remains to show that this decomposition is orthogonal, so that $(L_g)_*\iota_*\mathcal{H}' = \mathcal{H}_{g\cdot z} = \Ker(\pi_*)^\perp_{g\cdot z}$. Since the action is isometric, it suffices to prove this at a point $z\in M_H$, from which it follows for any $g\cdot z \in G\cdot M_H = M$. Taking $X \in \mathfrak{g}$, $Y \in \iota_*\mathcal{H}'_z$, note that
\begin{equation}
\langle \rho_z(X),\, Y\rangle \, = \, \omega(\rho_z(X), JY) \, = \, -d\mu^X(JY) = 0,\label{eq-rhocalc1} 
\end{equation}
since $Y\in\iota_*\mathcal{H}'_z \implies JY \in \iota_*\mathcal{H}'_z \leq T_z M$ and $\mu$ is identically 0 on $M$.

(ii) Since $\pi_L:M_H\rightarrow Q$ is a K\"ahler reduction, for $z\in M_H$, $X \in \mathcal{H}_z$, we have $J_Q \circ \pi_* X = \pi_* \circ J X$. Then, since $G\circlearrowright \overline M$ is a K\"ahler action, for $X \in \mathcal{H}_{g\cdot z}$:
\[ J_Q\circ \pi_* X = J_Q \circ \pi_* \circ (L_{g^{-1}})_* X = \pi_*\circ J(L_{g^{-1}})_*X = \pi_*(L_{g^{-1}})_* JX = \pi_*\circ JX. \]

(iii) We prove the statement in two cases. If $X \in \mathcal{H}$, then for all $Y \in TM$,
\[\pi^* \omega_Q(X,Y) \, = \, \omega_Q(\pi_*(X), \pi_*(Y)) \, = \, \langle \pi_* J_Q X, \pi_* Y\rangle \, = \, \langle JX, Y\rangle \, = \, \omega|_M(X,Y).\]
If $X \in \mathfrak{g}$, then by an identical argument to (\ref{eq-rhocalc1}), $\pi^* \omega_Q(\rho(X),Y) = \omega|_M(X,Y) = 0$. \\
\indent (c) Any $z \in M$ is of the form $g \cdot w$ for $w \in M_H$. Then,  $\pi_*|_{\mathcal{H}_{z}} = \pi_*|_{\mathcal{H}_w} \circ (L_{g^{-1}})_*|_{\mathcal{H}_{z}}$ is a composition of Hermitian isomorphisms.
\end{proof}

\begin{remark} Let $\xi \in \mathfrak{g}^*$ be a central value, let $(H)$ be an isotropy type, and let $M$ be a connected component of $M_\xi = \mu^{-1}(\xi) \cap \overline{M}_{(H)}$.  At a point $z \in M$, there is an orthogonal decomposition
\begin{equation} \label{eq:OrthDecomp2}
T_zM = \mathcal{H}_z \oplus T_z\mathcal{O}_z.
\end{equation}
 Comparing with (\ref{eq-orthdecompositon2}), we see that each $\mathcal{H}_z$ is a complex subspace of $\mathcal{E}_z$.  In particular, if the $G$-orbits of type $(H)$ are $(n-k)$-dimensional, then $\dim(M) \leq n+k$. \\
 \indent We are primarily interested in the $k=1$ case.  In this situation, $M$ automatically attains its maximum dimension, $\dim(M) = n+1$.  Indeed, by (\ref{eq-orthdecompositon2}), we have $n-1 \leq \dim(M) \leq n+1$.  Moreover, since $\mathcal{H}$ is a symplectic subspace, it is even-dimensional, so $\dim(\mathcal{H}) = 2$.
 \end{remark}


For more on K\"{a}hler reduction, see \cite{Sjamaar1991}, \cite{Boyer2008}.

\subsection{$G$-Invariant Lagrangians of K\"{a}hler Manifolds}\label{sec-3.2}

\indent \indent We now consider the $G$-invariant Lagrangian submanifolds of our K\"{a}hler manifold $\overline{M}^{2n}$. We first show that such submanifolds are constrained to lie in level sets of the moment map:

\begin{proposition}
     \label{prop:constrain} If $L \subset \overline{M}$ is a connected $G$-invariant immersed Lagrangian submanifold, then $L \subset \mu^{-1}(\xi)$ for some central value $\xi \in \mathfrak{g}^*$.
\end{proposition}

\begin{proof}
Without loss of generality, we may assume that $L$ is embedded, since by Remark \ref{rmk-immersedunionofembedded} a $G$-invariant immersed submanifold is a union of $G$-invariant embedded submanifolds. 
Fix $z \in L$, and let $\mathcal{O} = G \cdot z$ be its $G$-orbit.  Since $L$ is a $G$-invariant Lagrangian, we have $\mathcal{O} \subset L$, so that $T_z\mathcal{O} \subset T_zL$, and hence $J(T_z\mathcal{O}) \subset J(T_zL) = (T_zL)^\perp$.  It follows that $T_zL \subset [J(T_z\mathcal{O})]^\perp = \Ker(d\mu_z)$ by Lemma  \ref{lem:KerImgMoment}.  That is, each $X \in T_zL$ has $d\mu_z(X) = 0$.  Since $L$ is connected, we deduce that $\mu$ is constant on $L$.  Finally, since $\mathcal{O}\subset L$ is an isotropic orbit, Lemma \ref{lem-isotropiclevelset} implies that the value of $\mu$ on $L$ is central.
\end{proof}


\begin{remark}\label{rem-semisimple2} Suppose $G$ is semi-simple. By Remark \ref{rem-semisimple1}, if $L \subset \overline{M}$ is a connected $G$-invariant immersed Lagrangian submanifold, then $L \subset \mu^{-1}(0)$.
\end{remark}

\indent There is also a converse to Proposition \ref{prop:constrain}:

\begin{proposition}\label{prop-lagissymmetric}
    Let $\xi \in \mathfrak{g}^*$ be a central value, and let $L \subset \mu^{-1}(\xi)$ be a closed embedded Lagrangian submanifold. Then $L$ is $G$-invariant.
\end{proposition}
\begin{proof}
Let $p \in L$, and let $\mathcal{O}$ be the $G$-orbit of $p$. By Lemma \ref{lem:KerImgMoment}, $\text{Ker}(d\mu_p) = [J(T_p\mathcal{O})]^\perp$. Therefore, $T_pL \leq \mathrm{Ker}(d\mu_p) = [J(T_p\mathcal{O})]^\perp$, and hence since $L$ is Lagrangian, $T_p\mathcal{O} \leq T_pL$.  

Now, letting $g \in G$ be arbitrary, we aim to show that $g \cdot p \in L$. By the surjectivity of the exponential map, there exists $X \in \mathfrak{g}$ such that $\text{exp}(X) = g$. The flow of $\rho(-X)$ on $\overline M$ exists for all time, and indeed is given explicitly by $\Phi_t(z) = \text{exp}(tX)\cdot z$. Since the vector field $\rho(-X)$ is tangent to the submanifold $L$ by the above, it follows by the closedness of $L$ that the flow $\Phi$ preserves $L$. Therefore, $g \cdot p = \text{exp}(X)\cdot p = \Phi_1(p) \in L$, as required.
\end{proof}

\begin{corollary}
Let $L \subset \overline{M}$ be a connected embedded closed Lagrangian submanifold, and fix an isotropy type $(H)$ of the $G$-action.  Then $L$ is $G$-invariant of type $(H)$ if and only if $L \subset M_\xi = \mu^{-1}(\xi) \cap \overline{M}_{(H)}$ for some central value $\xi \in \mathfrak{g}^*$.
\end{corollary}

This result implies that, in the case where the connected component $M$ of $\mu^{-1}(\xi)$ contains a Lagrangian submanifold, the complex vector spaces $\mathcal{E}_p$ and $\mathcal{H}_p$ appearing in the decompositions $T_p\overline{M} = \mathcal{E}_p \oplus T_p\mathcal{O} \oplus J(T_p\mathcal{O})$ and $T_pM = \mathcal{H}_p \oplus T_p\mathcal{O}$ of equations (\ref{eq-orthdecompositon}) and (\ref{eq:OrthDecomp2}), respectively, are equal at $p \in M$:

\begin{corollary} \label{cor:H-equals-E}
    Let $\xi \in \mathfrak{g}^*$ be a central value, let $(H)$ be an isotropy type of the $G$-action, and let $M$ be a connected component of $M_\xi = \mu^{-1}(\xi) \cap \overline{M}_{(H)}$.  Let $n-k$ denote the dimension of the $G$-orbits of type $(H)$.  If $M$ contains an embedded Lagrangian submanifold, then the complex vector bundle $\mathcal{H}$ has complex rank $k$, and the bundles $\mathcal{H}$ and $\mathcal{E}$ are equal on $M$.
\end{corollary}

\begin{proof} Let $L \subset M$ be an embedded Lagrangian submanifold.  Since the vector bundles $\mathcal{H}$ and $\mathcal{E}$ have constant rank on $M$, it suffices to prove that for some $p \in L$, we have $\mathcal{H}_p = \mathcal{E}_p$. \\
\indent By the proof of Proposition \ref{prop-lagissymmetric}, $T_p\mathcal{O} \leq T_p L$.  Let $F_p$ be the orthogonal complement of $T_p\mathcal{O}$ in $T_p L$, so that $\dim(F_p) = k$.  Since $L$ is Lagrangian, $J F_p \perp F_p$, and so $F_p \oplus J F_p$ is a complex vector space of complex dimension $k$.

By (\ref{eq-orthdecompositon}) and (\ref{eq:OrthDecomp2}), $F_p \leq \mathcal{H}_p \leq \mathcal{E}_p$, and since $\mathcal{H}_p$ is complex, $F_p \oplus J F_p \leq \mathcal{H}_p \leq \mathcal{E}_p$. Since $\mathcal{E}_p$ has complex dimension $k$, it follows that all vector spaces in this inclusion have the same dimension, and therefore are equal.
\end{proof}

Since $G$-invariant Lagrangians of type ${(H)}$ must lie in the smooth manifold $M_\xi$, we may quotient by the $G$-action to obtain a Lagrangian $l$ in the K\"ahler quotient $Q_\xi$. We therefore have the following bijection:

\begin{proposition}[$G$-Invariant Lagrangians in $M$ correspond to Lagrangians in $Q$]\label{prop-bijection1} Let $\xi \in \mathfrak{g}^*$ be a central value, let $(H)$ be an isotropy type of the $G$-action, and let $M$ be a connected component of $M_\xi$.  Recall the K\"{a}hler quotient map $\pi \colon M \to Q := M/G$.  Then there is a bijection: 
\begin{align*}
    \{G\text{-invariant immersed Lagrangians } L
    \subset M\} & \longleftrightarrow \{\text{immersed Lagrangians } l\subset Q\} \\
    L & \longmapsto L/G \\
    \pi^{-1}(l) & \longmapsfrom l
\end{align*}
\end{proposition}
\begin{proof}
We first show that the correspondences $\Phi_1 \colon L \mapsto L/G$ and $\Phi_2 \colon l \mapsto \pi^{-1}(l)$ in the statement of the proposition are well-defined with the stated domain and codomain. Let $L \subset M$ be a $G$-invariant immersed Lagrangian, and choose a $G$-equivariant immersion $F \colon \hat L \rightarrow M$ with image $L$. Proposition \ref{prop-equivariantquotient} gives an immersion $f \colon \hat l = \hat L/G \rightarrow Q$ with image $L/G$, such that $f \circ \pi = \pi \circ F$.  It follows by the K\"{a}hler Reduction Theorem \ref{thm:KahlerReduction} that
\begin{align*} \pi^*f^* \omega_Q &= F^* \pi^* \omega_Q = F^* \omega|_{M} = 0 
\end{align*}
so that $f^*\omega_Q = 0$, and hence $f$ is a Lagrangian immersion.

In the other direction, consider a Lagrangian immersion $f \colon \hat l \to Q_\xi$ with image $l$. By the Principal Orbit Theorem (Theorem \ref{thm:PrincipalOrbit}), $\pi \colon M \to Q$ is a $G/H$-bundle, and we may define $\hat L := f^*M$ to be the pullback of this bundle by $f$.  Then $\hat L$ has a fibrewise $G$-action, and there is a natural $G$-equivariant map $F\colon \hat L \rightarrow M$ which is the inclusion when restricted to each fiber.  It is easy to check that $F$ is a Lagrangian immersion, and the image of $F$ is $L := \pi^{-1}(l)$.

Finally, it is clear that the maps $\Phi_1, \Phi_2$ are inverses of each other, and therefore are bijections.
\end{proof}

\subsection{$G$-Invariant Lagrangians of Calabi-Yau Manifolds}\label{sec-3.3}

\indent \indent We now suppose that our connected K\"{a}hler manifold $(\overline{M}^{2n}, \langle \cdot, \cdot \rangle, J, \omega)$ is a Calabi-Yau manifold, equipped with a holomorphic volume form $\Omega$. The following lemma shows how the $G$-action interacts with the Calabi-Yau structure $\Omega$:

\begin{lemma}[\cite{Konno2017}] \label{lem:KahlerActionCY} There exists a unique element $a \in \mathfrak{g}^*$ such that 
$$(L_{\exp(X)})^*\Omega = e^{ia(X)}\Omega$$
for all $X \in \mathfrak{g}$.
\end{lemma}


\begin{remark} The above lemma does not require the K\"{a}hler $G$-action to be Hamiltonian.
\end{remark}

\begin{definition} A K\"{a}hler $G$-action on $\overline{M}$ is called \emph{Calabi-Yau} if it preserves $\Omega$ (equivalently, if and only if $a = 0$ in Lemma \ref{lem:KahlerActionCY}).
\end{definition}





If $L \subset \overline M$ is a connected $G$-invariant Lagrangian submanifold, then Proposition \ref{prop:constrain} tells us that there exists a central value $\xi \in \mathfrak{g}^*$ such that $L \subset \mu^{-1}(\xi)$. Using the extra structure of the Lagrangian angle, we may explicitly describe how a mean curvature flow of $G$-invariant Lagrangians moves through the level sets of $\mu$. Most importantly, in the case of a Calabi-Yau action, the flow remains in a single level set.

\begin{proposition}\label{prop-flowlevelset} Let $\overline{M}$ be a Calabi-Yau manifold equipped with a Hamiltonian K\"{a}hler $G$-action, and let $a \in \mathfrak{g}^*$ be the Lie coalgebra value of Lemma \ref{lem:KahlerActionCY}.  Let $L_0 \subset \overline{M}$ be a $G$-invariant immersed graded Lagrangian submanifold with Lagrangian angle $\theta$ and mean curvature vector $\vec{H}$, so that $L_0 \subset \mu^{-1}(\xi)$ for some central $\xi \in \mathfrak{g}^*$.  Then
$$d\mu(\vec H) = -a.$$
	Therefore, if $F_t: L \rightarrow \overline M$ is a mean curvature flow starting at $L_0$, i.e.\ 
		\begin{align*}
			\left(\frac{\partial F}{\partial t} \right)^\perp &= \vec H &\text{for } t \in [0,T),\\
			F_0(L) &= L_0 &\text{at } t=0,
		\end{align*} 
		then $L_t$ is also a $G$-invariant Lagrangian, and $L_t \subset \mu^{-1}(\xi - at)$.
		In particular, if the $G$-action on $\overline M$ is Hamiltonian and Calabi-Yau, then $d\mu(\vec{H}) = 0$ and $L_t \subset \mu^{-1}(\xi)$.
\end{proposition}

\begin{proof}
	For $z \in L$ and $V \in \mathfrak{g}$, using Lemma \ref{lem:KahlerActionCY},
		\begin{align*}
			e^{i\theta(z)}\text{vol}_L = \Omega\big|_{T_zL}= e^{-ia(tV)}\left( L^*_{\exp{tV}}\Omega\big|_{T_{\exp{tV}\cdot z}L}\right)= e^{-ia(tV)}e^{i\theta(\exp{tV}\cdot z)}\text{vol}_L,
		\end{align*}
and hence
$$e^{i\theta(z)} =e^{i(-a(tV) +\theta(\exp{tV}\cdot z))}.$$
Differentiating with respect to $t$ and setting $t=0$,
\begin{align}
    e^{i\theta(z)}\left(d\theta_z(\rho_z(V)) - a(V)\right) &= 0 \\
    d\theta_z(\rho_z(V)) & = a(V). \label{eq-dtheta}
\end{align}
Since $\vec H = J\nabla \theta$, it follows that
\begin{align*}
    \langle d\mu(\vec{H}), V\rangle &= \omega(J\nabla\theta, \rho(V)) \,=\, -\langle \nabla \theta, \rho(V)\rangle	\,=\, -d\theta(\rho(V)) \,=\, -a(V).
\end{align*}
Finally, for an MCF $F_t:L \rightarrow M$, it follows that 
\begin{align*}
    \frac{d \mu}{dt} &= d\mu(\vec{H}) = -a
\end{align*}    
so $L_t \subset \mu^{-1}(\xi - at).$
\end{proof}

\section{$G$-Invariant Lagrangian Submanifolds of $\C^n$}
\indent \indent We now specialise to our primary objects of interest: $G$-invariant Lagrangians in $\C^n$, where $\mathbb{C}^n$ is equipped with the flat Calabi-Yau structure $(\langle \cdot, \cdot \rangle, J, \omega, \Omega)$ and endowed with a K\"{a}hler $G$-action. We suppose that $G$ is compact and connected and the $G$-action is faithful and linear, so we may embed $G \leq \U(n)$ and view the $G$-action on $\C^n$ as a restriction of the standard $\text{U}(n)$-action on $\C^n$. The standard Liouville form $\lambda$ of equation (\ref{eq-liouvilleform}) gives a canonical moment map $\mu \colon \C^n \to \mathfrak{g}^*$ for the group action, which we describe in $\S$\ref{sec-4.1}.

In light of the work of $\S$\ref{sect:HamKahler}, we specialise to connected Lagrangians with a constant isotropy type $(H)$.  By Proposition \ref{prop:constrain}, given such a Lagrangian $L \subset \C^n$ there exists a connected component $M$ of $M_\xi := \mu^{-1}(\xi)\cap \mathbb{C}^n_{(H)}$ for which $L \subset M$. We therefore restrict attention to a connected component $M \subset M_\xi$, which we assume contains at least one Lagrangian. By (\ref{eq-orthdecompositon}), (\ref{eq:OrthDecomp2}), and Corollary \ref{cor:H-equals-E}, at $z \in M$ there are orthogonal decompositions
\begin{align} \label{eq:OrthDecSec4}
    T_z\C^n & = \mathcal{H}_z \oplus T_z\mathcal{O} \oplus J(T_z\mathcal{O}) \\
    T_zM & = \mathcal{H}_z \oplus T_z\mathcal{O} \nonumber
\end{align}
where the complex subspaces $\mathcal{H}_z = \text{Ker}(d\pi_z)^\perp \subset T_z\C^n$ form the horizontal distribution of the projection $\pi \colon M \to Q := M/G$.  
Moreover, if the $G$-action on $M$ has $(n-k)$-dimensional orbits, then $M$ and its K\"{a}hler quotient $Q$ have dimensions
\begin{align*}
    \dim(M) & = n+k & \dim(Q) & = 2k.
\end{align*}

\indent Aside from working in $\mathbb{C}^n$, we introduce two further restrictions in this chapter. Firstly, in $\S$\ref{sec-4.2} we narrow our focus to the particular level set $\mu^{-1}(0) \subset \C^n$. The key reason for this is given by Proposition \ref{prop-zerosymmetry}; if $z \in M \subset \mu^{-1}(0)$, then the complex line $P_z$ through $z$ is contained in $M \cup \{0\}$. Secondly, in $\S$\ref{sec-4.3} we restrict to the case of cohomogeneity-one Lagrangians (i.e.\ , $k = 1$), so that the bijection of Proposition \ref{prop-bijection1} is with 1-dimensional curves in the K\"{a}hler quotient. This reduces the study of mean curvature flow of such Lagrangians to a modified curve-shortening flow. In $\S$\ref{sec-4.4} we explore the ramifications of making both assumptions --- working with cohomogeneity-one Lagrangians in $\mu^{-1}(0)$. Proposition \ref{prop-profilebijection} sets up a bijection of Lagrangians $L \subset M$ and curves $l \subset P_z$, obtained by intersection with $P_z$, providing an alternative to Proposition \ref{prop-bijection1}. Via this bijection, mean curvature flow of Lagrangians corresponds to a flow of curves in $P_z$ which is independent of the choices of $G$ and $M$ (Corollary \ref{cor-equivariantlmcf}). This surprising and powerful observation enables the arguments of the subsequent chapters.

\subsection{Hamiltonian K\"{a}hler Actions on $\mathbb{C}^n$}\label{sec-4.1}
\indent \indent Our first observation is that the K\"{a}hler $G$-action on $\C^n$ is Hamiltonian, and its moment map admits a well-known explicit formula.

To derive it, we equip $\mathfrak{u}(n)$ with the standard $\Ad$-invariant inner product given by $\langle X,Y \rangle_{\mathfrak{u}(n)} := -\text{tr}(XY)$, orthogonally decompose $\mathfrak{u}(n) = \mathfrak{g} \oplus \mathfrak{g}^\perp$ with respect to $\langle \cdot, \cdot \rangle_{\mathfrak{u}(n)}$, and let $\pi_{\mathfrak{g}} \colon \mathfrak{u}(n) \to \mathfrak{g}$ denote the orthogonal projection. We also note that our $G$-action on $\C^n$ has infinitesimal action $\rho_z \colon \mathfrak{g} \to T_z\C^n$ given by $\rho_z(X) = -Xz$.
	
\begin{proposition}\label{prop-cnmomentmap}
	Let $G \leq \U(n)$. Then the induced $G$-action on $\mathbb{C}^n$ is a Hamiltonian K\"ahler action, with moment map
	\begin{align*}
	    \mu^* & := \lambda \circ \rho \colon \mathfrak{g} \to C^\infty(\C^n; \R) \\
	    \mu(z)^\sharp &:= -\frac{1}{2}\pi_\mathfrak{g}(iz\overline z^T) ,
	\end{align*} 
	where $\lambda$ is the standard Liouville form given by (\ref{eq-liouvilleform}). Moreover, if $G \leq \SU(n)$, then the $G$-action is Calabi-Yau.
\end{proposition}
	
\begin{proof}

	Note that the Liouville form $\lambda$ satisfies $d\lambda = \omega$ and $\mathscr{L}_X\lambda = 0$ for $X \in \mathfrak{g}$. Therefore, by Lemma \ref{lem-momentliouville} a moment map for the action is given by $\langle \mu(z),X\rangle := \lambda \circ \rho(X)$. The explicit form above is then given by a calculation (see for example \cite{DaSilva2006}).
\end{proof}


The complex cone structure of $\mathbb{C}^n \setminus \{0\}$ yields a natural $\mathbb{C}^*$-action given by scalar multiplication: that is, the $\C^*$-orbit of a point $z \in \mathbb{C}^n \setminus \{0\}$ is a complex line with the origin removed. We note that the $\C^*$-action and $\U(n)$-action commute with one another. Two notable sub-actions are the \textit{Hopf $\U(1)$-action}, which is also a sub-action of $\U(n)$:
\begin{equation} \label{eq:HopfAction}
e^{i\alpha} \cdot (z_1, \ldots, z_n) = (e^{i\alpha}z_1, \ldots, e^{i\alpha}z_n),
\end{equation}
and the $\mathbb{R}^+$-action: $$ r \cdot (z_1,\ldots,z_n) = (rz_1,\ldots,rz_n).$$
For emphasis, we will occasionally write $\U(1)_\Delta \leq \U(n)$ to denote the particular $\U(1)$-subgroup defined by the Hopf action (\ref{eq:HopfAction}).  These two actions give rise to two distinguished vector fields on $\mathbb{C}^n$:
\begin{align*}
    \frac{\partial}{\partial \alpha} & = \sum_{j=1}^n -y_j \frac{\partial}{\partial x_j} + x_j \frac{\partial}{\partial y_j} & r\cdot \frac{\partial}{\partial r} & = \sum_{j=1}^n x_j \frac{\partial}{\partial x_j} + y_j \frac{\partial}{\partial y_j},
\end{align*}
satisfying $J\!\left(r\cdot\tfrac{\partial}{\partial r}\right) = \tfrac{\partial}{\partial \alpha}$, where here $(z_1, \ldots, z_n) = (x_1 + iy_1, \ldots, x_n + iy_n)$.

We now observe that the level sets of $\mu \colon \C^n \to \mathfrak{g}^*$ inherit the $\U(1)$ symmetry:


\begin{proposition}[$\U(1)$-invariance of $\mu^{-1}(\xi)$]\label{prop-hopfsymmetry} Let $G \leq \U(n)$, let $(H)$ be an isotropy type of the $G$-action on $\C^n$, and fix a central value $\xi \in \mathfrak{g}^*$. \\
 \indent (a) The subsets $\mu^{-1}(\xi)$, $\C^n_{(H)}$ and $M_\xi = \mu^{-1}(\xi) \cap \C^n_{(H)} \subset \C^n$ are all $\U(1)$-invariant.  Moreover, the $\U(1)$-action on $M_\xi$ preserves horizontal vectors. \\
\indent (b) For a connected component $M \subset M_\xi$, the $\U(1)$-action on $M$ descends to a $\U(1)$-action on $Q := M/G$.  That is, there is a $\U(1)$-action on $Q$ such that $L_{e^{i\alpha}} \circ \pi = \pi \circ L_{e^{i\alpha}}$.
\end{proposition}

\begin{proof} (a) By Proposition \ref{prop-cnmomentmap}, a moment map for the standard $\U(n)$-action on $\mathbb{C}^n$ is given by $\mu^\sharp_{\mathfrak{u}(n)}(z) := -\frac{1}{2}iz\overline z^T$, and we note that $\mu^\sharp = \pi_{\mathfrak{g}} \circ \mu^\sharp_{\mathfrak{u}(n)}$. Therefore, for any $\alpha \in (0,2\pi]$ and $z \in \C^n$, the equivariance of the moment map gives:
	\begin{align*}
	    \mu^\sharp(e^{i\alpha}\cdot z) = \pi_\mathfrak{g} [ \mu^\sharp_{\mathfrak{u}(n)}(e^{i\alpha}\cdot z)]  = \pi_\mathfrak{g}[ \text{Ad}_{e^{i\alpha}}\mu^\sharp_{\mathfrak{u}(n)}(z)] = (\pi_\mathfrak{g} \circ \mu^\sharp_{\mathfrak{u}(n)})(z) = \mu^\sharp(z).
	\end{align*}
	This shows that $\mu^{-1}(\xi)$ is $\U(1)_\Delta$-invariant.  Since the $\U(1)_\Delta$-action commutes with the $G$-action on $\C^n$, the remaining claims in part (a) follow directly.

\indent (b) This follows from the commuting of the $\U(1)_\Delta$- and $G$-actions.
\end{proof}

\subsection{The Zero Level Set; Profile Planes}\label{sec-4.2}

\indent \indent The level set $\mu^{-1}(0) \subset \C^n$ is particularly special as it is invariant under the full $\mathbb{C}^*$-action. This additional symmetry allows us to work with \textit{profile planes} instead of the K\"ahler quotient, which will greatly simplify the analysis.

\begin{definition} The \emph{profile plane} at $z \in \C^n \setminus \{0\}$ is the complex line $P_z = \mathrm{span}_{\C}(z)$ containing both $z$ and the origin.  The reason for this terminology will become clear in Proposition \ref{prop-profilebijection}.
\end{definition}

\begin{proposition}[$\mathbb{C}^*$-invariance of $\mu^{-1}(0)$]\label{prop-zerosymmetry} Let $G \leq \U(n)$, and let $(H)$ be an isotropy type of the $G$-action on $\C^n$.  Let $M_0 := \mu^{-1}(0) \cap \mathbb{C}_{(H)}^n$. \\
    \indent (a) The subsets $\mu^{-1}(0)$, $\mathbb{C}^n_{(H)}$ and $M_0$ are $\mathbb{C}^*$-invariant. \\
    \indent (b) The set $\mu^{-1}(0)$ is a complex cone.  That is, if $z \in \mu^{-1}(0) \setminus \{0\}$, then $P_z \subset \mu^{-1}(0)$. \\
    \indent (c) For a connected component $M \subset M_0$, the $\mathbb{C}^*$-action on $M$ descends to a $\C^*$-action on $Q:=M/G$. That is, there is a $\mathbb{C}^*$-action on $Q$ such that 
    \[ L_{re^{i\alpha}} \circ \pi = \pi \circ L_{re^{i\alpha}}.\]
    \indent (d) The complex line $P_z$ is orthogonal to $T_z\mathcal{O}_z \oplus J \,T_z\mathcal{O}_z$.  Therefore, by (\ref{eq:OrthDecSec4}), $P_z \leq \mathcal{H}_z$.
\end{proposition}
\begin{proof}
    (a) From Proposition \ref{prop-hopfsymmetry}, each of these subsets is $\U(1)$-invariant. Noting that for $r \in \mathbb{R}^+$, we have $\mu(rz) = -\frac{1}{2}\pi_\mathfrak{g}(i rz\overline{rz}^T) = r^2 \mu(z)$, 
    it follows that $\mu(rz) = 0$ if and only if $\mu(z) = 0$, and so $\mu^{-1}(0)$ is also $\mathbb{R}^+$-invariant, hence $\C^*$-invariant.  Since the $\C^*$-action preserves the stabiliser under the $G$-action, the $\C^*$-invariance of $\C^n_{(H)}$ and $M_0$ follows. \\

    (b) This follows immediately from (a) and the fact that $0 \in \mu^{-1}(0)$. \\

    (c) This follows from the commuting of the $\mathbb{C}^*$- and $G$-actions.\\
    
    (d) Since $M_0$ is $\C^*$-invariant, the infinitesimal $\C^*$-action vector fields lie in $TM_0 = \mathcal{H} \oplus T\mathcal{O}$.  In particular, the complex line $P = \text{span}_{\R}(\frac{\partial}{\partial r}, J(\frac{\partial}{\partial r})) \subset \mathcal{H} \oplus T\mathcal{O}$.  Since $P$ is a complex line contained in $\mathcal{H} \oplus T\mathcal{O}$, where $\mathcal{H}$ is complex and $T\mathcal{O}$ is isotropic, it follows that $P \subset \mathcal{H} = (T\mathcal{O} \oplus J(T\mathcal{O}))^\perp$.
    
    
    
\end{proof}

\begin{remark} The preceding proposition is trivial when $\U(1)_{\Delta} \leq G$, since then $\mu^{-1}(0) = \{0\}$. It is therefore primarily of interest in the case $G \leq \SU(n)$.
\end{remark}

\indent Since $G \leq \U(n) \leq \GL_n(\C)$, the $G$-action on $\C^n$ induces a natural $G$-action on $\CP^{n-1} = \{P_z \colon z \in \C^n \setminus 0\}$ via $g \cdot P_z := P_{gz}$.  It is easy to check that this action is well-defined, independent of representative $z \in \C^n \setminus 0$.  As we now explain, this $G$-action in turn yields a natural action of a finite cyclic group on each profile plane, which will be of vital importance for our study of $G$-invariant Lagrangian submanifolds.  


\begin{proposition}[Symmetry of the profile planes]\label{prop-profilesymmetry}
Let $G \leq \SU(n)$, $z \in \mu^{-1}(0)\setminus \{0\}$, and $P_z$ be the complex line through $z$. Let $\widetilde{H} := \{g \in G: \, g\cdot P_z = P_z\}$ and $H := G_z$ denote the stabilisers of $P_z \in \CP^{n-1}$ and $z \in \C^n$, respectively. Then:\\
\indent (a) $H$ is a normal subgroup of $\widetilde H$, and \[\widetilde H / H \cong C_m = \{1,e^{\tfrac{2\pi i}{m}},\ldots,e^{\tfrac{(m-1)2\pi i}{m}}\} \leq \U(1)_\Delta\]
for some positive $m \in \Z$. \\
\indent (b) There is a natural action of $\widetilde H / H \cong C_m$ on $P_z \subset \C^n$ given both by $[\,\widetilde h\,] \cdot w :=  \widetilde h \cdot w$ and by the standard inclusion $C_m \hookrightarrow \U(1)_\Delta$.
\end{proposition}
\begin{proof} Let $h \in H$ and $\widetilde{h} \in \widetilde{H}$.  By definition, this means $h \cdot z = z$ and $\widetilde{h} \cdot z = \lambda z$ for some $\lambda \in \C^*$.  Therefore,
\[ \widetilde h^{-1} h \widetilde h \cdot z \, = \, \widetilde h^{-1}  h \cdot \lambda z \, = \, \widetilde h^{-1} \cdot \lambda z \, = \, z \ \implies \ \widetilde h^{-1} h \widetilde h \in H,\]
and so $H \leq \widetilde H$ is a normal subgroup.\\
\indent Now, the natural $\widetilde{H}$-action on $P_z \subset \C^n$ is unitary, so gives rise to a Lie group homomorphism $\widetilde{H} \to \U(1)_\Delta$.  Since $\widetilde{H} \leq G$ is compact, its image $K \leq \U(1)_\Delta$ is compact.  On the other hand, the $\widetilde{H}$-action on $P_z$ descends to a unitary $(\widetilde{H}/H)$-action on $P_z$, which is faithful since
\[ \widetilde {h_1} \cdot z = \widetilde h_2 \cdot z \, \iff \, \, \widetilde {h_1}^{-1}\widetilde {h_2} \in H \,\iff\, \left[\,\widetilde {h_1}\,\right] = \left[\,\widetilde {h_2}\,\right]. \]
Hence, there is an injective group homomorphism $\widetilde{H}/H \hookrightarrow \U(P_z) \cong \U(1)_\Delta$ whose image is $K$. 
Consequently, $\widetilde{H}/H \cong K \leq \U(1)_\Delta$ is a compact subgroup of $\U(1)_\Delta$. \\ 
\indent  Let $\widetilde{\mathfrak{h}}$, $\mathfrak{h}$ be the Lie algebras of $\widetilde{H}$, $H$, respectively.  Choose a complement $\mathfrak{k}$ of $\mathfrak{h} \subset \widetilde{\mathfrak{h}}$, so that $\widetilde{\mathfrak{h}} = \mathfrak{h} \oplus \mathfrak{k}$.  Let $X \in \mathfrak{k}$, and consider the vector $\rho(X)|_z \in T_z\C^n$.  Observe that $\rho(X)|_z \in P_z$ and $\rho(X)|_z \in T_z\mathcal{O}_z$, so Proposition \ref{prop-zerosymmetry}(d) gives $\rho(X)|_z = 0$.  Since the $G$-action is faithful, we have $X = 0$, and hence $\mathfrak{k} = 0$.  We conclude that the compact subgroup $K \leq \U(1)_\Delta$ is $0$-dimensional, and hence $K \cong C_m$ is a finite cyclic group.
\end{proof}

\begin{example} \label{ex-CyclicGroup}
\indent (a) Take $G = \SO(n)$, and recall the diagonal $\SO(n)$-action on $\C^n$ discussed in Example \ref{ex-SO(n)}.  For the isotropy type $(H) = (\SO(n-1))$, one can compute that $\widetilde{H} = \mathrm{O}(n-1)$, and hence $\widetilde{H}/H \cong C_2$. \\
\indent (b) Take $G = T^{n-1}$, and recall the $T^{n-1}$-action on $\C^n$ discussed in Example \ref{ex-Torus}.  For the isotropy type $(H) = (\{\mathrm{Id}\})$, one has $\widetilde{H}/H \cong C_n$.
\end{example}

\subsection{Cohomogeneity-One Lagrangian Submanifolds of $\C^n$}\label{sec-4.3}

\indent \indent We now restrict our attention to \textit{cohomogeneity-one} Lagrangian submanifolds of $\C^n$. That is, we study $G$-invariant, immersed Lagrangian submanifolds $L \subset \C^n$ of a fixed isotropy type $(H)$ for which the orbits $G/H$ are $(n-1)$-dimensional. By Proposition \ref{prop:constrain}, there exists a central value $\xi \in \mathfrak{g}^*$ and a connected component $M \subset M_\xi$ such that $L \subset M$. Our first result gives a formula for the Lagrangian angle in terms of the quotient curve $l := L/G$.

\begin{lemma}\label{lem-laganglegeneral}
    Let $U \subset Q$ be an open subset equipped with a global unit vector field $e_1$. Then there exists a smooth function $f:U \rightarrow S^1$ depending on $e_1$ such that the following is true.
   
    Let $L \subset M$ be a connected, immersed, cohomogeneity-one graded Lagrangian submanifold with Lagrangian angle $\theta$, and let $l := L/G$ be the quotient curve in $Q$. Then for any $z \in L\cap \pi^{-1}(U)$, 
    \[ \theta(z) \equiv \arg\left( f\left(\pi(z)\right)\right) + \arg\left(l'(\pi(z))\right)\quad (\emph{mod } \pi),\]
    where the argument of $l'$ is defined relative to $e_1$.
\end{lemma}
\begin{proof}
Let $\pm \nu_z$ denote the pair of unit elements of $\Lambda^{n-1}(T_z\mathcal{O})$, for $z \in M$. Then we may define the 2-valued complex-linear unit 1-form $\widehat \beta \colon T\C^n|_M \to \C$,
\[ \widehat \beta := \Omega(\pm \nu, \cdot ).\]
Using $d\Omega = 0$, it follows from a calculation that for $X \in T\mathcal{O}$, $\iota_X \widehat\beta = \iota_X d\widehat\beta = 0$. Therefore, $\widehat{\beta}$ descends to a smooth $2$-valued complex-linear unit $1$-form $\beta \in \Omega^1(Q; \C)$, in the sense that $\widehat\beta(X) = \beta(\pi_* X)$ for $X \in TQ$.  Note that each $\beta_p(X)$ is a pair of complex numbers of opposite sign. 

Now, define the function $f \colon U \to S^1$ by $f(p) = \beta_p(e_1)$. Denoting by $T$ a global unit tangent vector of $l$ and by $\widehat T$ the horizontal lift of $T$, it follows by the definition of $\theta$ that
\begin{align*}
    \pm e^{i\theta(z)} \, &=\, \pm \Omega(\nu_z, \widehat T)\, = \, \beta_{\pi(z)}(T) \, = \, \pm f(\pi(z)) e^{i \arg(l'(\pi(z)))},
\end{align*}
which implies the result.
\end{proof}

We next show that if such a Lagrangian submanifold is almost-calibrated, then it must be embedded and diffeomorphic to $\mathbb{R}\times G/H$.

\begin{proposition} \label{prop-embedded}
Let $L \subset M$ be a connected, immersed, cohomogeneity-one $G$-invariant Lagrangian submanifold of type $(H)$. Then:
\begin{itemize}
 \item[(a)] If $L$ is zero-Maslov and embedded, then $L$ is diffeomorphic to $\mathbb{R} \times G/H$. Moreover, if $L$ is complete, then the ends of $L$ are unbounded.
\item[(b)] If $L$ is almost-calibrated, then $L$ is embedded.
\end{itemize}
	\end{proposition}

\begin{proof} (a) If $L$ is a Lagrangian satisfying the conditions of (a), then \cite[Lem.\ 3.26]{Joyce2015} implies that $L$ is non-compact.  Now, note that $L$ is a $G/H$-bundle over the $1$-manifold $l := L/G$, and that $G/H$ is compact.  Therefore, since $L$ is non-compact and connected, it follows that $l$ is a non-compact connected $1$-manifold, and hence is diffeomorphic to $\R$.  Thus, $L$ is a $G/H$-bundle over $\R$, and any embedding $f:\mathbb{R} \rightarrow M / G$ lifts to a $G$-equivariant embedding $F \colon \mathbb{R} \times G/H \rightarrow L$. \\
\indent Finally, suppose $L$ is complete, so that $L \subset \C^n$ is a closed subset.  For all $p \in G/H$, the sequences $F(k,p) \in L$ and $F(-k,p) \in L$ are unbounded as $k \to \infty$.  Indeed, if $\{F(k,p)\}$ (say) were bounded, then it would admit a convergent subsequence in $L$, and hence (since $F$ is an embedding) the sequence $(k,p) \in \R \times G/H$ would admit a convergent subsequence, which is absurd. \\
	
(b) If $L$ satisfies the conditions of (b), then by Proposition \ref{prop-equivariantquotient} the quotient $l := L/G$ is a connected immersed 1-manifold in $Q := M/G$, and so may be parametrised with one real parameter. Assume for a contradiction that $L$ is not embedded, from which it follows that $l$ is not embedded. We choose an immersion $F\colon \mathbb{R}\rightarrow Q$ whose image is $l$. \\
\indent Since $F$ is a continuous map that is not a homeomorphism onto its image, it follows that either $F$ is not injective or that there exists an open set $U \subset \mathbb{R}$ for which $F(\mathbb{R} \setminus U) \subset l$ is not a closed subset of $l$.  In either case, there exists an open set $U \subset \mathbb{R}$, a sequence $s_k \in \mathbb{R} \setminus U$, and a real number $s \in U$, for which $F(s_k) \to F(s)$. Without loss of generality, we may assume that $s < s_k$. \\
\indent Now, the distance between $F(s)$ and $F(s_k)$ can be made arbitrarily small.  Therefore, for any $\epsilon > 0$, using Lemma \ref{lem-laganglegeneral} and choosing $k$ sufficiently large, we may complete the curve segment $F((s,s_k))$ to a smooth embedded loop $\widetilde{\gamma} = F((s,s_k)) \cup \gamma$ in $Q$ in such a way that the change in Lagrangian angle across $\gamma$ is at most $\pi + \epsilon$.  By choosing $\epsilon > 0$ sufficiently small and using the fact that $L$ is almost-calibrated, we can make the loop $\widetilde{\gamma}$ in $Q$ correspond to a cohomogeneity-one Lagrangian in $M$ that is almost-calibrated, embedded, and compact, which contradicts (a).
\end{proof}




Finally, we prove that a cohomogeneity-one Lagrangian $L$ is exact if and only if $L \subset \mu^{-1}(0)$.

\begin{proposition} \label{prop-exact} Let $L \subset \C^n$ be a connected, almost-calibrated, cohomogeneity-one $G$-invariant Lagrangian submanifold of type $(H)$, so that $L \subset \mu^{-1}(\xi)$ for some central value $\xi \in \mathfrak{g}^*$.  Then $L$ is exact if and only if $\xi = 0$.
\end{proposition}

\begin{proof}
First, assume that $\xi = 0$. Let $\gamma:S^1 \rightarrow L$ be a loop, and choose $z \in L$. By Proposition \ref{prop-embedded}, $L$ is homeomorphic to $\mathbb{R} \times G/H$, and so $\gamma$ is homotopic to a loop $\overline\gamma:[0,2\pi) \rightarrow \mathcal{O}_z$ within $L$. Then there exists a smooth path $X_s \in \mathfrak{g}$ such that 
        \[ \frac{d\overline\gamma}{ds} = \rho(X_s). \]
        Therefore, by Proposition \ref{prop-cnmomentmap}, and using the fact that $d\lambda|_L = \omega|_L = 0$:
        \begin{align*}
            [\lambda][\gamma] &= \int_\gamma \lambda
            = \int_{\overline\gamma} \lambda 
            = \int_{0}^{2\pi} \lambda(\rho(X_s))\, ds
            = \int_0^{2\pi} \mu^*(X_s) = 0,
        \end{align*}
        and so $L$ is exact.
        
        Now assume $L$ is exact, and choose $z \in L$. We wish to show that $\mu(z) = 0$, for which it is sufficient to show that $\mu^X(z) = 0$ for all $X \in \mathfrak{g}$ such that $|\rho_z(X)|=1$. To do this, we would like to use a test loop obtained by exponentiating $X$. However, we cannot be sure that $\{\exp(tX)\cdot z: t\in\mathbb{R}\}$ contains a closed loop.
        
        Instead, note that since the orbit $\mathcal{O}_z$ is compact, for any $\varepsilon > 0$ there must exist $t_0$, $t_1 \in \mathbb{R}$ with $t_1 - t_0 > 1$ and such that
        \begin{align*}
            |\exp(t_0 X)\cdot z - \exp(t_1 X)\cdot z| < \varepsilon\\
            \implies |z - \exp((t_1 - t_0)X)\cdot z| < \varepsilon.
        \end{align*}
        Fix $\varepsilon > 0$ and $t_0, t_1 \in \R$ satisfying the above. Choosing $\varepsilon < \widetilde \varepsilon < 2\varepsilon$, we may complete the curve $\gamma(s):= \exp(sX)\cdot z$ to a smooth closed curve with constant speed $|\rho(X)| = 1$,
        \begin{align*}
            &\overline\gamma:[0,t_1 - t_0 + \widetilde \varepsilon] \rightarrow \mathcal{O}_z,\\
            &\overline\gamma(0) = \overline\gamma(t_1-t_0+\widetilde\varepsilon),\\
            &\overline\gamma(s) = \gamma(s) \quad \text{ for } s \in [0, t_1-t_0].
        \end{align*}
        It then follows that
        \begin{align*}
            0 = \left| \int_{\overline\gamma}\lambda \right| \, =\, \left| -\int_0^{t_1-t_0} \lambda(\rho(X))\, ds \, + \, \int_{t_1-t_0}^{t_1-t_0 + \widetilde \varepsilon} \lambda \left(\frac{d\overline\gamma}{ds}\right)\,ds \right| \, \geq \, |\mu^X(z)| - \max_{z \in \mathcal{O}_z}|z| \cdot 2\varepsilon.
        \end{align*}
        Since $\varepsilon > 0$ was arbitrary, this implies that $\mu^X(z) = 0$, as required.
\end{proof}

\subsection{Cohomogeneity One Lagrangians in the Zero Level Set}\label{sec-4.4}

\indent \indent We now study cohomogeneity-one $G$-invariant Lagrangians in the zero level set, $\mu^{-1}(0)$. For this purpose, we fix an isotropy type $(H)$ of the $G$-action on $\C^n$ with $(n-1)$-dimensional orbits $G/H$, and a connected component $M$ of $M_0$.  At any point $z \in M$, the inclusion $P_z \subset \mathcal{H}_z$ of Proposition \ref{prop-zerosymmetry}(c) is equality for dimension reasons, therefore $\text{dim}(M) = n+1$ and $\text{dim}(Q) = 2$. The orthogonal decomposition (\ref{eq:OrthDecSec4}) may be written as:
\begin{align}
T_z\mathbb{C}^n & = P_z \oplus T_z\mathcal{O}_z \oplus J(T_z\mathcal{O}_z) \label{eq-orthdecompositoncohom1} \\
&= P_z \oplus \rho_z(\mathfrak{g}) \oplus J\rho_z(\mathfrak{g}).\notag
\end{align}
In fact, we can say more:




\begin{proposition}\label{prop-mu0cohom1}
Let $M$ be a connected component of $M_0$ and $z \in M$. Then:\\
\indent(a) $M = G\cdot P_z = (G \times \mathbb{C}^*)\cdot \{z\}$. Therefore, $M$ is a smooth cone, and is closed. \\ 
\indent(b) For $w \in M$, there exists $g \in G$ such that $P_w = g \cdot P_z$.\\
\indent(c) The cyclic group $C_m$ acting on $P_z$ has order $m\,|\,2n$. 
\end{proposition}
\begin{proof} (a) Note that $G\cdot P_z$ is a closed submanifold without boundary of the same dimension as $M$.  Since $M$ is connected, it follows that $G \cdot P_z = M$.\\

\indent (b) By part (a), we have $w = g \cdot \lambda z$ for some $g \in G$ and $\lambda \in \mathbb{C}^*$, so $P_w = P_{g \cdot \lambda z} = P_{g \cdot z} = g \cdot P_z$. \\ 


\indent (c) As in the proof of Lemma \ref{lem-laganglegeneral}, let $\pm \nu_z$ denote the pair of unit elements of $\Lambda^{n-1}(T_z\mathcal{O})$, for $z \in M$, and define the 2-valued complex-linear unit 1-form $\widehat \beta_w := \Omega(\pm \nu_w, \cdot)$. Recall that for each $X \in T_w M$, $\widehat \beta_w(X)$ is a pair of complex numbers of opposite sign.

Now fix an element $g \in \widetilde{H}$ such that $[g] \in {\widetilde H}/{H}$ corresponds to $e^{\frac{2\pi i}{m}} \in C_m$. It follows that for $w \in P_z$, $X \in T_w P_z$,
\[ g \cdot w = e^{\frac{2\pi i}{m}} w, \quad (L_g)_* X = \left(L_{e^{\frac{2\pi i}{m}}}\right)_* X.\]
We may then calculate the value $(L_g)^*\widehat\beta_w(X)$ in two different ways. Denoting $\overline\alpha := \frac{2\pi}{m}$ for convenience, and using the fact that $L_{e^{i\phi}}^*\Omega = e^{in\phi}\Omega$,
\begin{align*}
(L_g)^*\widehat\beta_w(X) \, &= \, \Omega_{g\cdot w}(\pm \nu_{g\cdot w}, (L_g)_*X)\,=\, (L_{e^{i\overline\alpha}})^*\Omega_w(\pm \nu_w, X)\,=\, e^{in\overline\alpha}\widehat\beta_w(X).
\end{align*}
On the other hand, by the $G$-invariance of $\Omega$,
\begin{align*}
(L_g)^*\widehat\beta_w(X) \, &= \, \Omega_{g\cdot w}(\pm \nu_{g\cdot w}, (L_g)_*X)\,= \, (L_g)^*\Omega_w(\pm \nu_w, X)\, = \, \Omega_w(\pm \nu_w, X)\,= \, \widehat\beta_w(X).
\end{align*}
It follows that $e^{\frac{i 2n\pi}{m}} = \pm 1$, and so $m \mid 2n$.

\end{proof}

When working in the level set $\mu^{-1}(0)$, cohomogeneity-one $G$-invariant Lagrangians are characterised by their intersection with a complex line $P_z$. This gives the following bijection:

\begin{proposition}[Bijection between $G$-invariant Lagrangians and profile curves]\label{prop-profilebijection} Let $M$ be a connected component of $M_0 := \mu^{-1}(0) \cap \C^n_{(H)}$ and let $z \in M$. Define $P_z$ and $C_m$ as in Proposition \ref{prop-profilesymmetry}. Then there are bijections:
\begin{align*}
    \{G\text{-invariant immersed Lagrangians in }M\} & \longleftrightarrow \{C_m\text{-invariant immersed curves in }P_{z}\}\\
    \{G\text{-invariant embedded Lagrangians of }M\} & \longleftrightarrow \{C_m\text{-invariant embedded curves in }P_{z}\}\\
    L &\longmapsto L \cap P_z\\
    G\cdot l &\longmapsfrom l
\end{align*}
\end{proposition}

\begin{proof} 
Consider first a $G$-equivariant immersion $F:\widehat L \to M$. We will demonstrate that the map $F|_{F^{-1}(P_z)}: F^{-1}(P_z) \rightarrow P_z$ is a $C_m$-equivariant immersion.

Take $p \in F^{-1}(P_z)$, and choose an open neighbourhood $U \ni p$ in $\widehat L$ such that $F|_U$ is an embedding; the image $F(U)$ is therefore an embedded Lagrangian. By the decomposition (\ref{eq-orthdecompositoncohom1}), we have $T_pF(U) + T_p P_z = T_p M$, so the intersection of $F(U)$ and $P_z$ is transverse. It follows that $F(U) \cap P_z$ is a smooth 1-manifold. Therefore, $U \cap F^{-1}(P_z)$ is a 1-manifold and $F|_{U \cap F^{-1}(P_z)}$ is an embedding, which shows that $F|_{F^{-1}(P_z)}$ is an immersion. The $C_m$-equivariance simply follows from the $G$-equivariance of $F$, and if $F$ is an embedding then clearly $F|_{F^{-1}(P_z)}$ also is.

In the other direction, consider a $C_m$-equivariant immersed curve $f: \widehat l \rightarrow P_z$, and denote the stabiliser group $H := G_z$. We note that $G/H \times \widehat l$ is a smooth compact $n$-manifold, with a natural $\widetilde H / H \cong C_m$-action:
\begin{equation*}
    [\alpha]\cdot([g],p) := ([g\alpha^{-1}],\alpha\cdot p),
\end{equation*}
and a natural $G$-action:
\begin{equation*}
    \overline g\cdot([g],p) := ([\overline g g],p),
\end{equation*}
which commute with each other. Therefore, the $G$-action descends to the smooth $n$-dimensional compact quotient $\widehat L := ( G/H \times \widehat l\,\, ) / C_m$. Now, define the $G$-equivariant map
\begin{align*}
    F&: \widehat L \rightarrow M,\quad
    F([[g],p]) := g \cdot f(p).
\end{align*}
This map is well-defined, in the sense that it is independent of choice of representatives, and may be seen to be a proper immersion. Moreover, if $f$ is injective, then $F$ is also injective, since:
\begin{align*}
    F([[g],p]) = F([\widetilde g],p]) &\implies g\cdot f(p) = \widetilde g \cdot f(\widetilde p)\\
    &\implies \widetilde g^{-1}g \cdot p = \widetilde p\\
    &\implies [\widetilde g^{-1}g] \cdot ([g],p) = ([gg^{-1}\widetilde g],\widetilde g g p) = ([\widetilde g],\widetilde p)\\
    &\implies [[g],p] = [[\widetilde g],p].
\end{align*}
It follows that if $f$ is an embedding, then $F$ is also an embedding.

To see that $F$ is a Lagrangian immersion, note that for $p \in \widehat L$ and $w := F(p)$, $T_w\mathcal{O}_w \subset F_*(T_p\widehat L) \subset T_wM = P_w \oplus T_w\mathcal{O}_w$, recalling the decomposition (\ref{eq-orthdecompositoncohom1}).  Since $T_w\mathcal{O}_w$ is an $(n-1)$-dimensional isotropic subspace, $\widehat L$ is $n$-dimensional, and $P_w$ is a complex line, it follows that $F_*(T_pL)$ is isotropic, and hence is Lagrangian for dimension reasons.

The above two maps correspond to $L \mapsto L \cap P_z$ and $l \mapsto G \cdot l$ on the level of subsets, for a $G$-invariant immersed Lagrangian $L \subset M$ and a $C_m$-invariant immersed curve $l \subset P_z$ respectively. These maps may be seen to be inverses of each other, and therefore the bijection is complete.

\end{proof}

It is a remarkable fact that the mean curvature of a cohomogeneity-one Lagrangian submanifold of $\C^n$ that lies in $\mu^{-1}(0)$ may be expressed entirely in terms of the profile curve $l:= L\cap P_z$, independent of the choice of subgroup $G \leq \SU(n)$, and the same is almost true of the second fundamental form.  The study of mean curvature flow of cohomogeneity-one Lagrangians in $\mu^{-1}(0)$ therefore reduces to the study of one particular geometric flow of curves in the plane.

\begin{proposition}[Curvature of a $G$-invariant Lagrangian]\label{prop-curvature}
Let $M$ be a connected component of $M_0 := \mu^{-1}(0) \cap \C^n_{(H)}$. Let $L \subset M$ be a connected, cohomogeneity-one $G$-invariant embedded Lagrangian of type $(H)$. Let $z \in M$, let $l := P_z \cap L$ be the profile curve of $L$, and let $\vec{k}$ be the curvature of $l$. Then:

\indent (a) The mean curvature $\vec{H}$ of $L$ at $w\in l$ is given by $$ \vec{H} = \vec{k} - (n-1)\frac{w^\perp}{|w|^2},$$ 
where $w^\perp$ is the projection of the position vector $w$ to the normal space of $l \subset P_z$. \\
\indent (b) Recalling the definitions of $h \in \Gamma(\Sym^3(T^*L))$ and $\alpha \in \Omega^1(L)$ from (\ref{def-h}) and (\ref{def-alpha}), define
\begin{align}
k & := |\vec k|, &
p(w) & :=  \tfrac{|w^\perp|}{|w|^2}, &
r(w) & := |w|, &
a_{ij} & := g^{kl}\alpha_l h_{ijk} = \langle \vec{H}, A(e_i,e_j)\rangle.\label{eq-meancurvature}
\end{align}
Then there exists a constant $C = C(n,G,M)$ (where $M \subset \mu^{-1}(0)\setminus \{0\}$ is the connected component containing $L$) such that at $w \in l$,
\begin{align*}
|\vec{H}|^2 & = (k -(n-1)p)^2, & \big| A\big|^2 & \leq k^2 + Cr^{-2}, & |a|^2 & = (k^2 + (n-1)p^2)|\vec{H}|^2.
\end{align*}
\end{proposition}
\begin{proof}
Choose an arclength parametrisation $\gamma: (-\varepsilon,\varepsilon) \rightarrow l$ of $l$ with $\gamma(0) = w$. By identifying $P_z \cong \mathbb{C}$, we may find functions $r(s), \alpha(s)$ such that $\gamma(s) = r(s)e^{i \alpha(s)}$. Then we may define the unit tangent and normal vectors to $l$ via
\begin{align*}
    T := \frac{d \gamma}{d s} &= \frac{r'(s)}{r(s)}r(s) \frac{\partial}{\partial r} + \alpha'(s) \frac{\partial}{\partial \alpha}\\
    N := J\frac{d \gamma}{d s} &= -\alpha'(s)r(s) \frac{\partial}{\partial r} + \frac{r'(s)}{r(s)} \frac{\partial}{\partial \alpha}
\end{align*}
which when rotated by the $G$-action produce global unit vector fields on $L$.  Choose $X_1,\ldots,X_{n-1} \in \mathfrak{g}$ such that $\{\rho_w(X_1), \ldots \rho_w(X_{n-1}), T\} =: \{e_1,\ldots,e_n\}$ is an orthonormal basis of $T_w L$. Keeping in mind the orthogonal decomposition (\ref{eq-orthdecompositoncohom1}), and denoting by $A^{\mathcal{O}_w}$ the second fundamental form of the orbit $\mathcal{O}_w$, we calculate the components of the second fundamental form $h$ of $L \subset \C^n$ in this basis.  Using the index ranges $1 \leq i,j,k \leq n-1$ throughout the proof, we compute:
\begin{align*}
    h_{nnn}\big|_w &= \langle \overline \nabla_T T, JT\rangle\big|_w = \left\langle \frac{\partial^2 \gamma}{\partial s^2}, J\frac{\partial \gamma}{\partial s} \right\rangle\bigg|_{s=0} = k(0),\\
    h_{inn}\big|_w  &= \left\langle \frac{\partial^2 \gamma}{\partial s^2}, J\rho(X_i) \right\rangle\bigg|_{s=0} = 0,\\
    h_{ijn}\big|_w &= \left\langle \overline\nabla_{\rho(X_i)}\rho(X_j), JT\right\rangle\big|_w
    \,=\, -\alpha'(0) \left\langle \overline\nabla_{\rho(X_i)}\rho(X_j), r\frac{\partial}{\partial r}\right\rangle\bigg|_{w} + \frac{r'(0)}{r(0)}\left\langle \overline\nabla_{\rho(X_i)}\rho(X_j), \frac{\partial}{\partial \alpha}\right\rangle\bigg|_{w},\\
    h_{ijk}\big|_w &= \left\langle A^{\mathcal{O}_w}\left(\rho(X_i),\rho(X_j)\right), \rho(X_k)  \right\rangle\big|_w \leq |A^{\mathcal{O}_w}|.
\end{align*}

To calculate $h_{ijn}$: Note that if $Y = r\frac{\partial}{\partial r}$ or $\frac{\partial}{\partial \alpha}$, then $\langle \rho(X_i),Y \rangle = 0$, $\langle \rho(X_j),Y \rangle = 0$, and $\langle [\rho(X_i),\rho(X_j)],Y\rangle \rangle = \langle \rho[X_i,X_j],Y\rangle = 0$. Therefore by the Koszul formula, 
\begin{align*}
\langle \overline \nabla_{\rho(X_i)} {\rho(X_j)}, Y\rangle &= \frac{1}{2} \Big(\rho(X_i)\langle \rho(X_j),Y\rangle + \rho(X_j)\langle \rho(X_i),Y\rangle - Y\langle \rho(X_i),\rho(X_j)\rangle\\
&\quad\quad- \langle Y, [\rho(X_j),\rho(X_i)]\rangle - \langle \rho(X_j), [\rho(X_i),Y]\rangle - \langle \rho(X_i), [\rho(X_j),Y]\rangle\Big) \\
&= \frac{1}{2}  \Big(-Y\langle \rho(X_i),\rho(X_j) \rangle+ \langle \rho(X_i), [Y,\rho(X_j)]\rangle + \langle \rho(X_j), [Y,\rho(X_i)]\rangle\Big).
\end{align*}
Considering the group action $G \times \mathbb{C}^* \circlearrowright \mathbb{C}^n$ with Lie algebra $\mathfrak{g}\oplus\text{Lie}(\mathbb{C}^*)$ and infinitesimal action $\rho$, we note there exist $R,A \in \text{Lie}(\mathbb{C}^*)$ such that $\rho(R) = r\frac{\partial}{\partial r}$, $\rho(A) = \frac{\partial}{\partial \alpha}$. Since $\rho$ is a Lie algebra homomorphism, and $R,A$ are in the centre of the Lie algebra $\mathfrak{g}\oplus\text{Lie}(\mathbb{C}^*)$:
\begin{align*}
\left[\rho(X_i),r\frac{\partial}{\partial r}\right] &= [\rho(X_i),\rho(R)] = \rho([X_i,R]) \,=\, 0, \quad
\left[\rho(X_i),\frac{\partial}{\partial \alpha}\right] \,=\, 0.
\end{align*}
Furthermore, a calculation yields:
\begin{equation}
    Y \langle \rho(X_i), \rho(X_j) \rangle\big|_w = 
    \begin{cases}
        \displaystyle 2\delta_{ij} & \mbox{ if } Y = r\frac{\partial}{\partial r}\\
        0 & \mbox{ if } Y = \frac{\partial}{\partial \alpha}.
    \end{cases}
\end{equation}
Putting this together, 
\[ h_{ijn}|_w = \delta_{ij}\alpha'(0) = -\delta_{ij}\frac{\langle \gamma(0), N(0)\rangle}{|r(0)|^2} =  -\delta_{ij}p(0).\]

To estimate $h_{ijk}$: Choose $q \in M\cap S^{2n-1}$, and define $C':= |A^{\mathcal{O}_q}_q|$ (note that by symmetry, this does not depend on the choice of $q$). By Proposition \ref{prop-mu0cohom1}, there exists $g \in G$, $\beta \in [0,2\pi)$ such that
\begin{align*}
    r(0)e^{i\beta}\cdot g\cdot q = w &\implies |A^{\mathcal{O}_w}_w| = \frac{1}{r(0)^2} |A^{\mathcal{O}_q}_q| = \frac{C'}{r(0)^2}\\
    &\implies h_{ijk}\big|_w \leq \frac{C'}{r(0)^2}.
\end{align*}

By Proposition \ref{prop-flowlevelset}, $\vec{H}_w \in T_wM \cap T^\perp_wL$, a 1-dimensional space spanned by $N = Je_n$. Therefore, the mean curvature of $L$ at $w = r_0e^{i\alpha_0}$ is given by:
\begin{align*}
    \vec{H}\big|_w \,&=\, g^{ij}h_{ijk}Je_n\big|_w \,=\, \sum_{i=1}^{n-1} h_{iin}Je_n\big|_w \,=\,  \vec{k}(0) - (n-1)\frac{\gamma(0)^\perp}{|\gamma(0)|^2},\\
    |\vec{H}|^2\big|_w &= (k(0) + (n-1)p(0))^2.
\end{align*}
Using the inequality $p(0) \leq \frac{1}{r(0)}$, the norm of the second fundamental form is given by
\begin{align*}
    |A|^2\big|_w &= g^{ia}g^{jb}g^{kc}h_{ijk}h_{abc}\big|_w\,=\, \sum_{i,j,k=1}^n h_{ijk}h_{ijk} \\
    &= h_{nnn}h_{nnn} + 3\sum_{i=1}^{n-1} h_{iin}h_{iin} + \sum_{i,j,k = 1}^{n-1} h_{ijk}h_{ijk}\\
    &\leq k(0)^2 + 3(n-1)p(0)^2 + C(n,G,M)r(0)^{-2}.
\end{align*}
Finally,
\begin{align*}
    |a|^2\big|_w &= g^{ij}g^{kl}a_{ik}a_{jl}\big|_w \, = \, \sum_{i,j = 1}^{n-1}a_{ij}^2 + 2\sum_{i=1}^{n-1}a_{in}^2 + a_{nn}^2 = \sum_{i,j=1}^{n-1} (\alpha_n h_{ijn})^2 + (\alpha_n h_{nnn})^2\\
   &= \left( \sum_{i=1}^{n-1}(h_{iin})^2 + (h_{nnn})^2\right)|\vec H|^2\big|_w = \left(k(0)^2 + (n-1)p(0)^2\right)|\vec H|^2\big|_w.
\end{align*}
\end{proof}

\begin{corollary}[Equivariant LMCF Equation]\label{cor-equivariantlmcf} Let $M$ be a connected component of $M_0 := \mu^{-1}(0) \cap \C^n_{(H)}$ and $z \in M$. Let $F_t:\widehat L\rightarrow M$ be a family of connected cohomogeneity-one $G$-invariant immersed Lagrangian submanifolds, and define the profile curves $f_t:\widehat l\rightarrow P_z$ as in Proposition \ref{prop-profilebijection}.

Then $F_t$ is a mean curvature flow if and only if $f_t$ is a solution to the flow of curves in $P_{\overline z}$ given by
\begin{align}
 \frac{\partial f_t}{\partial t}^\perp & = \vec{k} - (n-1)\frac{f_t^\perp}{|f_t|^2}.\label{eqn-equivariantLMCF}
\end{align}
\end{corollary}

Finally, we prove the following important formula for the Lagrangian angle of a $G$-invariant Lagrangian submanifold.

\begin{lemma}[Lagrangian Angle of a $G$-Invariant Lagrangian in $\mu^{-1}(0)$]\label{lem-laganglelemma}
 Let $M$ be a connected component of $M_0 := \mu^{-1}(0) \cap \C^n_{(H)}$, and let $z \in M$. There exists a unitary isomorphism $\Phi_z \colon P_z \rightarrow \mathbb{C}$ such that the following is true.
 
 Let $L \subset M$ be a connected, immersed, cohomogeneity-one graded Lagrangian submanifold with Lagrangian angle $\theta$, and let $\gamma \colon \mathbb{R} \rightarrow l := L \cap P_z$ be a unit-speed parametrised component of the profile curve. Then 
\[ \theta \equiv \emph{arg}\!\left((\Phi_z\circ \gamma)'\right) + (n-1)\emph{arg}(\Phi_z \circ \gamma) \,\quad(\emph{mod } \pi). \]
\end{lemma}

\begin{proof}
Take $\pm \nu$, $\widehat \beta$ as in the proofs of Lemma \ref{lem-laganglegeneral} and Proposition \ref{prop-mu0cohom1}(c). Since $L_{e^{i\phi}}^* \Omega = e^{in\phi}\Omega$, it is possible to choose a unit vector $V \in P_z$ such that $\Omega_{V}(\pm\nu, V) = \pm 1.$ Define $\Phi_z \colon P_z \to \C$ to be the unique unitary isomorphism with $\Phi_z(V) = 1$. Then, taking arguments relative to $V$ and writing $\gamma(s) = r(s)e^{i \alpha(s)} V$,
\begin{align*}
    \pm e^{i\theta(s)} \, &= \, \Omega_{\gamma(s)}(\pm \nu, \gamma'(s)) \, = \, \beta_{\gamma(s)}(V) \cdot e^{i\,\text{arg}(\gamma'(s))} \\
    &= \Omega_{r(s)e^{i\alpha(s)}\cdot V} \left( \pm \nu_{r(s)e^{i\alpha(s)}\cdot V}, V\right)\cdot e^{i\,\text{arg}(\gamma'(s))}\\
    &= (L_{e^{i\alpha(s)}})^*\Omega_{r(s)\cdot V} \left(\pm \nu_{r(s)\cdot V}, e^{-i\alpha(s)}V\right)\cdot e^{i\,\text{arg}(\gamma'(s))}\\
    &= \pm e^{(n-1)\alpha(s)}\cdot e^{i\,\text{arg}(\gamma'(s))},
    \end{align*}
    so
    \begin{align*}
\theta(s) &\equiv \text{arg}(\gamma'(s)) + (n-1)\text{arg}(\gamma(s)) \quad (\text{mod } \pi)\\
    &\equiv \text{arg}\!\left(\Phi_z( \gamma(s))'\right) + (n-1)\text{arg}\left(\Phi_z(\gamma(s))\right) \,\quad(\text{mod } \pi).
\end{align*}
\end{proof}

In the almost-calibrated case, this formula for the Lagrangian angle implies that each connected component of the profile curve $l := L \cap P_z$ lies in a wedge.

\begin{lemma}[Wedge Lemma]\label{lem-wedge}
Let $M$ be a connected component of $M_0 := \mu^{-1}(0) \cap \C^n_{(H)}$. Let $L \subset M$ be a connected, almost-calibrated, cohomogeneity-one $G$-invariant Lagrangian submanifold of type $(H)$. Let $z \in L$, and let $l := L \cap P_z$ be the profile curve.

Then $l$ consists of $m$ connected components $\{l_1,\ldots,l_m\}$, such that $l = C_m \cdot l_i$ for each $i \in \{1,\ldots,m\}$. Moreover, each connected component $l_i$ is contained in a wedge $W \subset P$ bounded by two half-lines spanning an angle $\alpha < \frac{2\pi}{n}$.
\end{lemma}
\begin{proof}
   The formula for the Lagrangian angle given by Lemma \ref{lem-laganglelemma} is independent of $G, H$ and $M$. Therefore, the proof of \cite[Lem. 4.6]{Wood2019} is valid for the general cohomogeneity-one case, and it follows as in that Lemma that each connected component $l_i$ of $l$ is contained in a wedge spanning an angle $\alpha < \frac{2\pi}{n}$.
   
   To show that there are $m$ connected components, first note that the composition of the maps $\iota: P_z \to M$, $\pi:M \rightarrow Q := M/G$ is a smooth $m$-fold covering map. The profile curve may be expressed as $l = (\pi\circ\iota)^{-1}(L/G)$. Taking a point $q \in L/G \subset Q$, there are $m$ lifts $\{p_1,\ldots,p_m\} \subset P_z$ of $q$, related by $e^{\frac{2\pi i}{m}} \cdot p_1 = p_k$. Denoting by $l_1$ the unique lift of $L/G$ to $P_z$, there exists a wedge $W_1$ spanning an angle $\alpha < \frac{2\pi}{n}$ such that $l_1 \subset W_1$. Then, defining $W_k := e^{{2\pi i k}{n}} \cdot W_1$, $l_k := e^{{2\pi i k}{n}} \cdot l_1$, it follows that $l_k$ is the unique lift of $L/G$ to $P_z$ containing $p_k$, and $l = \cup_{k=1}^n l_k$. Finally, since $\alpha < \frac{2\pi}{n}$, the wedges $W_k$ are pairwise disjoint, and so $l_k$ are distinct connected components of $l$.
\end{proof}


\section{Soliton Solutions}\label{sec-5}

\indent \indent We now classify the cohomogeneity-one soliton solutions of LMCF, in particular special Lagrangians, shrinking and expanding solitons, and translating solitons. 

As before, we fix a compact connected Lie group $G \leq \SU(n)$ and an isotropy type $(H)$ with $(n-1)$-dimensional orbits $G/H$ in $\C^n$.  We will typically consider Lagrangians in $\mu^{-1}(0) \subset \C^n$, for which we fix a connected component $M$ of the $(n+1)$-dimensional coisotropic submanifold $M_0 := \mu^{-1}(0) \cap \mathbb{C}^n_{(H)}$ and $z \in M$. We denote by $P_z$ the complex line through $z$; by the isomorphism of Lemma \ref{lem-laganglelemma} we will often conflate the spaces $P_z$ and $\mathbb{C}$ for notational convenience.  Finally, $m \mid 2n$ will denote the order of the cyclic group $C_m \cong \widetilde{H}/H$ of Propositions \ref{prop-profilesymmetry} and \ref{prop-mu0cohom1}.


\subsection{Special Lagrangians}

\indent \indent The local existence of $G$-invariant special Lagrangians was demonstrated by Joyce \cite[Thm. 4.5]{Joyce2001b} who showed that given any isotropic $G$-orbit $\mathcal{O}$, there exists a $G$-invariant special Lagrangian submanifold containing $\mathcal{O}$. In general, these special Lagrangians do not admit a simple closed form. However, in the case where $L$ lies in $M \subset \mu^{-1}(0)$, we are able to explicitly classify all $G$-invariant special Lagrangians, utilising the $G$-independent formulation of the mean curvature of Proposition \ref{prop-curvature}.

\begin{theorem}[Uniqueness of cohomogeneity-one special Lagrangian cones]\label{thm-splagcones} For $k \in \mathbb{Z}$, define the parametrised curve
\[ \widetilde c_{k,\overline\theta}: (0,\infty)\rightarrow \mathbb{C}, \quad \widetilde c_{k,\overline\theta}(r) = re^{i\left(\tfrac{\overline\theta}{n} + \tfrac{k\pi}{n}\right)}.\]
Then for each $\overline\theta \in \mathbb{R}$, the correspondence of Proposition \ref{prop-profilebijection} provides a bijection between:
\begin{itemize}
    \item $G$-invariant connected special Lagrangian cones $\widetilde C\subset M$ with Lagrangian angle $\overline{\theta}$ and connected link; and
    \item $C_m$-orbits of curves $\widetilde c_{k,\overline\theta}$ for $k \in \{1,\ldots, \tfrac{2n}{m}\}$. 
\end{itemize}
\end{theorem}

\begin{proof}
Let $L \subset M$ be a connected special Lagrangian cone of angle $\overline\theta$. By Proposition \ref{prop-profilebijection}, $L$ corresponds to a $C_m$-invariant curve $l \subset P_z$. Since $L$ is connected and $\mathbb{R}^+$-invariant, it must be the $C_m$-orbit of a single ray, $\widetilde c(r) := re^{i\overline\alpha}$ for some $\overline\alpha \in (\frac{\overline\theta}{n},\frac{\overline\theta}{n}+\frac{2\pi}{m}]$. By Lemma \ref{lem-laganglelemma},
\begin{align*}
    \overline\theta \, \equiv\, \arg(\widetilde c') + (n-1)\arg(\widetilde c) \, \equiv \, n\overline\alpha \,\, (\text{mod } \pi),
\end{align*}
and therefore $\overline\alpha = \frac{\overline\theta+k\pi}{n}$ for some unique $k \in \{1,\ldots, \frac{2n}{m}\}$.
\end{proof}

\begin{theorem}[Uniqueness of cohomogeneity-one special Lagrangians]\label{thm-splags} For $k \in \mathbb{Z}$, $B > 0$, define the parametrised curve
\begin{align*}
\widetilde l_{B,k,\overline{\theta}}&: \left[-\tfrac{\pi}{2n},\tfrac{\pi}{2n}\right]\rightarrow \mathbb{C},\quad
    \widetilde l_{B,k,\overline{\theta}}(\alpha) := \frac{B}{\sqrt[n]{\cos\left( n\alpha\right)}}e^{i\left(\alpha+\tfrac{\overline\theta}{n} - \tfrac{\pi}{2n} + \tfrac{k\pi}{n}\right)}.
\end{align*}
Then for each $\overline\theta \in \mathbb{R}$, the correspondence of Proposition \ref{prop-profilebijection} provides a bijection between:
\begin{itemize}
    \item $G$-invariant complete connected special Lagrangians $\widetilde L\subset \mathbb{C}^n$ contained in $M$, with Lagrangian angle $\overline{\theta}$; and
    \item $C_m$-orbits of curves $\widetilde l_{B,k,\overline\theta}$ for $k \in \{1,\ldots, \tfrac{2n}{m}\}$ and $B>0$. 
\end{itemize}
\end{theorem}

\begin{proof}
Let $L$ be a complete connected smooth special Lagrangian of angle $\overline\theta$ contained in $M$, with profile curve $l$. By the $C_m$-invariance of $l$, we may choose a connected component $\gamma \subset l$ and a point $w = Be^{i\overline\alpha} \in \gamma$ minimising the distance to the origin such that $\overline\alpha \in (\frac{\overline\theta}{n} - \frac{\pi}{2n},\frac{\overline\theta}{n} - \frac{\pi}{2n} + \frac{2\pi}{m}]$. By Lemma \ref{lem-laganglelemma}, at the point $w$,
\begin{align*}
    \overline\theta \,\equiv\, \arg(\gamma') + (n-1)\arg(\gamma)  \,\,(\text{mod } \pi),
\end{align*}
and therefore $\overline\alpha \,= \, \frac{\overline\theta}{n} - \frac{\pi}{2n} + \frac{k\pi}{n}$ for some unique $k \in \{ 1, \ldots, \frac{2n}{m}\}$.

Locally near $w$ we may parametrise $\gamma$ as $\gamma(\alpha) = r(\alpha)e^{i(\alpha + \overline\alpha)}$, so that in this region (using Lemma \ref{lem-laganglelemma}),
\begin{align*}
    \gamma'(\alpha) = (r'(\alpha) + ir(\alpha))e^{i(\alpha+\overline\alpha)} &\implies \overline\theta \equiv \tan^{-1}\left(\tfrac{r(\alpha)}{r'(\alpha)}\right) + n(\alpha+\overline\alpha) \,\, (\text{mod } \pi)\\
    &\implies \frac{r'(\alpha)}{r(\alpha)} \, = \, \tan(n\alpha).
\end{align*}
Integrating this equation gives the unique complete solution given in the statement.
\end{proof}

\begin{remark}
The asymptotes of $\widetilde l_{B,k,\overline\theta}$ are given by $\widetilde c_{k-1,\overline \theta} \cup \widetilde c_{k,\overline \theta}$.  In particular, the special Lagrangians of Theorem \ref{thm-splags} are all asymptotically conical.
\end{remark}

\begin{remark} Fix a wedge $W \subset P_z$ spanning an angle $\alpha < \frac{2\pi}{n}$ and fix Lagrangian angle $\overline \theta$.  Then, up to scaling, there is at most one $G$-invariant connected smooth special Lagrangian $\widetilde L \subset M$ with profile curve $\widetilde l$ contained in $C_m \cdot W$ and Lagrangian angle $\overline \theta$. Further, there are at most two $G$-invariant special Lagrangian cones $\widetilde C \subset M$ with connected link, profile curve $\widetilde c$ contained in $C_m \cdot W$, and Lagrangian angle $\overline\theta$.
\end{remark}


\begin{remark} \label{remark-LawlorNeck} For the diagonal $\SO(n)$-action on $\C^n$ of Examples \ref{ex-SO(n)} and \ref{ex-CyclicGroup}(a), the cohomo-geneity-one special Lagrangians of Theorem \ref{thm-splags} were first constructed by Harvey and Lawson \cite[Thm.\ 3.5]{Harvey1982}, who took $\overline{\theta} = 0$ and $k = 0$.  In fact, these belong to the larger class of \emph{Lawlor necks}, discovered by Lawlor \cite{Lawlor1989} and extended by Harvey \cite[pg.\ 149-150]{Harvey1990} and Joyce \cite{Joyce2001}.  The corresponding $\SO(n)$-invariant special Lagrangian cones of Theorem \ref{thm-splagcones} are simply $n$-planes.
\end{remark}

\begin{remark} \label{remark-TorusCone}
For the $T^{n-1}$-action on $\C^n$ of Examples \ref{ex-Torus} and \ref{ex-CyclicGroup}(b), the cohomogeneity-one special Lagrangians of Theorem \ref{thm-splags} are examples of a larger family first discovered by Harvey and Lawson \cite[Thm.\ 3.1]{Harvey1982}.  The corresponding $T^{n-1}$-invariant special Lagrangian cones of Theorem \ref{thm-splagcones} are known as the \emph{Harvey-Lawson $T^{n-1}$ cones}.
\end{remark}

\subsection{Shrinkers and Expanders}

\indent \indent We now classify the cohomogeneity-one shrinkers and expanders of Lagrangian mean curvature flow --- i.e.\ , Lagrangian immersions $F\colon L\to \mathbb{C}^n$ satisfying the elliptic equation (\ref{eq-shrinker}), where the case of $\lambda > 0$ corresponds to a self-similarly shrinking solution, and the case of $\lambda < 0$ corresponds to a self-similarly expanding solution. Observing that the corresponding mean curvature flows sweep out an $\mathbb{R}^+$-invariant set, and that the only level set of $\mu$ containing an $\mathbb{R}^+$-invariant set is $\mu^{-1}(0)$, it follows that any solution to (\ref{eq-shrinker}) must lie in $\mu^{-1}(0)$, and we may therefore restrict to this case. \\
\indent Using Proposition \ref{prop-curvature} and the decomposition (\ref{eq-orthdecompositoncohom1}), we see that for a $G$-invariant Lagrangian $L \subset \C^n$ satisfying (\ref{eq-shrinker}), its profile curve $l := L \cap P_z$ satisfies the equation
\begin{equation} \vec{k} + \left( \lambda - \frac{n-1}{|l|^2}\right)l^\perp = 0. \label{eq-shrinkerprofile} \end{equation}
Conversely, any curve in $\mathbb{C}$ satisfying (\ref{eq-shrinkerprofile}) corresponds by the bijection of Proposition \ref{prop-profilebijection} to a solution of (\ref{eq-shrinker}).

In the particular case of $G = \SO(n)$ acting diagonally on $\C^n = \R^n \oplus \R^n$ (recall Example \ref{ex-SO(n)}), solutions of (\ref{eq-shrinker},\ref{eq-shrinkerprofile}) were classified in a collection of papers by Anciaux, Castro and Romon \cite{Anciaux2006,AnciauxCastroRomon2006,AnciauxRomon2009}; these solutions are now referred to as the Anciaux shrinkers and expanders. The following theorem summarises the results:

\begin{theorem}[Anciaux Shrinkers and Expanders in $\mathbb{C}^n$] \label{thm-anciaux} Consider the diagonal $\SO(n)$-action on $\mathbb{C}^n$, let $M := \mu^{-1}(0) \setminus \{0\}$ and fix $w \in M$. \\
    \indent (a) For $\lambda > 0$, there is a countable family of $\SO(n)$-invariant shrinking solutions of (\ref{eq-shrinker}): \[\left\{L^{(p,q)}\, : \, p,q \text{ relatively prime, } \tfrac{p}{q} \in \left(\tfrac{1}{2n}, \tfrac{1}{\sqrt{2n}}\right)\right\},\]  where $p$ is the winding number of the profile curve $l^{(p,q)} := L^{(p,q)} \cap P_w$, and $q$ is the number of maxima of its curvature. In the case of $n=1$, these curves are the Abresch-Langer solutions of self-similarly shrinking curve shortening flow \emph{\cite{Abresch1986}}.\\
    \indent (b) For $\lambda < 0$, there is a one-parameter family of expanding solutions of (\ref{eq-shrinker}): \[ \left\{L^\alpha \,:\, \alpha\in (0,\tfrac{\pi}{n})\right\},\] with the ends of the profile curve $l^\alpha := L^{\alpha} \cap P_w$ asymptotic to two lines spanning an angle $\alpha$.\\
    \indent (c) Up to the $
    \U(1)_{\Delta}$-action on $\mathbb{C}^n$, these are the only complete $\SO(n)$-invariant solutions to (\ref{eq-shrinker}) in $M \subset \mathbb{C}^n$.
\end{theorem}

\indent In fact, these examples may be used to describe all connected cohomogeneity-one Lagrangian shrinkers and expanders. Indeed, since equation (\ref{eq-shrinkerprofile}) is independent of the subgroup $G \leq \SU(n)$, we may use the bijection of Proposition \ref{prop-profilebijection} to identify $G$-invariant solutions in $M$ with $\SO(n)$-invariant solutions in $\mathbb{C}^n\setminus\{0\}$, and thereby upgrade Theorem \ref{thm-anciaux} to the general cohomogeneity-one case.  More precisely:

\begin{proposition}\label{prop-shrinkerbijection}
    There is a bijection between the following three classes:
    \begin{itemize}
        \item[1.] Connected immersed curves $\gamma \colon l \to \mathbb{C}\setminus\{0\}$ satisfying $\vec k + \left(\lambda - \frac{n-1}{|\gamma|^2}\right)\gamma^\perp = 0$, up to $\U(1)$-rotation;
        \item[2.] $\SO(n)$-invariant immersed Lagrangians $F \colon L\to\mathbb{C}^n \setminus \{0\}$ satisfying $\vec{H} + \lambda F^\perp = 0$, up to $\U(1)_\Delta$-rotation;
        \item[3.] $G$-invariant immersed Lagrangians $F \colon L\to M$ satisfying $\vec{H} + \lambda F^\perp = 0$, up to $\U(1)_\Delta$-rotation.
    \end{itemize}
\end{proposition}
\begin{proof}
    The bijection between classes $1$ and $3$ is given by Proposition \ref{prop-profilebijection}, albeit choosing a single connected component of the profile curve. Note that two connected immersed curves are related by a $\U(1)$-rotation if and only if the corresponding $G$-invariant Lagrangian is related by a $\U(1)_\Delta$-rotation.
    
    The bijection between classes $2$ and $3$ is given in the same way, choosing $G = \SO(n)$ and $M$ to be the unique connected component of $M_0 := \mu^{-1}(0) \cap \mathbb{C}^n_{(\SO(n-1))}$, and noting that any $\SO(n)$-invariant Lagrangian $L \subset \mathbb{C}^n\setminus\{0\}$ must lie in $M$.
\end{proof}

\begin{theorem}[Classification of cohomogeneity-one expanding and shrinking solitons] \label{thm-shrinkers}
Let $G \leq \SU(n)$ be compact and connected, and $H \leq G$ a subgroup such that $\dim(G / H) = n-1$. Any complete connected $G$-invariant type $(H)$ immersed Lagrangian expanding or shrinking soliton in $\mathbb{C}^n$ corresponds to one of the examples of Theorem \ref{thm-anciaux} via the bijection of Proposition \ref{prop-shrinkerbijection}.
\end{theorem}

Note that although the profile curves of these solutions are the same as those of Theorem \ref{thm-anciaux}, the corresponding Lagrangian submanifolds satisfying (\ref{eq-shrinker}) will be different for different choices of $G \leq \SU(n)$, and so are in general distinct from the Anciaux examples.

\subsection{Translators}

\indent \indent Finally, we consider translating solitons of Lagrangian mean curvature flow --- i.e.\  solutions of (\ref{eq-translator}). In the case of curve-shortening flow ($n=1$), the grim reaper translating soliton $\gamma(x) := \left(x,\log(\cos(x))\right)$ is the unique translating soliton up to scaling and rigid motions, and may be considered a trivial example of a cohomogeneity-one Lagrangian MCF with $G = \{e\}$. For $n >1$, however, there are no cohomogeneity-one examples:

\begin{theorem}\label{thm-translators}
Suppose $n > 1$.  For any compact connected $G \leq \SU(n)$, there are no cohomogeneity-one $G$-invariant immersed translating solitons in $\C^n$.
\end{theorem}
\begin{proof}
Assume for a contradiction that $L_t$ is a solution to (\ref{eq-translator}) with translation vector $V$. Since the $G$-action preserves $L_t$ for all $t$, it follows that $G$ stabilises $V$, and hence $G$ stabilises the complex line $E := \text{span}_{\R}(V, JV)$. Therefore,
\[ G \leq \SU(E^\perp) \cong \SU(n-1).\]
It follows that for any $z \in L_t$, the $G$-orbit $\mathcal{O}_z$ lies in the affine complex hyperplane $\pi_{E}(z) + E^{\perp} \cong \mathbb{C}^{n-1}$. By these identifications, the $(n-1)$-dimensional submanifold $\mathcal{O}_z$ may be viewed as a homogeneous Lagrangian submanifold of $\mathbb{C}^{n-1}$.  Since $G \leq \SU(n-1)$, the Lagrangian angle of $\mathcal{O}_z$ is preserved by the $G$-action, so $\mathcal{O}_z \subset \C^{n-1}$ is a special Lagrangian.  However, $\mathcal{O}_z$ is compact, and there are no compact special Lagrangian submanifolds in $\mathbb{C}^{n-1}$ for $n > 1$.
\end{proof}

Although Theorem \ref{thm-translators} rules out cohomogeneity-one translators for compact $G \leq \SU(n)$, there exist other examples if we allow $G \leq \SU(n) \ltimes \C^n$ to be noncompact.  For example, the product of a grim reaper curve with a real line yields a cohomogeneity-one translator in $\mathbb{C}^2$.

\section{Analysis of Singularities}\label{sec-6}

\indent \indent In this section, we study singularity formation for cohomogeneity-one LMCF in the zero level set. In particular, we prove that any singular point of the flow must be the origin (Theorem \ref{thm-locationofsingularities}), classify and prove uniqueness of Type I and Type II blowups (Theorems \ref{thm-typeiblowup} and \ref{thm-typeiiblowup}), and prove that every cohomogeneity-one special Lagrangian in $\mu^{-1}(0)$ occurs as the blowup of a mean curvature flow (Theorem \ref{thm-existenceofsingularities}). The theorems of \S 6 were established for the diagonal action $\SO(n)\circlearrowright \mathbb{C}^n$ of Example \ref{ex-SO(n)} in \cite{Wood2019}.

\indent We maintain the setup of the previous section.  That is, we fix a compact connected Lie group $G \leq \SU(n)$, which acts linearly on $\mathbb{C}^n$, and an isotropy type $(H)$ for which the $G$-orbits are $(n-1)$-dimensional. 
We also fix a connected component $M$ of the $(n+1)$-dimensional coisotropic $M_0 = \mu^{-1}(0) \cap \mathbb{C}^n_{(H)}$, and a point  $z_0 \in M$, letting $P := P_{z_0}$ denote the complex line through $z_0$.  We often identify $P$ with $\mathbb{C}$ via the unitary isomorphism $\Phi \colon P \to \mathbb{C}$ of Lemma \ref{lem-laganglelemma}. For any $G$-invariant Lagrangian $L \subset M$, we denote its profile curve by $l := L \cap P$. Note that $L \subset M$ implies that $L$ is of type $(H)$. By a slight abuse of notation, we will view $l$ as a curve in $\mathbb{C}$ by the above identification, so that quantities like $\text{arg}(l)$ make sense.

\subsection{Curvature Estimates}\label{sec-6.1}

\indent\indent In order to simplify the proofs of this section, we first derive curvature bounds for the equivariant flow. The following argument is due to Lambert, and was conveyed to the authors by private correspondence.

\begin{theorem}[Curvature Estimates]\label{thm-equicurvestimates}
	Let $F: L \times [0,T) \rightarrow M$ be a connected $G$-invariant Lagrangian MCF that is almost-calibrated --- i.e.\ , assume that the Lagrangian angle $\theta$ satisfies $\theta \in (\overline\theta-\frac{\pi}{2} +\varepsilon_\theta, \overline\theta + \frac{\pi}{2} - \varepsilon_\theta)$ for some $\varepsilon_\theta > 0$, $\overline\theta>0$. Assume that the mean curvature $\vec H$ is bounded on $L_0$. Denote by $r:\mathbb{C}^n\rightarrow \mathbb{R}$ the distance from the origin.
	
	Then there exists $C = C(n,G,M,\varepsilon_\theta, L_0)$ (independent of $t$) such that
	\[ \big|\vec H\big|^2, |A|^2 \, \leq \, C\left( 1 + \frac{1}{r^2} \right). \]
\end{theorem}

\begin{proof}
	Throughout, we allow the constant $C$ to change from line to line.
	We aim to estimate the function $\Gamma \colon L \times [0,T) \rightarrow \mathbb{R}^+$, $\Gamma := \big|\vec H\big|^2f(\theta)\psi$, where $f \colon (-\frac{\pi}{2} +\varepsilon_\theta, \frac{\pi}{2} - \varepsilon_\theta) \rightarrow \mathbb{R}^+$ is bounded, and $\psi$ will be a cutoff function defined on the ambient space. Using this estimate, we will prove that there exists $ C(n,G,M,\varepsilon_\theta, L_0)$ such that 
	\begin{equation} \label{eqn-gamma} 
	\big|\vec H\big|^2 \leq C\left( 1 + \frac{1}{r^2} \right).
	\end{equation}
	Then, defining $\vec p := \frac{l^\perp}{r^2}$, $p = |\vec p|$, $k = |\vec k|$ and using Proposition \ref{prop-curvature}:
	\begin{align*} 
	p &\leq \frac{1}{r},\quad	k^2 \leq \left( \big|\vec H\big| + \frac{(n-1)}{r}\right)^2,
\end{align*}
so
\begin{align*}
	|A|^2 \, &\leq \, k^2 + C(n,G,M)r^{-2}\\
	&\leq \left(\big|\vec H\big| + \frac{C(n)}{r}\right)^2 + \frac{C(n,G,M)}{r^2}\\
	&\leq 2\big|\vec H\big|^2 + \frac{C(n,G,M)}{r^2},
	\end{align*}
    therefore (\ref{eqn-gamma}) implies the theorem. \\
    
    \indent To begin, let $\Delta = \Delta^{L_t}$ denote the Laplacian on $L$ given by the metric induced from $F_t$, and define the symmetric $2$-tensor $a_{ij} := g^{kl}\alpha_lh_{ijk} = \langle \vec H, A(e_i,e_j)\rangle$. Using \cite{Smoczyk2000} Proposition 1.8.6 and Proposition \ref{prop-curvature}, we calculate the following evolution equations:
	\begin{align*}
	\left(\frac{\partial}{\partial t} - \Delta\right)\big|\vec H\big|^2 \, &= \, 2|a|^2 - 2|\nabla H|^2 \\
	&\leq 2(k^2 + (n-1)p^2)\big|\vec H\big|^2 - 2|\nabla\big|\vec H\big||^2\\
	&\leq 2\big|\vec H\big|^4 + \frac{C(n)}{r^2}\big|\vec H\big|^2 - 2|\nabla\big|\vec H\big||^2,\\
	\left(\frac{\partial}{\partial t} - \Delta\right)\theta \, &= \, 0,\\
	\left(\frac{\partial}{\partial t} - \Delta \right)\psi \, &= \,  - \Delta^{\mathbb{R}^n}\psi + \text{tr}^\perp(\mbox{Hess}(\psi)),
	\end{align*}
	and the following Laplacians:
	\begin{align*}
	\log(\Gamma) &= \log(\big|\vec H\big|^2) + \log(f) + \log(\psi)\\
	\Delta\log(\big|\vec H\big|^2) \, &= \, \frac{\Delta\big|\vec H\big|^2}{\big|\vec H\big|^2} - \frac{|\nabla \big|\vec H\big|^2|^2}{\big|\vec H\big|^4}\\
	\Delta\log(f) \,&=\, \frac{f'\Delta\theta}{f} + \left(\frac{f''}{f} - \frac{(f')^2}{f^2}\right)|\nabla \theta|^2.
	\end{align*}
	
	We first bound $\Gamma$ at an increasing maximum --- i.e.\  a spacetime point $(y,s)\in L\times [0,T)$ such that  $\Gamma(y,s) = \max_{x \in L}\Gamma(x,s)$ and $\frac{\partial \Gamma}{\partial t}(y,s) \geq 0$. At such a point, $\nabla\Gamma = 0$ and $( \frac{\partial}{\partial t} - \Delta)\Gamma \geq 0$. Noting that such a point is also an increasing maximum of $\log \Gamma$, it follows that
	\begin{align*}
	&\nabla\log(\big|\vec H\big|^2) + \nabla\log f + \nabla\log\psi = 0,
	\end{align*}
whence for any $\varepsilon >0$,
	\begin{align*}
	\left|\nabla\log(\big|\vec H\big|^2)\right|^2 &= \left| \frac{f'}{f}\nabla\theta + \frac{\nabla\psi}{\psi} \right|^2 \, \leq \, (1+ \varepsilon)\frac{(f')^2}{f^2}\big|\vec H\big|^2 + (1 + \tfrac{1}{\varepsilon})\frac{|\nabla \psi|^2}{\psi^2}.
	\end{align*} 
	It follows that:
	\begin{align}
	0 &\leq \left(\frac{\partial}{\partial t} - \Delta\right)\log(\Gamma) \notag\\
	&= \, \frac{1}{\big|\vec H\big|^2}\left(\frac{\partial}{\partial t} - \Delta\right)\big|\vec H\big|^2 + \frac{f'}{f}\left(\frac{\partial}{\partial t} - \Delta\right)\theta +\frac{1}{\psi}\left(\frac{\partial}{\partial t} - \Delta\right)\psi \notag\\
	& \quad \quad \quad + |\nabla \log(\big|\vec H\big|^2)|^2 - \left(\frac{f''}{f} - \frac{(f')^2}{f^2}\right)\big|\vec H\big|^2 + \frac{|\nabla \psi|^2}{\psi^2} \notag\\
	&\leq \left( 2 + \frac{f''}{f} - \frac{(f')^2}{f^2} \right) \big|\vec H\big|^2 + \frac{C(n)}{r^2} + \frac{1}{2}|\nabla\log(\big|\vec H\big|^2)|^2 \notag\\
	 & \quad \quad \quad - \frac{\Delta^{\mathbb{R}^n} \psi - \mbox{tr}^\perp \mbox{Hess}(\psi)}{\psi} + \frac{|\nabla \psi|^2}{\psi^2} \notag\\
	&\leq \left[ 2+(1+ \tfrac{1}{2}(1+\varepsilon))\frac{(f')^2}{f^2} - \frac{f''}{f}\right]\big|\vec H\big|^2 \notag\\
	& \quad \quad \quad + C(n,\varepsilon) \left( -\frac{\Delta^{\mathbb{R}^n} \psi + \mbox{tr}^\perp \mbox{Hess}(\psi)}{\psi} + \frac{|\nabla \psi|^2}{\psi^2} + \frac{1}{r^2} \right).\label{ineq-halfway}
	\end{align}
	
	We now choose our $f$ to simplify this inequality. Writing $\Psi := f^{-\frac{1}{2}(1+\varepsilon)}$, the square bracket is equal to
	\[2 + \frac{\Psi''}{b\Psi},\]
	where $b = \frac{1}{2}(1+\varepsilon)$.  Then, solving
	\[2 + \frac{\Psi''}{b\Psi} = -\frac{\varepsilon}{b}, \]
	we find that a suitable function $\Psi$ is 
	\begin{align*}
	\Psi(\theta) = \cos((\theta -\overline\theta)\sqrt{1+2\varepsilon})\\
	\implies f(\theta) = \cos^{-\frac{2}{1+\varepsilon}}((\theta -\overline\theta)\sqrt{1+2\varepsilon}),
	\end{align*}
	where we now fix $\varepsilon$ sufficiently small so that $f$ is bounded on $(\overline\theta-\frac{\pi}{2} +\varepsilon_\theta, \overline\theta +\frac{\pi}{2} - \varepsilon_\theta)$. Therefore, the square bracket in (\ref{ineq-halfway}) is now equal to $\frac{-\varepsilon}{\frac{1}{2}(1+\varepsilon)}$, so
\begin{equation}
    0 \leq \left[\frac{-\varepsilon}{\frac{1}{2}(1+\varepsilon)} \right]\big|\vec H\big|^2 + C(n,\varepsilon) \left( -\frac{\Delta^{\mathbb{R}^n} \psi + \mbox{tr}^\perp \mbox{Hess}(\psi)}{\psi} + \frac{|\nabla \psi|^2}{\psi^2} + \frac{1}{r^2} \right).\label{ineq-halfway2}
\end{equation}

	
	We now continue estimating from (\ref{ineq-halfway2}). Define $\psi_R = r^2 \rho_R(r)$, where $\rho_R \colon \R \to [0,1]$ is a smooth cutoff function satisfying:
    \begin{align*}
    \rho_R(r) = \begin{cases} 1 \quad \text{ for } r \leq R,\\ 0 \quad \text{ for } r \geq 2R, \end{cases} \quad
        |\rho_R| \leq 1, \quad
        \left| \frac{\partial \rho_R}{\partial x_i}\right| \leq \frac C R \sqrt{\rho_R}, \quad
        \left| \frac{\partial^2 \rho_R}{\partial x_i\partial x_j}\right| \leq \frac C {R^2}
    \end{align*}
    (for example $\rho_R(r) := \left( g(2 - \frac{r}{R})(g(2-\frac{r}{R}) + g(\frac{r}{R} - 1))^{-1}\right)^2$, where $g(x) := e^{-\tfrac{1}{x}}$).  Note that if we choose $\psi = \psi_R$, then the increasing maximum $(y,s)$ of $\Gamma$ must be inside $B_{2R}$. At $(y,s)$, we may therefore estimate using (\ref{ineq-halfway2}):
	\begin{align*}
	-\Delta^{\mathbb{R}^n} \psi + \mbox{tr}^\perp \mbox{Hess}(\psi) \, \leq \,C(n)&, \quad \quad \frac{|\nabla \psi|^2}{\psi} \, \leq \, C(n), \notag
	\end{align*}
	so
	\begin{align}
	\psi \left[\frac{\varepsilon}{\tfrac{1}{2}(1 + \varepsilon)}\right]\big|\vec H\big|^2 &\leq C(n) \notag\\
	\implies \Gamma \, &\leq  C(n, \varepsilon_{\theta}). \label{eq-gammaineq}
	\end{align}
	We may now estimate $\Gamma$ at an arbitrary point $(x,t) \in L \times [0,T)$. Choose $R = r(F_t(x))$. Since $\Gamma$ has compact support it follows that the supremum $\sup_{y\in L, \, s \in [0,t]} \Gamma$ is attained at a space-time point $(\overline y, \overline s)$. If $\overline s = 0$, then since $L_0$ has bounded mean curvature, there exists a constant $C(L_0)$ such that \[\Gamma(x,t) \leq \Gamma(\overline y,0) \leq C(L_0)R^2.\] If $\overline s \neq 0$, then $(\overline y,\overline s)$ is an increasing maximum, and so by (\ref{eq-gammaineq}), \[\Gamma(x,t) \leq \Gamma(\overline y, \overline s) \leq C(n, \varepsilon_\theta)\] It follows that in general,
	\begin{align*}
	    \Gamma(x,t) \leq C(n,\varepsilon_\theta, L_0)(1+R^2),
	 \end{align*}
and hence
\begin{align*}
|\vec{H}(x,t)|^2 \leq C(n,\varepsilon_\theta,L_0)\left( 1 + \frac{1}{r(x,t)^2}\right).
	\end{align*}
\end{proof}
The location of singularities theorem follows as a corollary of these estimates.

\begin{theorem}[Location of Singularities Theorem]\label{thm-locationofsingularities}
Let $L_t \subset M$ be a connected $G$-invariant almost-calibrated Lagrangian MCF, for $t\in (0,T)$. Then a singularity occurs at time $T$ if and only if $(O,T)$ is the unique singular space-time point for $L_t$.
\end{theorem}

\subsection{Convergence Theorems}

\indent \indent We now prove two key propositions regarding convergence of blowup sequences. Proposition \ref{prop-convergence1} allows us to pass from convergence of the $G$-invariant Lagrangian $L^i$ to convergence of their profile curves $l^i$ when considering the Type I blowup. Proposition \ref{prop-convergence2} meanwhile will be used to rule out certain bad behaviour for Type II blowups.

\begin{proposition}[Convergence of $G$-invariant Lagrangians]\label{prop-convergence1}
Let $L^i\subset M$ be a sequence of connected, $G$-invariant, almost-calibrated Lagrangian submanifolds with Lagrangian angles $\theta_i$, let $L^\infty_j \subset M \cup\{0\}$ be connected smooth Lagrangians or Lagrangian cones with connected link which are pairwise disjoint for $1 \leq j \leq N$, with Lagrangian angles $\overline\theta_j$, and let $m_j$ be integer multiplicities. Assume that for any $\phi \in C^\infty_c(\mathbb{C}^n)$, $f \in C^2(\mathbb{R})$, 
\[ \int_{L^i} f(\theta_i)\phi \,d\mathcal{H}^n \longrightarrow \sum_{j=1}^N m_j \int_{L^\infty_j} f(\overline\theta_j)\phi \,d\mathcal{H}^n\, \text{ as } i \rightarrow \infty. \]
Then:\\
(a) $L^\infty_j \cap M$ are connected, $G$-invariant Lagrangians,\\
(b) For any $\psi \in C^\infty_c(\mathbb{C})$, $f \in C^2(\mathbb{R})$,
\[ \int_{l^i} f(\theta_i) \psi \,d\mathcal{H}^1 \longrightarrow \sum_{j=1}^N m_j \int_{l^\infty_j} f(\overline \theta_j)\psi \, d\mathcal{H}^1 \, \text{ as } i \rightarrow \infty.\]
\end{proposition}
\begin{proof}
We work with the underlying Radon measures. Denoting the indicator function of a subset $A\subset \mathbb{C}^n$ by $\chi_{A}$, we make the definitions
\begin{align*}
\Theta_i:\mathbb{C}^n\rightarrow \mathbb{R}, \,\, \Theta_i &:= f(\theta_i) \chi_{L^i}, \quad \quad \quad \Theta_\infty:\mathbb{C}^n\rightarrow \mathbb{R},\,\, \Theta_i := \sum_{j=1}^N m_j\,f(\overline\theta_j)\chi_{L^\infty_j},\\
\mu_i &:= \mathcal{H}^n \mres \Theta_i, \quad \quad \quad \mu_\infty := \mathcal{H}^n \mres \Theta_\infty.
\end{align*}
Then, for $\phi \in C^\infty_c(\mathbb{C}^n)$, and denoting by $L_g\colon \C^n \to \C^n$ the action of $g\in G$ on $\mathbb{C}^n$, the assumptions of the proposition imply that
\[ \mu_i(\phi) \rightarrow \, \mu_\infty(\phi) \implies \mu_\infty(\phi \circ L_g) = \lim_{i\rightarrow\infty} \mu_i(\phi \circ L_g) = \lim_{i\rightarrow \infty}\mu_i(\phi) = \mu_\infty(\phi).\]
Note that since $L^\infty_j$ is a smooth Lagrangian, this measure-theoretic $G$-invariance of $\mu_\infty$ implies $G$-invariance of the supporting set $L^\infty := \bigcup_{j=1}^N L^\infty_j$, and therefore of the connected components $L^\infty_j$. This proves (a). Defining in the same way the profile curve measures,
\begin{align*}
    \overline\Theta_i:\mathbb{C}\rightarrow \mathbb{R}, \,\, \overline\Theta_i &:= f(\theta_i) \chi_{l^i}, \quad \quad \quad \overline\Theta_\infty:\mathbb{C}^n\rightarrow \mathbb{R},\,\, \overline\Theta_i := \sum_{j=1}^N m_j\,f(\overline\theta_j)\chi_{l^\infty_j},\\
\overline\mu_i &:= \mathcal{H}^1 \mres \overline\Theta_i, \quad \quad \quad \overline\mu_\infty := \mathcal{H}^1 \mres \overline\Theta_\infty,
\end{align*}
our aim is to prove that for all $\psi\in C^\infty_c(\mathbb{C})$, we have $\overline\mu_i(\psi) \rightarrow \overline\mu_\infty(\psi)$.

Consider $\psi \in C^\infty_c(\mathbb{C})$ supported in $B_\delta(x)$ for some $x \in \mathbb{C}\setminus\{0\}$ and $\delta< |x|$. For each $y \in \mathbb{C}^n$, by Proposition \ref{prop-mu0cohom1} there exists $h\in G$ such that $h\cdot y \in P_z$. Picking such an $h$ for each $y$ we may define the $G$-invariant function $\phi:M\rightarrow \mathbb{R}$,  $\phi(y) := \sum_{\alpha \in C_m} \psi(\alpha\cdot h\cdot y)$ (this definition is independent of the choices of $h$). We may then extend $\phi$ to a smooth, $G$-invariant compactly supported function $\phi \in C^\infty_c(\mathbb{C}^n)$. 

Now, by $G$-invariance, the co-area formula, and noting that for $x \in P_z$, $\mathcal{H}^{n-1}(\mathcal{O}_x) = |x|^{n-1}\mathcal{H}^{n-1}(\mathcal{O}_1)$, we calculate the following upper and lower bounds:
\begin{align}
    \mu_\infty(\phi)\, = \, \lim_{i\rightarrow \infty} \int_{L^i} f(\theta_i)\phi \,d\mathcal{H}^n \, &= \, \lim_{i\rightarrow \infty} \int_{l^i}\left( \int_{\mathcal{O}_{l^i(s)}}f(\theta_i) \phi \, d\mathcal{H}^{n-1} \right) \,d\mathcal{H}^1\notag\\
    &= \, \mathcal{H}^{n-1}(\mathcal{O}_1) \cdot m\cdot \lim_{i\rightarrow \infty} \int_{l^i} |l^i(s)|^{n-1} f(\theta_i) \psi \, \,d\mathcal{H}^1\notag\\
    &\geq  \left||x|-\delta\right|^{n-1} \cdot \mathcal{H}^{n-1}(\mathcal{O}_1)\cdot m \cdot \lim_{i\rightarrow \infty}\overline\mu_i(\psi), \label{eq-measurebound1} \\
    \mu_\infty(\phi) = \sum_{j=1}^N m_j\,\int_{L^\infty_j} f(\overline\theta_j)\,\phi\, d\mathcal{H}^n \,&\geq\, ||x| - \delta|^{n-1} \cdot \mathcal{H}^{n-1}(\mathcal{O}_1)\cdot m \cdot \overline\mu_\infty(\psi) \label{eq-measurebound3},    
\end{align}
and
\begin{align}
    \mu_\infty(\phi)\, \,&\leq\, \left| |x| + \delta\right|^{n-1} \cdot \mathcal{H}^{n-1}(\mathcal{O}_1) \cdot m \cdot \lim_{i\rightarrow \infty} \overline\mu_i(\psi), \label{eq-measurebound2} \\
    \mu_\infty(\phi) &\leq ||x|+\delta|^{n-1} \cdot \mathcal{H}^{n-1}(\mathcal{O}_1)\cdot m \cdot \overline\mu_\infty(\psi) \label{eq-measurebound4}.    
\end{align}
The inequality (\ref{eq-measurebound1}) implies that the Radon measures $\overline\mu_i$ are uniformly bounded on compact sets.  Thus, by compactness for Radon measures, after passing to a subsequence there exists a Radon measure $\overline\mu$ such that for all $\psi \in C^\infty_c(\mathbb{C})$, \[ \lim_{i\rightarrow \infty}\overline\mu_i(\psi) = \overline\mu(\psi). \]
\indent We now show $\overline \mu = \overline \mu_\infty$. By considering $\psi$ with $\text{supp}(\psi) \subset B_\delta(x) \setminus(l^\infty \cap \{0\})$, the bound (\ref{eq-measurebound1}) implies that $\text{supp}(\overline\mu) \subset l^\infty$. Therefore, $\overline\mu$ is a $1$-rectifiable Radon measure supported in $l^\infty$, and so there exists $\overline\Theta:\mathbb{C} \rightarrow \mathbb{R}$ supported on $l^\infty$ such that 
\[ \overline\mu(\psi) = \int_{l^\infty} \overline\Theta \, \psi\, d\mathcal{H}^1, \]
where for $\mathcal{H}^1$-almost-every $y\in \text{supp}(\overline\mu)$, $\overline\Theta(y) = \lim_{r\rightarrow 0} \frac{\overline\mu(B(y,r))}{\omega_n r^n}$ (see for example \cite[\S 1.3]{Ilmanen1994}). Similarly, for $\mathcal{H}^1$-almost-every $y\in \text{supp}(\overline\mu_\infty)$, $\overline\Theta_\infty(y) = \lim_{r\rightarrow 0} \frac{\overline\mu_\infty(B(y,r))}{\omega_n r^n}$. By (\ref{eq-measurebound1} - \ref{eq-measurebound4}), 
\begin{align*}
    \frac{||x|-\delta|^{n-1}}{||x|+\delta|^{n-1}} \overline\mu_\infty(\psi) \geq \overline\mu(\psi) \geq \frac{||x|+\delta|^{n-1}}{||x|-\delta|^{n-1}} \overline\mu_\infty(\psi),
\end{align*}
which implies that $\overline\Theta = \overline\Theta_\infty$ on $\mathcal{H}^1$-almost-every $x \in \mathbb{C}$, and consequently that $\overline\mu = \overline\mu_\infty$. This proves that $\lim_{i\rightarrow \infty} \overline\mu_i(\phi) =\overline\mu_\infty(\phi)$ as required.

\end{proof}

\begin{proposition}[Convergence of Translated $G$-invariant Lagrangians]\label{prop-convergence2}
Let $L^i \subset \mathbb{C}^n$ be a sequence of connected, almost-calibrated Lagrangian submanifolds with $O \in L^i$, and let $L^\infty\subset \mathbb{C}^n$ be a connected Lagrangian submanifold, $\nu \in S^{2n-1}$ and $y_i \in M$ such that 
\begin{itemize}
    \item $L^i + y_i$ is a $G$-invariant Lagrangian in $M$,
    \item $y_i \rightarrow \infty$ and $\frac{y_i}{|y_i|} \rightarrow \nu$,
    \item $L^i$ converges to $L^\infty$ in $C^\infty_\emph{loc}$.
\end{itemize}
Then $\nu \in M$, and $T_\nu\mathcal{O}_\nu \subset L^\infty$. Since $\mathcal{O}_\nu$ is isotropic, it follows that $L^\infty$ contains an isotropic $(n-1)$-plane.
\end{proposition}
\begin{proof}
Recall $z_0 \in M$, so that $\frac{z_0}{|z_0|} \in M \cap S^{2n-1}$. By Proposition \ref{prop-mu0cohom1}, $M \cap S^{2n-1} = (\U(1)\times G)\cdot \frac{z_0}{|z_0|}$, and so $M \cap S^{2n-1}$ is closed.  Therefore, since $y_i\in M$, we have $\hat y_i := \frac{y_i}{|y_i|} \in M\cap S^{2n-1}$, and hence $\nu \in M$. \\
\indent Now by Proposition \ref{prop-mu0cohom1}, there exist $g_i \in G$ such that $g_i \hat y_i \in P_\nu$. Since $G$ is compact, by passing to a subsequence we may assume that $g_i \rightarrow \overline g$ for $\overline g \in G$ such that $\overline g \nu \in P_\nu$, from which it follows that $\overline g^{-1} g_i\hat y_i \in P_\nu$, $\overline g^{-1} g_i \rightarrow \text{Id}$, and $\overline g^{-1} g_i \hat y_i \rightarrow \nu$. Furthermore, defining $\alpha_i \in \mathbb{C}^*$ such that $\overline g^{-1} g_i\hat y_i = \alpha_i\cdot \nu$, it follows that $\alpha_i \to 1$. Then, since $L^i + y_i$ is $G$-invariant, 
\begin{align} \mathcal{O}_{y_i} \subset L^i + y_i &\implies |y_i|\left( \mathcal{O}_{\hat y_i} - \hat y_i \right) \subset L^i\notag\\
    &\implies |y_i|g_i^{-1}\overline g\left(\mathcal{O}_{\overline g^{-1} g_i\hat y_i} - \overline g^{-1} g_i\hat y_i \right) \subset L^i\notag\\
    &\implies  g_i^{-1}\overline g\alpha_i \cdot |y_i|\left(\mathcal{O}_\nu - \nu \right) \subset L^i. \label{eq-orbitcontainment}
\end{align}
Since $\mathcal{O}_\nu$ is a smooth $(n-1)$-dimensional submanifold and $|y_i|\rightarrow \infty$, it follows that $|y_i|\left(\mathcal{O}_\nu - \nu\right)$ converges locally smoothly to $T_\nu \mathcal{O}_\nu$. Since $g_i^{-1} \overline g \alpha_i \rightarrow \text{Id}$, it follows that the left hand side of the inclusion of (\ref{eq-orbitcontainment}) smoothly converges to $T_\nu \mathcal{O}_\nu$, which along with the smooth convergence of $L^i$ to $L^\infty$ implies that $T_\nu \mathcal{O}_\nu \subset L^\infty$, as required.
\end{proof}


\begin{remark}
In fact, more is true: given the assumptions of Proposition \ref{prop-convergence2}, $L^\infty$ is translation invariant with respect to the additive action of the abelian group $T_\nu\mathcal{O}_\nu$. However, we will not require the stronger result in this work.
\end{remark}

\subsection{The Type I Blowup}

\indent \indent In this subsection, we prove Theorem \ref{thm-typeiblowup} characterising Type I blowups of cohomogeneity-one almost-calibrated Lagrangian mean curvature flow, which we show to be unit-density pairs of Theorem \ref{thm-splagcones}'s special Lagrangian cones. A sketch of the proof is as follows. Applying Theorem A of \cite{Neves2007} and Proposition \ref{prop-convergence1}, a Type I blowup must be a union of $G$-invariant special Lagrangian cones, and the Lagrangian angle converges in an integral sense. In Lemma \ref{lem-keylemma}, we use this convergence along with the curvature bounds of Theorem \ref{thm-equicurvestimates} to rule out higher-multiplicity cones in the blowup, demonstrating unit-density. By Lemma \ref{lem-wedge} and Theorem \ref{thm-splagcones}, for each Lagrangian angle $\overline\theta$ there are only two possible cones that may occur in the blowup, and we may argue that the blowup is their union.

\begin{lemma}[Multiplicity-One Lemma]\label{lem-keylemma}
    Let $l^i \subset \mathbb{C}\setminus \{0\}$ be a sequence of complete connected embedded smooth curves. Let $l^\infty \subset \mathbb{C}$ be a half-line from the origin, $\overline x \in l^\infty$, $0 < d < |\overline x|$, and define $\theta^i = \emph{arg}\!\left((l^i)'\right) + (n-1)\emph{arg}\!\left(l^i\right)$. Assume further that:
    \begin{itemize}
        \item[(a)] There exist $m \in \mathbb{N}$, $\overline\theta \equiv \emph{arg}(l^\infty)\,\, (\emph{mod }\pi)$ such that for all $\psi \in C^\infty_c(B_d(\overline x))$ and $f\in C^2(\mathbb{R})$,
        \[ \int_{l^i} f(\theta^i)\psi \, d\mathcal{H}^1 \longrightarrow m \cdot \int_{l^\infty} f(\overline\theta) \psi \, d\mathcal{H}^1,\]
        \item[(b)] There exists $c$ such that $|\vec{k}(l^i)| \leq c$ in $B_{d}(\overline x)$ for all $i$.
    \end{itemize}
    Then $m$ = 1.
\end{lemma}

\begin{proof} Without loss of generality (for example, by applying a rotation and a scaling) we may assume that $l^\infty$ is the positive $x$-axis, $\overline x = 1$, and $\overline\theta \equiv 0\,\,(\text{mod } \pi)$. We will make the assumption that $\overline\theta \equiv 0 \,\,(\text{mod } 2\pi)$, since the other case is identical. Throughout the proof, all balls are centred at $\overline x$. For a curve $\eta \subset \mathbb{C}\setminus \{0\}$, define $T_{\varepsilon,\delta}(\eta) := \{ x \in \eta\cap B_\delta \, : \, |\theta(x) -\overline\theta| > \varepsilon \}$.  We make three claims:\\

\indent \emph{Claim 1.} \emph{(Uniform convergence of angle)}\\  For all $\delta < \tfrac{d}{2}$ and $\varepsilon > 0$, there exists $N = N(\varepsilon,\delta)$ such that for all $i > N$, $T_{\varepsilon,\delta}(l^i) = \O$. \\

\indent \emph{Claim 2.} \emph{(Single connected component)}\\ There exists $\overline\delta$ with $d > \overline\delta > 0$ such that if $\delta,\varepsilon < \overline\delta$ satisfy $l^i \cap B_{{\delta}/{2}} \neq \O$ and $T_{\varepsilon,2\delta}(l^i) = \O$, then there is only one connected component of $l^i \cap B_\delta$.\\

\indent \emph{Claim 3.} \emph{(Convergence of Length)}\\ For any $\delta \leq d$, $\mathcal{H}^1(l^i \cap B_\delta) \rightarrow 2\delta m.$\\

\indent From these claims, we may deduce the result as follows.  Choose $\overline{\delta} > 0$ as in Claim 2, and let $0 < \delta, \varepsilon < \overline{\delta}$ be small enough that $\cos(\varepsilon + (n-1)\delta) \geq \frac{2}{3}$ and $2\delta < \frac{d}{2}$.  By Claim 1, there exists $N = N(\varepsilon, 2\delta)$ such that $i > N$ implies $T_{\varepsilon, 2\delta}(l^i) = \O$.  Choosing $N$ larger if necessary, assumption (a) implies that for $i > N$ we have $l^i \cap B_{\delta/2} \neq \O$.  Therefore, by Claim 2, there is only one connected component of $l^i \cap B_\delta$, which we label $\eta^i$ (note that therefore $l^i \cap B_\delta = \eta^i)$.  In particular, we have $T_{\varepsilon, \delta}(\eta^i) = \O$, so that every $x \in \eta^i$ satisfies $|\theta(x) - \overline{\theta}| \leq \varepsilon$. \\
\indent Now, using that $\eta \subset B_\delta$, and then invoking Lemma \ref{lem-laganglelemma}, suppressing the index $i$ for notational clarity, we have bounds: 
\begin{align*}
    \arg(\eta) &\in [-\sin^{-1}(\delta)\,,\, \sin^{-1}(\delta)] \subset [-\delta, \delta],\\
    \implies\arg(\eta') &\in \left[\,\overline\theta - \varepsilon - (n-1)\delta\,,\,\overline\theta + \varepsilon + (n-1)\delta\right].
\end{align*}
Therefore, parametrising $\eta$ by unit speed (remembering that $\overline\theta = 2k\pi$), we have that $\eta'(s) = e^{i(\lambda(s) + \overline\theta)} = e^{i\lambda(s)}$ for some $\lambda(s) \in \left[ -\varepsilon -(n-1)\delta\,,\, \varepsilon + (n-1)\delta \right]$, and hence
\begin{align*}
    2\delta \geq \left| \int_{\eta} \eta' \,ds \right| \geq \left| \int_\eta \cos(\lambda(s)) \,ds\right|  \geq \mathcal{H}^1 (\eta) \cos(\varepsilon + (n-1)\delta) \geq \mathcal{H}^1(\eta) \cdot \frac{2}{3}
\end{align*}
from which we obtain
\begin{align*}
    \mathcal{H}^1(l^i \cap B_\delta) = \mathcal{H}^1(\eta^i) \leq 2\delta \cdot \frac{3}{2}. 
\end{align*}
Finally, by Claim 3, it follows that $m=1$, as required. It therefore remains to prove the three claims. \\

\emph{Proof of Claim 1:} Parametrise $l^i$ by arclength, say $l^i(s) = r(s) e^{i\alpha(s)}$.  A calculation shows that within $B_d$, by assumption (b):
\begin{align}
    \theta & = \tan^{-1}(\tfrac{r\alpha'}{r'}) + n\alpha \notag\\
    \implies \theta' & = \langle l'', il'\rangle + (n-1)\alpha'\notag\\
    \implies |\theta'| \,&\leq\, |\vec k| + \frac{n-1}{r} \,\leq\, c + \frac{n-1}{1-d} \,=:\, \widetilde c. \label{eq-thetaprimebound}
\end{align}

Now choose $\delta < \tfrac{d}{2}$ and $\varepsilon\leq \overline\varepsilon_\delta$, where $\overline\varepsilon_\delta := \frac{1}{4}(d-2\delta)(1+\widetilde c)$. Furthermore, define $\varepsilon' = \frac{\varepsilon}{1+\widetilde c}$ and choose smooth functions $\psi:\mathbb{C} \to [0,1]$, $f:\mathbb{R}\to [0,1]$ so that $\psi = 1$ on $B_{\tfrac{d}{2}}$ and $\psi = 0$ outside of $B_{d}$, and $f = 0$ on $(\overline\theta - \frac{\varepsilon'}{2},\overline\theta + \frac{\varepsilon'}{2})$ and $f=1$ outside of $(\overline\theta - \varepsilon',\overline\theta + \varepsilon')$. By assumption (a), we have
\[ \mathcal{H}^1\!\left(T_{\varepsilon', \delta + \varepsilon'}(l^i)\right) \leq \int_{l^i} f(\theta) \psi \,d\mathcal{H}^1 \rightarrow m \cdot \int_{l^\infty} f(\overline\theta) \psi \,d\mathcal{H}^1 = 0, \]
and therefore there exists $N = N(\varepsilon, \delta)$ such that for all $i > N$, $\mathcal{H}^1\!\left(T_{\varepsilon', \delta + \varepsilon'}(l^i)\right) \leq \varepsilon'$.

Now, since $\mathcal{H}^1\left(T_{\varepsilon', \delta + \varepsilon'}(l^i)\right) \leq \varepsilon'$, for $i > N$ any connected component of $l^i \cap B_{\delta + \varepsilon'}$ intersecting $B_\delta$ must include a point $x$ such that $\theta(x) \in (\overline\theta - \varepsilon', \overline\theta + \varepsilon')$. Therefore by (\ref{eq-thetaprimebound}), on $B_\delta$, $\theta \in (\overline\theta - \varepsilon'- \widetilde c \varepsilon', \overline\theta + \varepsilon' + \widetilde c\varepsilon') = (\overline\theta - \varepsilon, \overline\theta + \varepsilon)$ as required.

Finally, since $\varepsilon < \tilde \varepsilon$ implies $T_{\delta,\varepsilon}(l^i) \supset T_{\delta,\tilde \varepsilon}(l^i)$, we note that given $\delta<\tfrac{d}{2}$ and $\varepsilon > 0$, for all $i > N(\min(\varepsilon,\overline \varepsilon_\delta), \delta)$ it follows that $T_{\varepsilon, \delta}(l^i) \subset T_{\min(\overline\varepsilon_\delta,\varepsilon), \delta}(l^i) = \O$.\\

\emph{Proof of Claim 2:} Consider $\delta, \varepsilon < \overline\delta$, where $\overline\delta < \frac{\pi}{4n}$ will be chosen during the course of the proof. Choose $x \in l^i\cap B_{\frac{\delta}{2}}$, let $\eta^i$ be the connected component of $l^i \cap B_{\delta}$ containing $x$, and let $\overline\eta^i$ be the connected component of $l^i \cap B_{2\delta}$ containing $x$. Since $T_{\varepsilon, 2\delta}(l^i) = \O$, within $B_{2\delta}$ we have the following bound by Lemma \ref{lem-laganglelemma} (suppressing the index $i$ for notational clarity):
\begin{align*}
    \text{arg}(l') &\in \left[\,\overline\theta - \varepsilon - (n-1)\delta\,,\, \overline\theta + \varepsilon + (n-1)\delta\right]\\
    &\equiv \left[ -\varepsilon - (n-1)\delta\,,\, \varepsilon + (n-1)\delta\right] \,\,(\text{mod }2\pi)\\
    &\subset \left[-n\overline\delta\,,\, n\overline\delta\right]\,\,(\text{mod }2\pi).
\end{align*}
Note that since $n\overline\delta < \frac{\pi}{4}$, this bound along with the fact that $l^i \cap B_{\frac{\delta}{2}} \neq \O$ implies that $\overline\eta$ must enter $B_{2\delta}$ through the left semicircle of $B_{2\delta}$ and leave $B_{2\delta}$ through the right semicircle.  Moreover, since $\eta$ intersects $B_{\frac{\delta}{2}}$, if $\overline\delta$ is taken sufficiently small, then it also implies that $\overline\eta \cap B_\delta = \eta$, so that $\overline \eta$ contains only one connected component of $l \cap B_\delta$. We fix $\overline\delta$ such that both of these facts hold.

Now, for a contradiction, assume that there is a second connected component $\xi \subset l \cap B_\delta$, contained in a connected component $\overline\xi \subset l \cap B_{2\delta}$. Then by the same argument, $\overline\xi$ enters on the left of $B_{2\delta}$ and leaves on the right. By the embeddedness assumption, $\overline\xi$ and $\overline\eta$ are disjoint. However, this behaviour is topologically impossible, since $\overline\eta$ and $\overline\xi$ are connected components of a complete connected embedded curve $l$, the ends of which must diverge to infinity.\\

\emph{Proof of Claim 3:} For $0 < a < d-\delta$, we define smooth compactly supported test functions $\check \psi, \hat\psi \colon \mathbb{C} \to [0,1]$ satisfying
\begin{align*}
    \check\psi := \begin{cases}  1 \quad \text{ on } B_{\delta - a}\\ 0 \quad \text{ outside } B_{\delta},\end{cases} \quad \hat\psi := \begin{cases} 1 \quad \text{ on } B_{\delta}\\ 0 \quad \text{ outside } B_{\delta + a}.\end{cases}
\end{align*}
Then
\[ \int_{l^i} \check\psi \,d\mathcal{H}^1 \,\leq\, \mathcal{H}^1(l^i \cap B_\delta) \, \leq\, \int_{l^i} \hat\psi\, d\mathcal{H}^1, \]
and as $i \rightarrow \infty$, by assumption (a) (choosing $f \equiv 1$),
\begin{align*}
    \int_{l^i} \check \psi\,d\mathcal{H}^1 \rightarrow m\cdot\int_{l^\infty} \check \psi \,d\mathcal{H}^1 \, \geq 2(\delta - a)m, \quad
    \int_{l^i} \hat\psi\,d\mathcal{H}^1 \rightarrow m\cdot \int_{l^\infty} \hat \psi \,d\mathcal{H}^1 \, \leq 2(\delta + a)m.
\end{align*}
Since $a > 0$ may be chosen arbitrarily small, the result follows.



\end{proof}

Having established this, we now characterise the Type I blowups of cohomogeneity-one $G$-equivariant Lagrangian mean curvature flow, using the notation of Theorem \ref{thm-splagcones}.

\begin{theorem}[Structure of Type I Blowups]\label{thm-typeiblowup}
    Let $L_t \subset M$ be a connected $G$-invariant almost-calibrated Lagrangian MCF for $t \in [0,T)$, with $T$ the singular time.
    
    Then there exist $\overline\theta$ and $k \in \{1,\ldots, \frac{2n}{m}\}$ such that any Type I blowup at time $T$ is the special Lagrangian cone $L^\infty$ with profile curve $l^\infty = C_m \cdot (\tilde c_{k-1,\overline\theta} \cup \tilde c_{k,\overline\theta})$.
\end{theorem}

\begin{proof}
Without loss of generality, assume $T=0$. By Theorem \ref{thm-locationofsingularities}, $(O,0)$ is the unique singular space-time point of the flow. Let $W \subset \C$ be the wedge given by Lemma \ref{lem-wedge}, and note that by Theorem \ref{thm-splagcones}, for each $\overline\theta$ there exist at most two $G$-invariant special Lagrangian cones $\widetilde C \subset M$ with connected link, profile curve $\widetilde c$ contained in $C_m \cdot W$ and Lagrangian angle $\overline\theta$, which have consecutive values of the parameter $k$.

Consider a sequence $L_s^i$ of Type I rescalings, defined by $ L_s^i := \lambda_i L_{\lambda_i^{-2}s}$ with $\lambda_i\to \infty$. By Theorem \ref{thm-nevesab} (A), there exist $N$ special Lagrangian cones $L^\infty_j$ with connected link, Lagrangian angles $\overline\theta_j$ and multiplicities $m_j \in \mathbb{Z}^+$ such that for $\phi \in C^\infty_c(\mathbb{C}^n)$, $f \in C^2(\mathbb{R})$, $s <0$,
\[ \lim_{i\rightarrow \infty}\int_{L_s^i} f(\theta_{i,s})\phi \,d\mathcal{H}^n \, = \, \sum_{j=1}^N m_j f(\overline\theta_j)\int_{L^\infty_j}\phi \,d\mathcal{H}^n.\]
In other words, the Type I blowup corresponding to the sequence $L^i_s$ is $L^\infty := \sum_{j=1}^N m_j \cdot L^\infty_j$. We will show that each $m_j = 1$ and that $N = 2$. \\
\indent By the almost-calibrated condition, the special Lagrangian cones $L^\infty_j$ are distinct from one another. By Proposition \ref{prop-convergence1}, $L^\infty_j$ is $G$-invariant for each $1\leq j\leq N$, and by Theorem \ref{thm-splagcones} the profile curves $l^\infty_j$ are $C_m$-orbits of half-lines from the origin. Moreover, for all $s<0$ the sequence $l^i_s$ converges in the sense of Proposition \ref{prop-convergence1} to the profile curves $l^\infty_j$ with multiplicities $m_j$. Choosing $\overline x_j \in l^\infty_j$ and $d > 0$ such that $B_d(\overline x_j) \cap l^\infty_k = \O$ for each $k \neq j$, we note that the conditions of Lemma \ref{lem-keylemma} are satisfied for each $j$ (in particular, the curvature bound follows from Theorem \ref{thm-equicurvestimates}), and we may therefore conclude that $m_j = 1$ for all $j$.

We now show that $N = 2$. Choose $s<0$ such that (perhaps on passing to a subsequence) the conclusion of Theorem \ref{thm-nevesab} (B) holds for the sequence $(L^i_s)$ and $R=1$. Therefore, choosing a convergent sequence $\Sigma^i$ of connected components of $L^i_s \cap B_1(O)$ intersecting $B_{\frac{1}{4}}(O)$, there exists a special Lagrangian cone $\Sigma^\infty \subset \bigcup_{i=1}^N L^\infty_j$ with Lagrangian angle $\overline \theta$ such that for all $\phi \in C^\infty_c(\mathbb{C}^n)$, $f \in C^2(\mathbb{R})$,
\[ \lim_{i\rightarrow \infty}\int_{L_s^i} f(\theta_{i,s})\phi \,d\mathcal{H}^n \, = \, f(\overline\theta_j)\int_{\Sigma^\infty}\phi \,d\mathcal{H}^n.\]
By the same argument as above, $\Sigma^\infty$ is $G$-invariant, and the profile curve $\sigma^\infty$ is a union of $C_m$-orbits of half-lines $\bigcup_{k=1}^M \sigma^\infty_k$, for $1 \leq M \leq N$. Note that since $\Sigma^i\cap B_{\frac{1}{4}}(O) \neq \O$, it follows that $\mathcal{H}^1\left(\sigma^\infty \cap B_1(O)\right) \geq \frac{3}{2}$, so $M \geq 2$, and since there are only two $G$-invariant special Lagrangian cones of angle $\overline\theta$ with profile curve in $W$, we have $M=2$, and hence $N \geq 2$.  Explicitly, there exist $k, \overline\theta$ such that $\sigma^\infty = C_m \cdot (\tilde c_{k-1,\overline\theta} \cup \tilde c_{k,\overline\theta})$. \\
\indent Finally, $N>2$ would imply that there exists a second sequence of connected components $\widehat\Sigma^i$ of $L^i \cap B_1(O)$ intersecting $B_{\frac{1}{4}}(O)$, converging to $\widehat \Sigma^\infty$ with profile curve $\widehat \sigma^\infty = C_m \cdot (\widetilde c_{\hat k-1, \hat\theta} \cup \widetilde c_{\hat k, \hat\theta})$. Since $m_j = 1$ for all $j$, it follows that $\hat \theta \neq \overline\theta$. However, the convergence of $\sigma^i \to \sigma^\infty$ and $\widehat \sigma^i \to \widehat \sigma^\infty$ are only possible if $\sigma$ and $\widehat \sigma$ intersect for sufficiently large $i$, contradicting embeddedness of $l^i$. Therefore, $N=2$ and $l^\infty = \sigma^\infty = C_m \cdot (\widetilde c_{k-1, \overline\theta} \cup \widetilde c_{ k, \overline\theta})$.

Finally, by Theorem \ref{thm-nevesab} (A), the angle $\overline\theta$ does not depend on the choice of blowup sequence.
\end{proof}

\subsection{The Type II Blowup}

\indent \indent In this subsection, we prove Theorem \ref{thm-typeiiblowup} characterising Type II blowups of cohomogeneity-one Lagrangian mean curvature flow, which we show to be the unique special Lagrangians of Theorem \ref{thm-splags}. Our first step is to note that if the Type I blowups at a singular space-time point all have Lagrangian angle $\overline \theta$, then the Type II blowups are special Lagrangians of the same angle. This follows as a corollary of the following theorem, proven in \cite{Wood2019}.

\begin{theorem}[Convergence of Lagrangian Angle for the Type II Rescalings]\label{prop-typeIIangle}
	Let $L_t$ be an almost-calibrated Lagrangian MCF in $\mathbb{C}^n$ with Lagrangian angle $\theta_t$, which forms a singularity at time $T$ with singular space-time point $(O,0)$. Assume that any convergent sequence of Type I rescalings $L^{\sigma_i}_s$ converge to the same special Lagrangian cone $C$, with angle $\overline{\theta}$. 
	
	Let $X_i=(x_i,t_i)$ be a sequence of space-time points such that $(x_i,t_i)\rightarrow (O,0)$, let $\lambda_i \in \mathbb{R}$ satisfy $-\lambda_i^2 t_i \rightarrow \infty$, and define the rescalings
	\[ L^{X_i,\lambda_i}_\tau \, := \, \lambda_i\left( L_{t_i+\lambda_i^2 \tau} - x_i \right) \]
	with Lagrangian angle $\theta_\tau^i$.
	
	Then for any bounded parabolic region $\Omega \times I \subset \mathbb{C}^n \times \mathbb{R}$, $\theta^i_\tau \rightarrow \overline\theta$ uniformly in $\Omega\times I$. In particular, if the singularity is Type II and the rescalings are Type II rescalings, then the Type II blowup is a special Lagrangian with angle $\overline\theta$.
\end{theorem}

We may now conclude that any Type II blowup of a $G$-equivariant Lagrangian mean curvature flow $L_t \subset M$ must be one of the $G$-invariant special Lagrangians in $M$ given by Theorem \ref{thm-splags}.

\begin{theorem}[Structure of Type II Blowups]\label{thm-typeiiblowup}
    Let $L_t \subset M$ be a connected $G$-invariant almost-calibrated Lagrangian MCF for $t \in [0,T)$, with $T$ the singular time. Let $\overline \theta \in \R$ and $k \in \{1,\ldots, \frac{2n}{m}\}$ be the constants given by Theorem \ref{thm-typeiblowup}.
    
    Then there exists $B > 0$ such that any Type II blowup at time $T$ is a translation of the connected $G$-invariant special Lagrangian $\widetilde L^\infty$ with profile curve $\widetilde l^\infty = C_m \cdot \widetilde l_{B,k,\overline \theta}$. In particular, the asymptotes of the Type II blowup are given by the unique Type I blowup of Theorem \ref{thm-typeiblowup}.
\end{theorem}
\begin{proof}
Let $L_t \subset M$ be a connected $G$-invariant almost-calibrated LMCF for $t \in [0,T)$, with a singularity at time $T$. By Theorem \ref{thm-locationofsingularities}, $(O,T)$ is the unique singular space-time point of the flow. Let $\overline\theta$, $k$ be the constants given by Theorem \ref{thm-typeiblowup}. Let $W \subset \mathbb{C}$ be the wedge given by Lemma \ref{lem-wedge}, and note that by Theorem \ref{thm-splags} there is a unique special Lagrangian $\widetilde L$ with profile curve $\widetilde l \subset C_m \cdot W$, $\sup_{\widetilde L} |A| = 1$, and Lagrangian angle $\overline\theta$, given by $\widetilde l_{B,k,\overline\theta}$ for some unique $B > 0$. We will prove that any Type II blowup is given by this special Lagrangian.

Consider then a Type II blowup $L_s^{(p_i,t_i)} \rightarrow \widetilde L_s^\infty$, denoting by $A_i$ the scaling factors. By the construction of Type II blowups, $\widetilde L^\infty_s$ is a smooth mean curvature flow. Moreover, since any Type I blowup is the same special Lagrangian cone with angle $\overline\theta$ by Theorem \ref{thm-typeiblowup}, it follows by Proposition \ref{prop-typeIIangle} that $\widetilde L^\infty_s$ is a special Lagrangian with angle $\overline\theta$, and so a static mean curvature flow. We may therefore work with a particular time slice. Fix $s \in \mathbb{R}$, and define
\begin{align*}
    x_i := F_{t_i}(p_i), \quad\quad L^i := L_s^{(p_i,t_i)} = A_iL_{A_i^{-2}s + t_i} - A_ix_i, \quad\quad y_i := A_i x_i,
\end{align*}
so that $L^i + y_i$ is a sequence of $G$-invariant Lagrangian submanifolds, and if $\widetilde L^\infty := \widetilde L^\infty_s$,
\[ \int_{L^i} \phi \,d\mathcal{H}^n \longrightarrow \int_{\widetilde L^\infty} \phi \,d\mathcal{H}^n \, \text{ as } i \rightarrow \infty. \] Passing to a subsequence, we may assume $\frac{y_i}{|y_i|} \rightarrow \nu \in S^{2n-1}$, and also that either $|y_i|\rightarrow \infty$ or there exists $y_\infty$ such that $y_i \rightarrow y_\infty$.

Assume that $|y_i|\rightarrow \infty$. Then Proposition \ref{prop-convergence2} implies that $\widetilde L^\infty$ contains an isotropic $(n-1)$-plane. It follows from \cite[Thm.\ III.5.5]{Harvey1982} that $\widetilde L^\infty$ is a special Lagrangian  $n$-plane. This is a contradiction, since by the properties of Type II blowups, $\sup_{\widetilde L^\infty} |A| = 1$. 

We must therefore have that there exists $y_\infty$ such that $y_i \rightarrow y_\infty$. Then 
\[ A_i L_{A_i^{-2}s + t_i} \to \widetilde L^\infty + y_\infty, \]
and $\widetilde L^\infty + y_\infty$ is a $G$-invariant submanifold by Proposition \ref{prop-convergence1}, with $\sup_{\widetilde L^\infty + y_\infty}|A| = 1$ and with profile curve contained in $C_m \cdot W$. As already noted, there is only one such special Lagrangian of angle $\overline\theta$, with profile curve $C_m \cdot \widetilde l_{B,k,\overline\theta}$, and so we are done.
\end{proof}

\subsection{Existence of Singularities with Prescribed Models}

\indent \indent Finally, we consider the behaviour of cohomogeneity-one Lagrangian mean curvature flow for a particular class of initial conditions which form finite-time singularities. In the process, we demonstrate existence of singularities modelled on the examples of Theorems \ref{thm-splagcones} and \ref{thm-splags}. We restrict to the case of connected Lagrangian flows $L_t \subset M$ such that each connected component $\eta_t$ of the profile curve $l_t$ is \textit{asymptotically linear}, i.e.\  there exist distinct half-lines $c_1, c_2$ and a ball $B_R$ such that for all $t$, $\eta_t \setminus B_R$ may be expressed as a graph over $c_1 \cup c_2$ that decays to $0$ in $C^0$ at infinity. We will denote the spanning angle between the asymptotes $c_1,c_2$ of each connected component of $l$ by $\alpha_{\text{span}}(l)$.

Our key tool is the following avoidance principle for equivariant LMCF \cite[Thm.\ 4.2]{Wood2019}, which allows us to employ barrier arguments. 

\begin{theorem}[Preservation of Embeddedness / Avoidance Principle]\label{thm-avoidance}
Let $f_t,\tilde f_t: \mathbb{R} \rightarrow \mathbb{C}\setminus\{0\}$ be asymptotically linear solutions of (\ref{eqn-equivariantLMCF}) for $t \in [0,T)$ with images $l_t, \tilde l_t$ respectively, such that $l_0$ and $\tilde l_0$ are embedded and disjoint and the asymptotes of $l_t$ and $\tilde l_t$ are different. Then $l_t$, $\tilde l_t$ are embedded and disjoint for all $t \in [0,T)$.
\end{theorem}

We first consider the case where $\alpha_{\text{span}} > \frac{\pi}{n}$. In this case, we may use a barrier flow constructed by Neves that collapses to the origin to `force' a singularity to occur. The proof is identical to that of \cite[Thm.\ 4.11]{Wood2019}.

\begin{theorem}[Finite-Time Singularity for Wide Asymptotic Angles] 
Let $L_t \subset M$ be a connected, embedded $G$-invariant Lagrangian mean curvature flow such that the profile curve $l_t$ is asymptotically linear, and $\alpha_\text{span}(l_t) > \frac{\pi}{n}.$ Then $L_t$ forms a finite-time singularity.
\end{theorem}
\begin{proof}
Let $\eta_t$ be the family of asymptotically linear curves with $\alpha_\text{span}(\eta_t) = \beta$ solving (\ref{eqn-equivariantLMCF}) with initial condition
\[ \eta_0(s) := \left(\sin\left(\tfrac{\pi s}{\beta}\right) \right)^{\frac{-\beta}{\pi}}e^{is}. \]
This flow exists by an argument of Neves \cite[Thm.\ 4.1]{Neves2007}, and forms a finite-time singularity at the origin if $\frac{2\pi}{n} > \beta > \frac{\pi}{n}$ \cite[Thm.\ 4.10]{Wood2019}.

Now consider a connected component $\gamma_t$ of $l_t$, choose $\beta<\alpha$, and rotate and scale the curve $\eta_0$ so that $\gamma_0$ lies in between $\eta_0$ and the origin, and $\eta_0$ and $\gamma_0$ are disjoint. Let $T$ be the singular time of $\eta_t$. It follows by Theorem \ref{thm-avoidance} that $\gamma_t$ approaches the origin by time $T$, and (\ref{eq-meancurvature}) then implies that the curvature of $\gamma_t$ becomes unbounded by time $T$.
\end{proof}
As a corollary, we note that for each of the special Lagrangians of Theorem \ref{thm-splags}, there exists a Lagrangian mean curvature flow forming a finite-time singularity modelled on that special Lagrangian.
\begin{theorem}[Existence of Singularities with Prescribed Models]\label{thm-existenceofsingularities}
Let $L^\infty$ be a complete connected $G$-invariant special Lagrangian of constant isotropy type such that $L^\infty \subset \mu^{-1}(0)$.

Then $L^\infty$ is asymptotically conical, and there exists a Lagrangian mean curvature flow $L_t$ forming a Type II singularity such that:
\begin{itemize}
    \item Any Type I blowup is the asymptotic cone of $L^\infty$,
    \item Any Type II blowup is $L^\infty$.
\end{itemize}
\end{theorem}
\begin{proof}
Let $(H)$ be the isotropy type of $L^\infty$, so there exists a connected component $M \subset \mu^{-1}(0) \cap \mathbb{C}^n_{(H)}$ with $L^\infty \subset M$. It follows from Theorems \ref{thm-splagcones} and \ref{thm-splags} that there exist $B>0,\, k \in \mathbb{Z},\, \overline\theta \in \mathbb{R}$ such that $L^\infty$ has profile curve $C_m \cdot \tilde l_{B,K,\overline\theta}$, and is asymptotically conical to a special Lagrangian cone with profile curve $C_m \cdot (\tilde c_{k-1,\overline\theta} \cup \tilde c_{k, \overline\theta}).$

Now, the Neves example $\eta_t \subset \mathbb{C}$ provides an example of a flow in $M$ with a finite-time singularity. By Theorems \ref{thm-typeiblowup} and \ref{thm-typeiiblowup}, the Type I and Type II blowups are of the required forms, after applying a $\U(1)_\Delta$ rotation to the flow.
\end{proof}

\begin{remark}
Consider the $T^{n-1}$-action on $\mathbb{C}^n$, as discussed in Examples \ref{ex-Torus} and \ref{ex-CyclicGroup}(b) and Remark \ref{remark-TorusCone}. Let $L_t$ be an almost-calibrated Lagrangian mean curvature flow in $M_0 := \mu^{-1}(0) \setminus \{0\}$; by Propositions \ref{prop-embedded} and \ref{prop-exact} $L_t$ must be homeomorphic to $\mathbb{R}\times T^{n-1}$. Assuming that $T$ is the singular time of the flow, it follows by Theorem \ref{thm-locationofsingularities} that the unique singular point is the origin. By Theorems \ref{thm-typeiblowup} and \ref{thm-typeiiblowup}, there exist $\overline{\theta} \in \R$ and $k \in \{1,2\}$ such that:
\begin{itemize}
    \item Any Type I blowup at time $T$ is the \textit{pair} of Harvey-Lawson $T^{n-1}$-invariant special Lagrangian cones whose profile curve is $C_n \cdot (\widetilde{c}_{k-1, \overline{\theta}} \cup \widetilde{c}_{k,\overline{\theta}})$.
    \item There is a unique $B > 0$ such that any Type II blowup at time $T$ is a translation of the connected $T^{n-1}$-invariant special Lagrangian with profile curve $C_n \cdot \widetilde l_{B,k,\overline \theta}$.
\end{itemize}
Moreover, by Theorem \ref{thm-existenceofsingularities}, there exists a flow exhibiting such a finite-time singularity. This contrasts with the result of Lambert-Lotay-Schulze \cite[Thm. 1.2]{Lambert2021} that there is no singularity of almost-calibrated Lagrangian mean curvature flow in a Calabi-Yau 3-fold such that the blowdown of the Type II blowup is given by a \textit{single} Harvey-Lawson $T^{2}$-cone.
\end{remark}

\section{Examples}\label{sec-7}

\indent \indent Above, we studied $G$-invariant Lagrangians and $G$-equivariant LMCF, where $G \leq \SU(n)$ is a compact connected Lie subgroup acting linearly on $\C^n$.  Our analysis focused on the cohomogeneity-one case, in which there exists an orbit type with $(n-1)$-dimensional isotropic orbits.  In this section, we provide examples of compact connected Lie subgroups $G \leq \SU(n)$ admitting such an orbit type.

\indent In $\S$7.1, we explain how, if $G$ is semisimple, $(n-1)$-dimensional isotropic orbits in $\C^n$ may be recast as Lagrangian orbits in $\CP^{n-1}$.  Then, in $\S$7.2, we quote Bedulli and Gori's classification \cite{Bedulli2008} of the unitary representations of compact \emph{simple} Lie groups admitting a Lagrangian orbit in $\CP^{n-1}$.  In this way, their list --- consisting of 7 infinite families and 14 sporadic exceptions --- also classifies the compact simple Lie subgroups $G \leq \SU(n)$ that admit $(n-1)$-dimensional isotropic orbits in $\C^n$.  Finally, in $\S$7.3, we mention a further example in which the compact Lie group $G$ is not semisimple.  

\subsection{Preliminaries on the Semisimple Case}

\indent \indent Let $G \leq \U(n)$ be a compact semisimple Lie group, so that $\mathfrak{z}(\mathfrak{g}) = 0$.  For $z \in \C^n \setminus 0$, let $\mathcal{O}_z \subset \C^n$ be the $G$-orbit of $z$, and let $H$ be the stabiliser of $z$.  By Lemma \ref{lem-isotropiclevelset}, the orbit $\mathcal{O}_z \subset \C^n$ is isotropic if and only if $z \in \mu^{-1}(0)$. \\
\indent As discussed in $\S$4.2, the $G$-action on $\C^n\setminus 0$ induces a $G$-action on $\CP^{n-1}$ via $g \cdot [z] := [gz]$, where we write $[z] = P_z \in \CP^{n-1}$ for the complex line through $z \in \C^n \setminus 0$.  We let
$$\mathcal{O}_{[z]} = \{[w] \in \CP^{n-1} \colon w \in \mathcal{O}_z\} \subset \CP^{n-1}$$
denote the $G$-orbit of $[z]$, and let $\widetilde{H}$ be the stabiliser of $[z]$.  Note that the $G$-equivariant map $q_z \colon \mathcal{O}_z \to \mathcal{O}_{[z]}$ via $w \mapsto [w]$ is a principal $(\widetilde{H}/H)$-bundle.  In particular, $\dim(\mathcal{O}_z) = \dim(\mathcal{O}_{[z]}) + \dim(\widetilde{H}/H)$. \\ 
\indent Let $\tau \colon \C^n\setminus 0 \to \CP^{n-1}$ be the usual $\U(n)$-equivariant surjective submersion $\tau(z) = [z]$, and equip $\CP^{n-1}$ with its standard (Fubini-Study) K\"{a}hler structure ($g_{\text{FS}}, J_{\text{FS}}, \omega_{\text{FS}})$.  With respect to the Fubini-Study structure, the $G$-action on $\CP^{n-1}$ is K\"{a}hler and Hamiltonian.  To see that the action is Hamiltonian, we recall a well-known formula for the moment map (see, e.g., \cite[pg. 12-14]{Kirwan1984}).  Let $\widetilde{\mu} \colon \C^n \setminus 0 \to \mathfrak{g}^*$ be
$$\widetilde{\mu}(z) = \frac{1}{|z|^2} \mu(z).$$
Then $\widetilde{\mu}$ is $\C^*$-invariant, so descends to a function $\nu \colon \CP^{n-1} \to \mathfrak{g}^*$, meaning $\widetilde{\mu} = \tau^*\nu$.  In the sequel, we use the same symbol $\rho$ to denote the infinitesimal $G$-action on both $\C^n$ and $\CP^{n-1}$.

\begin{lemma} The map $\nu \colon \CP^{n-1} \to \mathfrak{g}^*$ is a moment map for the $G$-action on $\CP^{n-1}$, i.e.\ :
\begin{align*}
    d\nu^X & = -\iota_{\rho(X)}\omega_{\mathrm{FS}}, \ \ \text{ for all } X \in \mathfrak{g}, \\
    \nu(g \cdot [z]) & = \Ad_g^*\!\left(\nu([z])\right), \ \ \text{ for all } g \in G,\, [z] \in \CP^{n-1}.
\end{align*}
\end{lemma}

\begin{proof} Before beginning, we set $\gamma := \partial(|z|^2) = \sum \overline{z}_j dz_j$ for convenience.  Noting that $\text{Re}(\gamma) = \frac{1}{2}(\gamma + \overline{\gamma}) = \frac{1}{2}d(|z|^2)$, we see that every $X \in \mathfrak{g}$ satisfies $\text{Re}(\gamma)(\rho(X)) = 0$, so that $\gamma(\rho(X)) = -\overline{\gamma}(\rho(X))$.  Recalling the Liouville $1$-form
$$\lambda = \frac{1}{2} \sum x_jdy_j - y_jdx_j = -\frac{i}{4}\sum \overline{z}_j dz_j - z_jd\overline{z}_j = -\frac{i}{4} (\gamma - \overline{\gamma}),$$
we use Lemma \ref{prop-cnmomentmap} to observe
$$\mu^X(z) = \lambda(\rho(X)) = -\frac{i}{4}\left[ \gamma(\rho(X)) - \overline{\gamma}(\rho(X)) \right] = -\frac{i}{2} \gamma(\rho(X)).$$
\indent Now, for the first claim of the lemma, we note that $\omega = \frac{i}{2}\partial\overline{\partial}(|z|^2) = \frac{i}{2}\partial\overline{\gamma}$ to compute
\begin{align*}
    \tau^*\omega_{\text{FS}} = \frac{i}{2}\partial\overline{\partial}(\log |z|^2) = \frac{i}{2}\partial\!\left( |z|^{-2} \overline{\gamma}\right) 
    & = -\frac{i}{2}|z|^{-4} \gamma \wedge \overline \gamma + |z|^{-2}\omega,
\end{align*}
and hence
\begin{align*}
    \tau^*(-\iota_{\rho(X)}\omega_{\text{FS}}) = -\iota_{\rho(X)}(\tau^*\omega_{\text{FS}}) & = \frac{i}{2}|z|^{-4} \iota_{\rho(X)}(\gamma \wedge \overline{\gamma})  - |z|^{-2}\omega(\rho(X), \cdot) \\
    & = \frac{i}{2}|z|^{-4} \gamma(\rho(X)) \left(  \gamma + \overline{\gamma} \right) - |z|^{-2} d\mu^X \\
    & = -|z|^{-4} \mu^X(z) \left( \gamma + \overline{\gamma} \right) - |z|^{-2} d\mu^X.
\end{align*}
On the other hand, we may also compute
\begin{align*}
\tau^*(d\nu^X) = d(\tau^*\nu^X) = d\widetilde{\mu}^X = d(|z|^{-2}\mu^X) & = -|z|^{-4}\mu^X(z)d(|z|^2) - |z|^{-2}d\mu^X
\end{align*}
from which we obtain
$$\tau^*(d\nu^X) = \tau^*(-\iota_{\rho(X)}\omega_{\text{FS}}).$$
Noting that $\tau^*$ is injective gives the first claim.  For the $G$-equivariance, we simply note that
$$\nu([gz]) = \widetilde \mu(gz) = \frac{1}{|gz|^2}\mu(gz) = \text{Ad}_g^*\!\left[ \frac{1}{|z|^2}\mu(z) \right] = \text{Ad}_g^*(\nu([z])).$$
\end{proof}

\begin{lemma}
Let $G \leq \U(n)$ be a compact semisimple Lie group, and let $z \in \C^n \setminus 0$. \\
\indent (a) The $G$-orbit $\mathcal{O}_z \subset \C^n$ is isotropic if and only if the $G$-orbit $\mathcal{O}_{[z]} \subset \CP^{n-1}$ is isotropic. \\
\indent (b) If $G \leq \SU(n)$ and $\mathcal{O}_z \subset \C^n$ is isotropic, then $\dim(\mathcal{O}_z) = \dim(\mathcal{O}_{[z]})$.
\end{lemma}

\begin{proof} (a) By definition of $\nu$, we have $\mu(z) = |z|^2\widetilde{\mu}(z) = |z|^2\nu([z])$.  Consequently,
$$\mathcal{O}_z \text{ is isotropic } \iff z \in \mu^{-1}(0) \iff [z] \in \nu^{-1}(0) \iff \mathcal{O}_{[z]} \text{ is isotropic.}$$
\indent (b) Suppose that $G \leq \SU(n)$.  If $\mathcal{O}_z \subset \C^n$ is isotropic, then $z \in \mu^{-1}(0)$.  Now, Proposition \ref{prop-profilesymmetry} implies that $\widetilde{H}/H \cong C_m$ is finite, so $q_z$ is a local diffeomorphism and $\dim(\mathcal{O}_{[z]}) = \dim(\mathcal{O}_z)$. 
\end{proof}

\begin{corollary}\label{cor-LagIsoCorresp} Let $G \leq \SU(n)$ be a compact semisimple Lie group, and let $z \in \C^n \setminus 0$.  The $G$-orbit $\mathcal{O}_{[z]} \subset \CP^{n-1}$ is Lagrangian if and only if the $G$-orbit $\mathcal{O}_z \subset \C^n$ is an isotropic submanifold of dimension $(n-1)$.  
\end{corollary}

\begin{remark}
In fact, if $G \leq \SU(n)$ is a compact semisimple Lie group, each Lagrangian $G$-orbit in $\CP^{n-1}$ is a minimal submanifold; see \cite[Prop. 3.1]{Bedulli2008}. Consequently, the corresponding isotropic $G$-orbit in the unit sphere $\Sph^{2n-1}(1) \subset \C^n$ is a special Legendrian submanifold, and hence is the link of a special Lagrangian cone in $\C^n$, which accords with our discussion in $\S$5.1.
\end{remark}

\subsection{The Classification for $G$ Simple}

\indent \indent We now turn to Bedulli and Gori's classification of compact simple groups in $\U(n)$ that admit a Lagrangian orbit in $\CP^{n-1}$.

\begin{theorem}[\cite{Bedulli2008}] Let $G$ be a compact simple Lie group.  Suppose $G$ acts on $\CP^{n-1}$ via a unitary representation $\sigma \colon G \to \U(n)$.  Then there exists a Lagrangian $G$-orbit in $\CP^{n-1}$ if and only if $G$ appears in the table of Figure \ref{fig-bedulligori}.
\end{theorem}

\begin{remark}\label{rem-guncompact} If $G \leq \U(n)$ is a compact simple Lie group, then $G \leq \SU(n)$.  To see this, let $K$ be the kernel of the determinant map $\det \colon G \to \U(1)$, and let $K^0$ be its identity component.  Since $K^0$ is a connected normal subgroup of $G$, we have $K^0 = \{\mathrm{Id}\}$ or $K^0 = G$ by the simplicity of $G$.  In the former case, $K$ is finite, so $\dim(G) = \dim(G/K) = \dim(\det(G)) \leq 1$, contradicting the simplicity of $G$.  Thus, $K^0 = G$, so $K = G$, so $G \leq \SU(n)$.
\end{remark}

The upshot is that, by Corollary \ref{cor-LagIsoCorresp}, the table of Figure \ref{fig-bedulligori} classifies the compact simple Lie subgroups $G \leq \SU(n)$ that admit $(n-1)$-dimensional isotropic orbits in $\C^n$.  For each such subgroup $G$, the results of the preceding sections yield information about $G$-equivariant Lagrangian mean curvature flow.  For illustration, here is an explicit example.


\begin{example} Let $G = \SU(2)$ act on $\C^4 = \Sym^3(\C^2)$ in the standard way.  That is, for $g \in \SU(2)$ and a homogeneous cubic polynomial $p \in \Sym^3(\C^2)$ in two variables $w = (w_1, w_2) \in \C^2$, we define
$$(g \cdot p)(w) := p\!\left(g^Tw \right) = p\!\left(g^T\begin{bmatrix} w_1 \\ w_2 \end{bmatrix} \right).$$
More explicitly, with respect to the basis $\{w_1^3, \sqrt{3}w_1^2w_2,\sqrt{3}w_1w_2^2, w_2^3\}$ of $\Sym^3(\C^2)$, the $\SU(2)$-action on $\C^4$ is given by
$$\begin{pmatrix} a & -\overline{b} \\ b & \overline{a} \end{pmatrix} \cdot \begin{bmatrix} z_1 \\ z_2 \\ z_3 \\ z_4 \end{bmatrix} := \begin{pmatrix}
a^3 & -\sqrt{3}a^2\overline{b} & \sqrt{3}a\overline{b}^2 & -\overline{b}^3 \\
\sqrt{3}a^2b & a(|a|^2 - 2|b|^2) & -\overline{b}(2|a|^2 - |b|^2) & \sqrt{3}\overline{a}\overline{b}^2 \\
\sqrt{3}ab^2 & b(2|a|^2 - |b|^2) & \overline{a}(|a|^2 - 2|b|^2) & -\sqrt{3}\overline{a}^2\overline{b} \\
b^3 & \sqrt{3}\overline{a}b^2 & \sqrt{3}\overline{a}^2b & \overline{a}^3
\end{pmatrix}
\begin{bmatrix} z_1 \\ z_2 \\ z_3 \\ z_4 \end{bmatrix}$$
 where $a,b \in \C$ satisfy $|a|^2 + |b|^2 = 1$.  A calculation shows that the moment map of the $\SU(2)$-action is the function $\mu \colon \C^4 \to \mathfrak{su}(2) \subset \mathfrak{su}(4)$ given by
 $$\mu(z) = \frac{i}{40}\begin{bmatrix}
  3p(z) & \sqrt{3}r(z) & 0 & 0 \\
  \sqrt{3}\,\overline{r(z)} & p(z) & 2r(z) & 0 \\
  0 & 2\overline{r(z)} & -p(z) & \sqrt{3}r(z) \\
  0 & 0 & \sqrt{3}\,\overline{r(z)} & -3p(z)
 \end{bmatrix},$$
 where
 \begin{align*}
     p(z) & := 3|z_1|^2 + |z_2|^2 - |z_3|^2 - 3|z_4|^2 & r(z) & := 2\sqrt{3}(z_1\overline{z}_2 + z_3\overline{z}_4) + 4z_2\overline{z}_3.
 \end{align*}
By Proposition \ref{prop:constrain}, every $\SU(2)$-invariant Lagrangian submanifold $L_0 \subset \C^4$ lies in the level set
$$\mu^{-1}(0) = \left\{ z \in \C^4 \colon p(z) = 0,\, r(z) = 0\right\}\!.$$

While the principal $\SU(2)$-orbits in $\C^4$, such as that of $(0,1,i,0)$, have trivial stabiliser, the $\SU(2)$-action also admits both exceptional and singular orbits. For example, the point $z = (1,0,0,1) \in \mu^{-1}(0)$ has stabiliser
\begin{align*}
    H = \text{Stab}(z) &= C_3 = \left\{\begin{pmatrix} \zeta & 0 \\ 0 & \overline{\zeta}\end{pmatrix} \colon \zeta^3 = 1\right\},\\
    \widetilde H = \text{Stab}([z]) &= C_3 \rtimes C_4 = \left\{ \begin{pmatrix} \zeta & 0 \\ 0 & \overline\zeta \end{pmatrix}\begin{pmatrix} 0 & i \\ i & 0 \end{pmatrix}^k \,:\, \zeta^3 = 1, \,k \in \{0,1,2,3\}\right\} 
\end{align*}
and the $\SU(2)$-orbit is an isotropic submanifold of $\mathbb{C}^4$ diffeomorphic to the Lens space $\mathbb{S}^3/C_3$. It follows that $M := \SU(2) \cdot (P_z \setminus 0) \subset \mu^{-1}(0)$ is a connected component of $M_0 := \mu^{-1}(0) \cap  \mathbb{C}^4_{(C_3)}$, and in fact further calculation yields that $M = M_0 = \mu^{-1}(0)\setminus 0$ (see \cite[Lem. 3.14]{Marshall1999}). Therefore, any $\SU(2)$-invariant Lagrangian lies in $M$. 

Now, let $L_0 \subset M$ be a connected, immersed $\SU(2)$-invariant Lagrangian of type $(C_3)$. If $L_0$ is almost-calibrated, then $L_0$ is exact, embedded, and homeomorphic to $\R \times (\Sph^3/C_3)$ by Propositions \ref{prop-embedded} and \ref{prop-exact}, and any mean curvature flow starting at $L_0 \subset M$ must stay within $M$ by Proposition \ref{prop-flowlevelset}.

\indent The $\SU(2)$-invariant special Lagrangians of type $(C_3)$ are given in Theorems \ref{thm-splagcones} and \ref{thm-splags}, and were discussed by Marshall in \cite{Marshall1999}.  Theorem \ref{thm-shrinkers} classifies the cohomogeneity-one $\SU(2)$-invariant Lagrangian shrinkers and expanders in $\C^4$, while Theorem \ref{thm-translators} rules out the existence of cohomogeneity-one $\SU(2)$-invariant translators.

Let $L_t \subset M \subset \C^4$, $t \in [0,T)$, be a connected $\SU(2)$-equivariant almost-calibrated Lagrangian mean curvature flow.  By Theorem \ref{thm-locationofsingularities}, a singularity occurs at time $T$ if and only if $(O,T)$ is the unique singular space-time point for $L_t$. Suppose that $T$ is, in fact, the singular time.  Then there exist $\theta \in \R$ and $k \in \{1,2\}$ such that, by Theorems \ref{thm-typeiblowup} and \ref{thm-typeiiblowup}: 
\begin{itemize}
    \item Any Type I blowup at time $T$ is Marshall's $\SU(2)$-invariant special Lagrangian cone --- namely, the special Lagrangian with profile curve $C_4 \cdot (\widetilde{c}_{k-1, \overline{\theta}} \cup \widetilde{c}_{k,\overline{\theta}})$.
    \item There is a unique $B > 0$ such that any Type II blowup at time $T$ is a translation of the connected $\SU(2)$-invariant special Lagrangian with profile curve $C_4 \cdot \widetilde l_{B,k,\overline \theta}$. In particular, the asymptotes of the Type II blowup are given by the unique Type I blowup of Theorem \ref{thm-typeiblowup}.
\end{itemize}
Finally, Theorem \ref{thm-existenceofsingularities} shows that there exists a mean curvature flow with such a finite-time singularity, and therefore all $\SU(2)$-invariant smooth special Lagrangians occur as Type II blowups of Lagrangian mean curvature flows.
\end{example}

\begin{figure}[h]
$$\begin{tabular}{| c | c | c | c | c |} \hline
$G$ & $\sigma$ & $n$ & $\widetilde{H}^0$ & $\widetilde{H}/\widetilde{H}^0$
 \\ \hline \hline
 $\SU(p)$ & $2\Lambda_1$ & $\frac{1}{2}p(p+1)$ & $\SO(p)$ & $C_p$ \\ \hline
 $\SU(p)$ & $\Lambda_1 \oplus \Lambda_1^*$ & $2p$ & $\SU(p-1)$ & $C_2$ \\ \hline
 $\SU(p)$ & $\Lambda_1 \oplus \cdots \oplus \Lambda_1$ & $p^2$ & $1$ & $C_p$ \\ \hline
 $\SU(2p)$ & $\Lambda_2$ & $p(2p-1), \ p \geq 3$ & $\Sp(p)$ & $C_{2p}$ \\ \hline
 $\SU(2p+1)$ & $\Lambda_2 \oplus \Lambda_1$ & $2p^2 + 3p + 2, \  p \geq 2$ & $\Sp(p)$ & $C_{p+1}$ \\ \hline
$\SU(2)$ & $3\Lambda_1$ & $4$ & $1$ & $C_3 \rtimes C_4$ \\ \hline
$\SU(6)$ & $\Lambda_3$ & $20$ & $\SU(3) \times \SU(3)$ & $C_4$ \\ \hline
$\SU(7)$ & $\Lambda_3$ & $35$ & $\text{G}_2$ & $C_7$ \\ \hline
$\SU(8)$ & $\Lambda_3$ & $56$ & $\text{Ad}(\SU(3))$ & $C_{16}$ \\ \hline \hline
$\Sp(p)$ & $\Lambda_1 \oplus \Lambda_1$ & $4p$ & $\Sp(p-1)$ & $C_2$ \\ \hline
$\Sp(3)$ & $\Lambda_3$ & $14$ & $\SU(3)$ & $C_4$ \\ \hline  \hline
$\SO(p)$ & $\Lambda_1$ & $p, \ p \geq 3$ & $\SO(p-1)$ & $C_2$ \\ \hline
$\Spin(7)$ & spin rep. & $8$ & $\text{G}_2$ & $C_2$ \\ \hline
$\Spin(9)$ & spin rep. & $16$ & $\Spin(7)$ & $C_2$ \\ \hline
$\Spin(10)$ & $\Lambda_{\text{even}} \oplus \Lambda_{\text{even}}$ & $32$ & $\text{G}_2$ & -- \\ \hline
$\Spin(11)$ & spin rep. & $32$ & $\SU(5)$ & $C_4$ \\ \hline
$\Spin(12)$ & $\Lambda_{\text{even}}$ & $32$ & $\SU(6)$ & $C_4$ \\ \hline
$\Spin(14)$ & $\Lambda_{\text{even}}$ & $64$ & $\text{G}_2 \times \text{G}_2$ & $C_8$ \\ \hline \hline
$\text{E}_6$ & $\Lambda_1$ & $27$ & $\text{F}_4$ & $C_3$ \\ \hline
$\text{E}_7$ & $\Lambda_1$ & $56$ & $\text{E}_6$ & -- \\ \hline
$\text{G}_2$ & $\Lambda_2$ & $7$ & $\SU(3)$ & $C_2$ \\ \hline
\end{tabular}$$
\caption{Bedulli-Gori's classification of compact simple Lie group actions on $\mathbb{CP}^{n-1}$ with a Lagrangian orbit. In the table, we identify a representation $\sigma \colon G \to \U(n)$ with the highest weights of its irreducible components.  Here, $\Lambda_1, \Lambda_2, \ldots$ are the fundamental dominant weights of $G$, $\Lambda_{\text{even}}$ denotes the even half-spin representation of $\Spin(2p)$, and $\widetilde{H}^0$ is the identity component of $\widetilde{H}$.}
\label{fig-bedulligori}
\end{figure}

\subsection{A Non-Semisimple Example}

\indent \indent To conclude, we mention a further example involving a non-semisimple group, discussed in \cite[pg. 4]{Bryant2004}.

\begin{example} For $p,q \geq 3$, consider the action of $G = \Sph^1 \times \SO(p) \times \SO(q)$ on $\C^{p+q}$ via
$$(e^{i\theta}, A, B) \cdot (z,w) := (e^{qi\theta}Az, e^{-pi\theta}Bw).$$
For $p \neq q$, there are at least $9$ orbit types.  In the following table we use the shorthand ``$\{x,y\}$ (in)dependent" to mean that the set of vectors $\{x,y\}$ is an $\R$-linearly (in)dependent set. 
$$\begin{tabular}{| c | c | c |} \hline
$z = x+iy$ & $w = u+iv$ & dim of $G$-orbit $\mathcal{O}_{(z,w)}$ \\ \hline \hline
$\{x,y\}$ independent & $\{u,v\}$ independent & $2p + 2q - 5$ \\ \hline
$\{x,y\}$ independent & $\{u,v\}$ dependent, $w \neq 0$ & $2p + q - 3$ \\ \hline
$\{x,y\}$ independent & $0$ & $2p - 2$ \\ \hline
$\{x,y\}$ dependent, $z \neq 0$ & $\{u,v\}$ independent & $p + 2q - 3$ \\ \hline
$\{x,y\}$ dependent, $z \neq 0$ & $\{u,v\}$ dependent, $w \neq 0$ & $p+q-1$ \\ \hline
$\{x,y\}$ dependent, $z \neq 0$ & $0$ & $p$ \\ \hline
$0$ & $\{u,v\}$ independent & $2q - 2$ \\ \hline
$0$ & $\{u,v\}$ dependent, $w \neq 0$ & $q$ \\ \hline
$0$ & $0$ & $0$ \\ \hline
\end{tabular}$$
In particular, we note that the points $(z,w) = (x+iy, u+iv)$ having $z \neq 0$, $w \neq 0$, and $\{x,y\}$ and $\{u,v\}$ both $\R$-linearly dependent have $G$-orbits of real dimension $p+q-1$.  Moreover, one can check that these orbits are isotropic submanifolds of $\C^{p+q}$. \\
\indent The moment map is the function
\begin{align*}
 \mu \colon \C^{p+q} & \to \mathfrak{u}(1) \oplus \mathfrak{so}(p) \oplus \mathfrak{so}(q) \\
    \mu(z,w) & = \left( i(q|z|^2 - p |w|^2),\, \text{Im}(z\overline{z}^T),\, \text{Im}(w\overline{w}^T) \right)\!.
\end{align*}
Note that $\mathfrak{z}(\mathfrak{u}(1) \oplus \mathfrak{so}(p) \oplus \mathfrak{so}(q)) = \{(ic, 0,0) \colon c \in \R\}$.  The $\mu$-level sets at the central values $(ic,0,0)$ are
\begin{align*}
    \mu^{-1}(ic,0,0) & = \left\{ \left(\zeta \mathbf{r}, \eta \mathbf{s} \right) \in \C^p \oplus \C^q \colon \mathbf{r} \in \Sph^{p-1}, \mathbf{s} \in \Sph^{q-1} \text{ and } \zeta, \eta \in \C \text{ s.t. } q|\zeta|^2 - p|\eta|^2 = c \right\}\!.
\end{align*}
By Proposition \ref{prop:constrain}, every $G$-invariant Lagrangian submanifold $L_0 \subset \C^{p+q}$ lies in one of the level sets $\mu^{-1}(ic,0,0)$.  Further, by Proposition \ref{prop-flowlevelset}, any mean curvature flow starting at $L_0$ must stay within its initial level set $\mu^{-1}(ic,0,0)$. In the case of $\mu^{-1}(0,0,0)$, the results of $\S$6 follow as in the previous examples. \end{example}

\bibliographystyle{abbrv}
\bibliography{main}

\end{document}